\documentclass[reqno]{amsart}

\usepackage{amsmath,amsthm,amsfonts,amssymb,young}
\usepackage{array, boldline, makecell, booktabs}
\usepackage{bm}
\usepackage{ytableau}
\usepackage{euscript, mathrsfs} 
\usepackage{tikz}
\usetikzlibrary{backgrounds}
\usepackage{amsmath}
\usepackage{mathtools}
\usepackage[colorlinks]{hyperref}
\usepackage{cleveref}
\usepackage[margin=1in]{geometry}
\numberwithin{equation}{section}    
\usepackage[all]{xy}
\usepackage{graphicx,import}
\usepackage{amscd}
\usepackage{color}
\usepackage{enumerate}
\usepackage{fourier}
\usepackage{tikz}
\usepackage{cases}
\usepackage[style=alphabetic,maxnames=99,maxalphanames=6]{biblatex}
\usepackage[T1]{fontenc}
\usepackage[utf8]{inputenc}

\newtheorem{thm}{Theorem}[section]
\newtheorem{lem}[thm]{Lemma}
\newtheorem{proposition}[thm]{Proposition}
\newtheorem{corollary}[thm]{Corollary}
\theoremstyle{definition}
\newtheorem{example}[thm]{Example}
\newtheorem{definition}[thm]{Definition}

\newtheorem{rmk}[thm]{Remark}

\DeclareMathOperator{\I}{I}
\DeclareMathOperator{\Dyck}{Dyck}
\DeclareMathOperator{\type}{type}
\DeclareMathOperator{\Krew}{Krew}
\DeclareMathOperator{\area}{area}

\DeclareMathOperator{\MLD}{MLD}
\DeclareMathOperator{\bounce}{bounce}
\DeclareMathOperator{\grFrob}{grFrob}
\DeclareMathOperator{\csum}{rsum}
\DeclareMathOperator{\rsum}{csum}
\DeclareMathOperator{\Merge}{Merge}
\DeclareMathOperator{\Split}{Split}

\DeclareMathOperator{\Inv}{Inv}

\DeclareMathOperator{\coarea}{coarea}

\DeclareMathOperator{\pinv}{pinv}

\DeclareMathOperator{\Park}{Park}

\DeclareMathOperator{\iDes}{iDes}
\DeclareMathOperator{\Des}{Des}

\DeclareMathOperator{\bw}{bw}
\DeclareMathOperator{\BW}{BW}

\DeclareMathOperator{\maj}{maj}

\DeclareMathOperator{\arm}{arm}
\DeclareMathOperator{\leg}{leg}
\DeclareMathOperator{\coleg}{coleg}
\DeclareMathOperator{\coarm}{coarm}
\DeclareMathOperator{\inv}{inv}
\DeclareMathOperator{\dinv}{dinv}

\DeclareMathOperator{\st}{st}
\DeclareMathOperator{\LLB}{LB}
\DeclareMathOperator{\LB}{lb}
\DeclareMathOperator{\wt}{wt}

\DeclareMathOperator{\Set}{Set}
\DeclareMathOperator{\Comp}{Comp}
\DeclareMathOperator{\stat}{stat}
\DeclareMathOperator{\cycling}{Cyc}
\DeclareMathOperator{\DR}{DR}
\DeclareMathOperator{\CC}{\mathbb{C}}
\DeclareMathOperator{\rev}{rev}

\DeclareMathOperator{\dimm}{dim}
\DeclareMathOperator{\boldOP}{\mathbf{OP}}
\DeclareMathOperator{\Mat}{Mat}
\DeclareMathOperator{\Matoo}{Mat^{\geq0}}
\DeclareMathOperator{\rred}{red}

\title{Shuffle formula in science fiction for Macdonald polynomials}

\author{Donghyun Kim}
\address{Department of Mathematical Sciences \\ Seoul National University \\Seoul 08826 \\ Republic of Korea}
\email{hyun920310@snu.ac.kr}

\author{Seung Jin Lee}
\address{Department of Mathematical Sciences \\ Research institute of Mathematics \\ Seoul National University \\ Seoul 08826 \\ Republic of Korea}
\email{lsjin@snu.ac.kr}

\author{Jaeseong Oh}
\address{June E Huh Center for Mathematical Challenges \\ Korea Institute for Advanced Study \\ Seoul 08826 \\ Republic of Korea}
\email{jsoh@kias.re.kr}

\begin{document}

\maketitle

\begin{abstract}
We initiate the study of the Macdonald intersection polynomials $\operatorname{I}_{\mu^{(1)},\dots,\mu^{(k)}}[X;q,t]$, which are indexed by $k$-tuples of partitions $\mu^{(1)},\dots,\mu^{(k)}$. These polynomials are conjectured to be equal to the bigraded Frobenius characteristic of the intersection of Garsia-Haiman modules, as proposed by the science fiction conjecture of Bergeron and Garsia. In this work, we establish the vanishing identity and the shape independence of the Macdonald intersection polynomials. Additionally, we unveil a remarkable connection between $\operatorname{I}_{\mu^{(1)},\dots,\mu^{(k)}}$ and the character $\nabla e_{k-1}$ of diagonal coinvariant algebra by employing the plethystic formula for the Macdonald polynomials of Garsia--Haiman--Tesler. Furthermore, we establish a connection between $\operatorname{I}_{\mu^{(1)},\dots,\mu^{(k)}}$ and the shuffle formula $D_{k-1}[X;q,t]$, utilizing novel combinatorial tools such as the column exchange rule and the lightning bolt formula for Macdonald intersection polynomials. Notably, our findings provide a new proof for the shuffle theorem.
\end{abstract}

\textit{2020 Mathematics Subject Classification [MSC] codes:} 05E05, 05E10.
\section{Introduction}\label{Sec: Introduction}
Inspired by the science fiction conjecture proposed by Bergeron and Garsia \cite{BG99}, we examine \emph{Macdonald intersection polynomials}, which are conjectured to be equal to the bigraded Frobenius characteristic of the intersection of Garsia-Haiman modules. To provide a comprehensive overview of the primary results, we discuss the historical context and background of Macdonald polynomials, the shuffle theorem, and the science fiction conjecture.

\subsection{Macdonald polynomials}
In his seminal paper \cite{Mac88}, Macdonald introduced the \emph{Macdonald $P$-polynomials} $P_\mu[X;q,t]$, which can be seen as a $q,t$ generalization of Schur functions. Macdonald polynomials specialize to important families of symmetric functions such as Jack symmetric functions, Hall-Littlewood polynomials. They have been widely researched and have numerous applications in representation theory, algebraic geometry, and mathematical physics, among others. Garsia and Haiman introduced a combinatorial variant of Macdonald polynomials called the \emph{modified Macdonald polynomial}, which is defined by the plethystic substitution of the \emph{Macdonald $J$-polynomial} of \cite{Mac88}:
\[
    \widetilde{H}_\mu[X;q,t] = t^{n(\mu)}J_\mu\left[\dfrac{X}{1-t^{-1}};q,t^{-1}\right].
\]
One of the most fascinating questions surrounding Macdonald polynomials was the \emph{Macdonald positivity conjecture}, which deals with the Schur positivity of the modified Macdonald polynomials.

In 1993, Garsia and Haiman introduced the \emph{Garsia-Haiman module} $V_\mu$ associated to a partition $\mu\vdash n$ as the subspace in $\CC[\mathbf{x}_n,\mathbf{y}_n]\coloneqq\mathbb{C}[x_1,\dots,x_n,y_1,\dots,y_n]$ spanned by partial derivatives of the `generalized' Vandermonde determinant $\Delta_\mu$ associated to $\mu$ \cite{GH93, GH96module}. They expected that the Garsia-Haiman modules serve as a representation-theoretic model for the modified Macdonald polynomials. More specifically, they conjectured that
\[
    \grFrob(V_\mu;q,t) = \widetilde{H}_\mu[X;q,t],
\]
where $\grFrob$ denotes the bigraded Frobenius characteristic. This conjecture is known as the \emph{$n!$--conjecture}, which in particular implies that $\dimm V_\mu = n!$. Haiman \cite{Hai01} ultimately proved the $n!$--conjecture by utilizing the geometry of the Hilbert scheme of $n$ points in a plane, thereby confirming the Macdonald positivity conjecture. However, his proof does not offer a combinatorial formula for Schur coefficients, the \emph{$(q,t)$--Kostka polynomials} $\widetilde{K}_{\lambda,\mu}(q,t)$ which remains a significant open problem.

The celebrated result of Haglund, Haiman, and Loehr made important progress in giving a combinatorial formula for the modified Macdonald polynomials. For a partition $\mu$ of $n$, the Haglund--Haiman--Loehr formula \cite{HHL05} gives
\begin{equation}\label{Eq: HHL formula}
    \widetilde{H}_\mu[X;q,t] = \sum_{w \in \mathfrak{S}_n} q^{\inv_\mu(w)} t^{\maj_\mu(w)} F_{\iDes(w)}.
\end{equation}
Here, the sum is over all permutations of $[n]$, the functions $\inv_\mu$ and $\maj_\mu$ are certain statistics associated to a partition $\mu$, and $F$ is the (Gessel's) fundamental quasisymmetric function.

\subsection{The nabla operator and shuffle theorem}

The \emph{nabla operator} $\nabla$, introduced by Bergeron and Garsia, is an operator acting on symmetric functions with coefficients in $\mathbb{Q}(q,t)$. It is an eigen-operator for the {modified Macdonald polynomials}:
\[
    \nabla \widetilde{H}_\mu[X;q,t] = T_\mu \widetilde{H}_\mu[X;q,t]
\]
where $T_\mu\coloneqq\prod_{(i,j)\in\mu}t^{i-1}q^{j-1}$. Remarkably, the nabla operator is closely related to the \emph{diagonal coinvariant algebra} $\DR_n$ which is defined as 
\[
    \DR_n\coloneqq\CC[\mathbf{x}_n,\mathbf{y}_n]/\langle\CC[\mathbf{x}_n,\mathbf{y}_n]^{\mathfrak{S}_n}_+\rangle,
\]
where the symmetric group $\mathfrak{S}_n$ acts on the polynomial ring $\CC[\mathbf{x}_n,\mathbf{y}_n]$ by the diagonal action. Haiman \cite{Hai02} proved the \emph{$(n+1)^{n-1}$--conjecture} which asserts that $\dimm \DR_n = (n+1)^{n-1}$ and provided a formula for the bigraded Frobenius character of $\DR_n$,
\[
    \grFrob(\DR_n) = \nabla e_n,
\]
where $e_n$ is the $n$-th elementary symmetric function. 

The \emph{shuffle conjecture}, proposed by Haglund, Haiman, Loehr, Remmel, and Ulyanov \cite{HHLRU05} anticipates a combinatorial formula
\[
    D_n[X;q,t]\coloneqq\sum_{(\pi,w) \in \Park_n}t^{\area(\pi)}q^{\dinv(\pi,w)} F_{\iDes(w)}
\]
for $\nabla e_n$. Here, the sum is over all \emph{parking functions} $(\pi,w)$ of length $n$, and the functions $\area$ and $\dinv$ are statistics associated to each parking function. We refer to the generating function $D_n[X;q,t]$ as the \emph{shuffle formula}. Carlsson and Mellit \cite{CM18} eventually proved the shuffle conjecture by introducing the Dyck path algebra, which is closely related to the double affine Hecke algebra. Since then, numerous captivating aspects of the nabla operator and its generalizations, including their combinatorial formulas, have been extensively investigated \cite{Hai02, Mel18, BHMPS21LW, CM21, BHMPS21Delta, DM22}.

\subsection{Science fiction conjecture and Macdonald intersection polynomials}

Prior to Haiman's proof of the Macdonald positivity conjecture, Bergeron and Garsia proposed a conjectural combinatorial approach for constructing a basis for $V_\mu$ in \cite{BG99}, known as the \emph{science fiction conjecture}. Specifically, they presented series of conjectures concerning the intersection of Garsia-Haiman modules. To elaborate, for a partition $\mu\vdash n+1$, let $\mu^{(1)},\dots,\mu^{(k)}\subseteq \mu$ be $k$ distinct partitions such that
\[
    |\mu/\mu^{(1)}|=\cdots=|\mu/\mu^{(k)}|=1.
\] 
Let $T^{(i)}\coloneqq T_{\mu^{(i)}}$ for $i=1,\dots,k$. The science fiction conjecture suggests that
\begin{enumerate}
    \item the bigraded Frobenius characteristic of the intersection of the modules $\bigcap_{i=1}^{k} V_{\mu^{(i)}}$ is given by
    \[
        \I_{\mu^{(1)},\dots,\mu^{(k)}}[X;q,t]\coloneqq\sum_{i=1}^{k} \left(\prod_{j\neq i}\dfrac{T^{(j)}}{T^{(j)}-T^{(i)}}\right)\widetilde{H}_{\mu^{(i)}}[X;q,t], \qquad \text{and}
    \] 
    \item the dimension of $\bigcap_{i=1}^{k} V_{\mu^{(i)}}$ is $n!/k$.
\end{enumerate}
We refer to the symmetric function $\I_{\mu^{(1)},\dots,\mu^{(k)}}[X;q,t]$ as the \emph{Macdonald intersection polynomial}, and the second assertion as the $n!/k$--conjecture. In an earlier paper \cite{KLO22}, the authors examined Macdonald intersection polynomials for two partitions (when $k=2$), which was considered in Butler's conjecture \cite{But94}. More specifically, the authors provided a (positive) combinatorial formula for $\I_{\mu^{(1)},\mu^{(2)}}[X;q,t]$ that is analogous to the Haglund--Haiman--Loehr formula \eqref{Eq: HHL formula}. In this paper, we delve deeper into the properties of Macdonald intersection polynomials $\I_{\mu^{(1)},\dots,\mu^{(k)}}[X;q,t]$.

\subsection{Main results} We state the first main result of this paper. Direct implication of Theorem \ref{thm: main theorem} (b) and (c) is a new proof of the shuffle theorem.

\begin{thm}\label{thm: main theorem}
Given a partition $\mu \vdash n+1$, let $\mu^{(1)},\dots,\mu^{(k)}\subseteq \mu$ be $k$ distinct partitions such that
\[
    |\mu/\mu^{(1)}|=\cdots=|\mu/\mu^{(k)}|=1.
\]
The Macdonald intersection polynomial $\I_{\mu^{(1)},\dots,\mu^{(k)}}[X;q,t]$ satisfies the following.
\begin{enumerate}[(a)]
    \item For $m<k-1$, we have
    \begin{equation*}
        e_{n-m}^\perp\I_{\mu^{(1)},\dots,\mu^{(k)}}[X;q,t]=0.
    \end{equation*}
    \item We have
    \begin{equation*}
        \dfrac{1}{T_{\bigcap_{i=1}^k \mu^{(i)}}}e_{n+1-k}^\perp\left(\I_{\mu^{(1)},\dots,\mu^{(k)}}[X;q,t]\right) = \nabla e_{k-1}.
    \end{equation*}
    In particular, $e_{n+1-k}^\perp\I_{\mu^{(1)},\dots,\mu^{(k)}}[X;q,t]$  does not depend on the partitions $\mu^{(1)}, \dots, \mu^{(k)}$, and $\mu$ (up to a constant).
    \item We have
    \begin{equation*} 
        \dfrac{1}{T_{\bigcap_{i=1}^k \mu^{(i)}}}e_{n+1-k}^\perp\left(\I_{\mu^{(1)},\dots,\mu^{(k)}}[X;q,t]\right) = D_{k-1}[X;q,t].
    \end{equation*}
\end{enumerate}
Here, $e_n^\perp$ denotes the adjoint operator to the multiplication by $e_n$ with respect to the Hall inner product. 
\end{thm}

\begin{corollary}[Shuffle theorem \cite{CM18}]
We have $\nabla e_n=D_{n}[X;q,t]$.
\end{corollary}

We refer to the equation in Theorem~\ref{thm: main theorem} (a) as the \emph{vanishing identity}. The second and the third points give miraculous relationships between the Macdonald intersection polynomial $\I_{\mu^{(1)},\dots,\mu^{(k)}}[X;q,t]$, the character $\nabla e_{k-1}$ of the diagonal coinvariant algebra, and the shuffle formula $D_{k-1}[X;q,t]$. On the one hand, this suggests that there might be a representation-theoretic background to directly connect the intersection of Garsia--Haiman modules and the diagonal coinvariant algebra. On the other hand, our findings imply that finding an explicit combinatorial formula for the Macdonald intersection polynomials is very challenging.

We also study the specialization of Macdonald intersection polynomials at $q=t=1$. In \cite{BG99}, Bergeron and Garsia showed that the `dimension' of the Macdonald intersection polynomial is $n!/k$:
\begin{equation}\label{eq: mac int dimension}
     \langle \I_{\mu^{(1)},\dots,\mu^{(k)}}[X;1,1],e_{(1^n)}\rangle=\frac{n!}{k}
\end{equation}
which is compatible with the $n!/k$--conjecture. In this paper, we give the first explicit description of the Macdonald intersection polynomial at $q=t=1$.

Recall that the modified Macdonald polynomial specialized at $q=t=1$ gives
\[
    \widetilde{H}_\mu[X;1,1]=h_{(1^n)}
\]
where $n$ is the size of a partition $\mu$. As an analogue of the above equation, in the author's previous paper, it was shown that for a partition $\mu\vdash n+1$, and distinct partitions $\mu^{(1)},\mu^{(2)}\subseteq\mu$ such that $|\mu/\mu^{(1)}|=|\mu/\mu^{(2)}|=1$, we have
\[
    \I_{\mu^{(1)},\mu^{(2)}}[X;1,1] = h_{(2,1^{n-2})}.
\]
Our result generalizes these identities to arbitrary $k$.
\begin{thm}\label{thm: h-expansion, Kreweras} For a partition $\mu\vdash n+1$, let $\mu^{(1)},\dots,\mu^{(k)}\subseteq \mu$ be $k$ distinct partitions such that \[|\mu/\mu^{(1)}|=\cdots=|\mu/\mu^{(k)}|=1.\] Then we have
\[
    \I_{\mu^{(1)},\dots,\mu^{(k)}}[X;1,1]=\sum_{\lambda\vdash k-1}(-1)^{k-1-\ell(\lambda)}\Krew(\lambda+(1^{\ell(\lambda)}))h_{\lambda+(1^{n+1-k})},
\]
where $\Krew(\lambda)$ is the Kreweras number indexed by a partition $\lambda$. See Table~\ref{tab: Kreweras theorem} for the example. The definition of the Kreweras number is given in Section~\ref{Sec: Macdonald intersection polynomials at q=t=1}.
\end{thm}
We also recover \eqref{eq: mac int dimension} from Theorem \ref{thm: h-expansion, Kreweras} (Corollary \ref{cor: n!/k}).
\begin{table}[h]
    \centering
        \begin{tabular}{|c|c|}
        \hline
        & $\I_{\mu^{(1)},\dots,\mu^{(k)}}[X;1,1]$ \\\hline
        $k=2$ & $h_{21\cdots1}$ \\\hline
        $k=3$ & $2h_{221\cdots1}-h_{31\cdots1}$ \\\hline
        $k=4$ & $5h_{2221\cdots1}-5_{321\cdots1}+h_{41\cdots1}$ \\\hline
        $k=5$ & $14h_{22221\cdots1}-21h_{3221\cdots1}+3h_{331\cdots1}+6h_{421\cdots1}-h_{51\cdots1}$ \\\hline
        $k=6$ &
        $42h_{222221\cdots1} - 84h_{32221\cdots1} + 28h_{3321\cdots1} + 28h_{4221\cdots1} - 7h_{431\cdots1} - 7h_{521\cdots1} + h_{61\cdots1}$\\\hline
        \end{tabular}
    \caption{The $h$-expansions of $\I_{\mu^{(1)},\dots,\mu^{(k)}}[X;1,1]$ up to $k=6$.}
    \label{tab: Kreweras theorem}
\end{table}
\subsection{Organization}
This paper is structured as follows. We commence with Section~\ref{Sec: Preliminaries}, laying out the essential preliminary concepts and background information. In Section~\ref{Sec: Vanishing identity, Macdonald intersection polynomials and nabla}, we review the plethystic formula for the Macdonald polynomials, studied by Garsia, Haiman, and Tesler \cite{GHT99}. By utilizing their formula, we proceed to establish Theorem~\ref{thm: main theorem} (a) (vanishing identity). Furthermore, we demonstrate Theorem~\ref{thm: main theorem} (b), which proves shape independence and establishes a connection between Macdonald intersection polynomials and the characters of the diagonal coinvariant algebras. Sections~\ref{Sec: New fermionic formula}, ~\ref{Sec: generalized Macdonald polynomials} and ~\ref{Sec: proof of main theorem} are dedicated to proving Theorem \ref{thm: main theorem} (c). Section~\ref{Sec: New fermionic formula} provides some background on the shuffle formula $D_n[X;q,t]$. Then we present a fermionic formula, which we call a lightning bolt formula (Theorem~\ref{thm: fermionic formula}) due to Haglund--Haiman--Loehr--Remmel--Ulyanov. 
Moving forward, in Section~\ref{Sec: generalized Macdonald polynomials}, we explore generalized Macdonald polynomials for filled diagrams, and the column exchange rule introduced in \cite{KLO22}. 
Section~\ref{Sec: proof of main theorem} proves the lightning bolt formula for Macdonald intersection polynomials which is our key discovery that finalizes the proof of Theorem~\ref{thm: main theorem} (c).
In Section~\ref{Sec: Macdonald intersection polynomials at q=t=1}, we investigate the specialization of Macdonald intersection polynomials at $q=t=1$, and present an $h$-expansion formula (Theorem~\ref{thm: h-expansion, Kreweras}) that incorporates the Kreweras numbers. 
Section~\ref{Sec: Future questions}  presents a few remarks and discusses future questions.
Finally, Appendix~\ref{Sec: Appendix} includes a compilation of technical lemmas.

\section{Preliminaries}\label{Sec: Preliminaries}
\subsection{Partitions}

Let $n$ be a positive integer. A \emph{composition} of $n$ is a sequence $\mu=(\mu_1,\dots,\mu_\ell)$ of positive integers that sum to $n$. We denote $\ell(\mu)$ to be the length of $\mu$, number of its parts, and $|\mu|:=\sum_{i}\mu_i$ to be its size. We say that a composition $\mu$ is a \emph{partition} if its parts are weakly decreasing. We use the notation $\mu\models n$ to indicate that $\mu$ is a composition of $n$, and $\mu\vdash n$ to indicate that $\mu$ is a partition of $n$. We use \emph{Young diagrams} in French notation to display a partition as
\[
    \mu=\{(i,j)\in \mathbb{Z}_+\times \mathbb{Z}_+ : j\le \mu_i\},
\]
where each element of this set represents a \emph{cell} in the diagram. The \emph{conjugate partition} of $\mu$, denoted by $\mu'=(\mu'_1,\mu'_2,\dots)$, is obtained by reflecting the Young diagram of $\mu$ along the diagonal $y=x$. 

For a cell $u$ in a partition $\mu$, we define the \emph{arm} (\emph{coarm}) of $u$, denoted by $\arm_\mu(u)$ ($\coarm_\mu(u)$), as the number of cells strictly to the right (left) of $u$ in the same row. Similarly, we define the \emph{leg} (\emph{coleg}) of $u$, denoted by $\leg_\mu(u)$ ($\coleg_\mu(u)$), as the number of cells strictly above (below) $u$ in the same column. For example, for a partition $\mu=(5,4,3,1)$ and a cell $u=(2,2)$, we have $\arm_\mu(u)=2, \coarm_\mu(u)=1, \leg_\mu(u)=1$ and $\coleg(u)_\mu=1$ as displayed in the figure below.
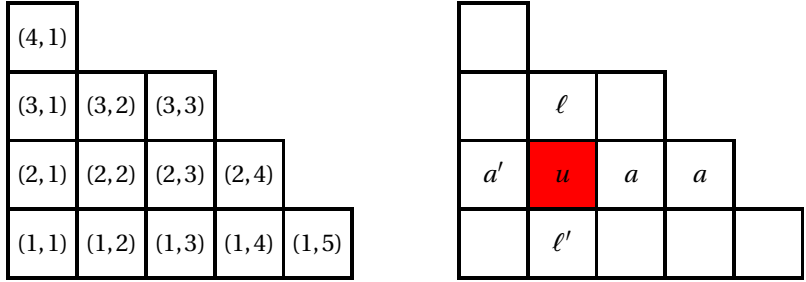
\[
    \ytableausetup{boxsize=2.5em, boxframe=0.1em}
    \begin{ytableau}
    (4,1)  \\
    (3,1) & (3,2) & (3,3) \\
    (2,1) & (2,2) & (2,3) & (2,4)\\
    (1,1) & (1,2) & (1,3) & (1,4) & (1,5)
    \end{ytableau}\qquad\qquad    \begin{ytableau}
      \\
    & \ell &  \\
    a' & *(red) u & a & a\\
    &  \ell' &  &  &
    \end{ytableau}
\]

\subsection{Symmetric functions}
Let $\mathbf{Sym}$ be the graded ring of symmetric functions in infinite variables $X=x_1,x_2,\dots$ over the ground field $\mathbb{C}(q,t)$. We use the notations from \cite{Mac88} for families of symmetric functions indexed by a partition: $m_\lambda$ for the \emph{monomial symmetric function}, $h_\lambda$ for the \emph{homogeneous symmetric function}, $e_\lambda$ for the \emph{elementary symmetric function}, $p_\lambda$ for the \emph{power sum symmetric function}, $s_\lambda$ for the \emph{Schur function}, and $J_\lambda$ for the \emph{integral form Macdonald polynomial} (or \emph{Macdonald $J$-polynomial}). The (Hall) inner product on symmetric functions, denoted by $\langle -,-\rangle$, is defined such that
\begin{equation*}
    \langle s_\lambda, s_\mu \rangle = \delta_{\lambda,\mu}.
\end{equation*}

The \emph{modified Macdonald polynomials} are defined using a plethystic substitution as follows:
\[
    \widetilde{H}_\mu[X;q,t] = t^{n(\mu)}J_\mu\left[\dfrac{X}{1-t^{-1}};q,t^{-1}\right].
\]
where $n(\mu)=\sum_i (i-1)\mu_i$. Each of the above families of symmetric functions form a basis for $\mathbf{Sym}$.

\subsection{Quasisymmetric functions}
Let $\mathbf{QSym}$ be the graded ring of quasisymmetric functions in infinite variables $X=x_1,x_2,\dots$ over the ground field $\mathbb{C}(q,t)$.  Given a composition $\alpha=(\alpha_1,\alpha_2,\dots,\alpha_\ell)\models n$, the \emph{monomial quasisymmetric function} $M_\alpha[X]$ is defined by  
\[
    M_\alpha[X] \coloneqq \sum_{i_1<i_2<\cdots<i_\ell} x_{i_1}^{\alpha_1}x_{i_2}^{\alpha_2}\cdots x_{i_\ell}^{\alpha_\ell},
\]
and the (Gessel's) \emph{fundamental quasisymmetric function} \cite{Ges84} $F_\alpha[X]$ is defined by
\begin{equation*}
    F_\alpha[X] \coloneqq \sum_{\beta\le\alpha}M_\beta,
\end{equation*}
where $\beta\le \alpha$ means we can obtain $\alpha$ by adding consecutive parts of $\beta$. Each of the above two families forms a basis for $\mathbf{QSym}$.

Compositions of $n$ are in one-to-one correspondence with the subsets of $[n-1]$. We will denote the natural bijection $\Comp$ from $2^{[n-1]}$ to compositions of $n$ defined by
\begin{align*}
\Comp:\{s_1<s_2<\cdots<s_\ell\} \mapsto (s_1,s_2-s_1,\dots,s_\ell-s_{\ell-1},n-s_\ell), \qquad \Comp(\emptyset)=(n).
\end{align*}
The inverse of $\Comp$ is denoted by $\Set$. Under this correspondence, for a subset $S\subseteq[n-1]$, we write $F_{n,S}$ as shorthand for $F_{\Comp(S)}$. If $n$ is clear from the context, we denote $F_{S}$ for $F_{n,S}$. Given a permutation $w\in \mathfrak{S}_n$, we will usually associate a fundamental quasisymmetric function $F_{\iDes(w)}$, where 
\begin{equation*}
\iDes(w)\coloneqq\Des(w^{-1})=\{i:w^{-1}_i>w^{-1}_{i+1}\}.
\end{equation*}

We end this section with a useful identity which will be used in Section~\ref{Sec: proof of main theorem}. For $n\ge1$, define a linear operator $E_n^\perp$ on $\mathbf{QSym}$ by
\begin{align}E_{n}^\perp F_\alpha=\begin{cases}
    F_{\beta^-} \quad \text{if} \quad \alpha = (1^{(n-1)},\beta)\label{eq: qsym perp}\\
    0 \quad \text{otherwise},
    \end{cases}
\end{align}
where $\beta^-$ is the composition obtained by subtracting one from the first component of $\beta$ if $\beta_1>1$, otherwise $\beta^{-}=(\beta_2,\dots)$. It is worth noting that for a composition $\beta=(b_1,\dots,b_r)$ there are two compositions
\begin{equation*}
    \alpha^{(1)}=(\underbrace{1,\dots,1}_{n-1},b_1+1,b_2,\dots,b_r), \qquad \alpha^{(2)}=(\underbrace{1,\dots,1}_{n},b_1,\dots,b_r)  
\end{equation*}
such that $E_{n}^\perp F_{\alpha^{(1)}}=E_{n}^\perp F_{\alpha^{(2)}}=F_{\beta}$. Then the skewing operator $e_n^\perp$ on symmetric functions is related to $E_n^\perp$ as follows.
\begin{proposition}\cite{GKLLRT95}\label{prop: en perp gessel fundamental quasi}
For a symmetric function $f$, we have
\[
    e_n^\perp f = E_n^\perp f,
\]
where $f$ on the right-hand side is considered as a quasisymmetric function.
\end{proposition}

For the rest of the paper, we abuse the notation and refer to $E_n^{\perp}f$ as $e_n^{\perp}f$ for $f\in\mathbf{QSym}$.
\begin{rmk}
The ring $\mathbf{NSym}$ of \emph{noncommutative symmetric functions} is the dual of the ring $\mathbf{QSym}$ of quasisymmetric functions with respect to (extended) Hall inner product. The \emph{elementary noncommutative symmetric functions} $E_\alpha$ form a basis for $\mathbf{NSym}$. In addition, for a partition $\lambda$, $E_\lambda$ descends to the elementary symmetric function $e_\lambda$ via natural surjection from $\mathbf{NSym}$ onto $\mathbf{Sym}$. The operator $E_n^\perp$ is the adjoint operator of multiplying $E_n$ with respect to the Hall inner product. Although there is a choice of $E_n$ which descends to $e_n$, regardless of choice, Proposition~\ref{prop: en perp gessel fundamental quasi} holds. For our purpose, we have chosen the elementary noncommutative symmetric functions $E_n$.
\end{rmk}

\section{Proof of Theorem~\ref{thm: main theorem} (a) and (b)}\label{Sec: Vanishing identity, Macdonald intersection polynomials and nabla}
In this section, we prove two results concerning Macdonald intersection polynomials. The first is Theorem~\ref{thm: main theorem} (a), the vanishing identity. The second is Theorem~\ref{thm: main theorem} (b), which is an identity between the Macdonald intersection polynomials $\I_{\mu^{(1)},\dots,\mu^{(k)}}[X;q,t]$ and $\nabla e_{k-1}$. This, in particular, gives the shape independence, which asserts that 
\[
\dfrac{1}{T_{\bigcap_{i=1}^k \mu^{(i)}}}e_{n+1-k}^\perp\I_{\mu^{(1)},\dots,\mu^{(k)}}[X;q,t]
\]
does not depend on partitions $\mu^{(1)}, \dots, \mu^{(k)}$, and $\mu$.

\subsection{Plethystic formula for the Macdonald polynomials}
We first review the plethystic formula for Macdonald polynomials studied by Garsia, Haiman, and Tesler in \cite{GT96, GHT99}. To do so, we introduce some relevant notation and basic facts. Let $g[t_1,t_2,\dots]$ be a Laurent series in the variables $t_1,t_2,\dots$. If a symmetric function $f$ is expressed as the formal power series
\[
    f = \hat{f}[p_1,p_2,\dots],
\]
then the \emph{plethystic substitution} of $g$ into $f$ is
\[
    f[g] = \hat{f}[p_1,p_2,\dots]|_{p_k\mapsto g[t_1^k,t_2^k,\dots]}.
\]
We will use brackets to denote plethystic substitutions. We use two different minus signs $-$ and $\epsilon$ to denote
\[
    p_k[-X]=-p_k[X], \qquad \text{ and } \qquad p_k[\epsilon X] = (-1)^k p_k[X].
\]
It is worth noting that, for a symmetric function $f$, we have
\[
    \omega f [X] = f[-\epsilon X],
\]
where $\omega$ is the involution on $\mathbf{Sym}$ sending $p_\lambda \rightarrow (-1)^{|\lambda|-\ell(\lambda)}p_\lambda$.
We denote the biexponent generator for a given partition $\mu$ by $B_\mu(q,t)$, which can be expressed as $$B_\mu(q,t) := \sum_{c\in\mu} q^{\coarm_\mu(c)}t^{\coleg_\mu(c)}.$$ Additionally, we define $M := (1-q)(1-t)$ and $D_\mu =D_\mu(q,t) := MB_\mu(q,t) - 1$.

Following the notation of \cite{GT96}, the $*$-inner product, denoted by $\langle -, - \rangle_*$, is defined such that 
$$\langle p_{\rho},p_{\nu}\rangle_* = \begin{cases} (-1)^{|\rho|-\ell(\rho)} z_\rho \prod_i (1-q^{\rho_i})(1-t^{\rho_i}) & \text{ if } \rho=\nu \\ 0 & \text{ otherwise,}\end{cases}$$ where $z_\rho$ is the integer that makes $n!/z_\rho$ the number of permutations of cycle type $\rho$. It is straightforward to check $\langle f,g\rangle_*=\langle \omega(f),\omega(g)\rangle_*$ for any symmetric functions $f$ and $g$. The modified Macdonald polynomials form an orthogonal basis with respect to the $*$-inner product \cite[Theorem 1.1]{GT96}. In particular, we have
\begin{equation}\label{Eq: orthogonality of Mac w.r.t <>*}
    \langle\widetilde{H}_\mu ,\widetilde{H}_\lambda\rangle_*=\begin{cases}
    \tilde{h}_\mu(q,t)\tilde{h}'_\mu(q,t) &\text{ if } \mu = \lambda \\
    0, &\text{ otherwise,}
    \end{cases}
\end{equation}
where
\[
    \tilde{h}_\mu(q,t) =\prod_{c\in\mu}(q^{\arm_\mu(c)}-t^{\leg_\mu(c)+1}) \qquad \text{and} \qquad \tilde{h}'_\mu(q,t)=\prod_{c\in\mu}(t^{\leg_\mu(c)}- q^{\arm_\mu(c)+1}).
\]

The $*$-inner product has several useful properties, one of which is \cite[Proposition 1.8]{GHT99}: For two symmetric functions $P$ and $Q$ of degree $n$, we have
\begin{equation}\label{Eq: *inner}
    \langle P[X],Q[X] \rangle_* = \langle P[MX], \omega Q[X]\rangle,    
\end{equation}
where $\langle-,-\rangle$ is the usual Hall inner product. 

Garsia, Haiman, and Tesler provided a general result for the $*$-inner product with modified Macdonald polynomials, which extends the plethystic formula for the $(q,t)$-Kostka polynomials due to Garsia and Tesler \cite{GT96}. To state their result, it is helpful to introduce the notation $g^*[X] := g[X/M]$ for a symmetric function $g[X]$.
The main theorem of their work can be presented as follows. 
\begin{thm}\cite[Theorem I.2]{GHT99}\label{thm: GHT main}
For a symmetric function $f$, define $\Pi'_f$ by
\[
    \Pi'_f[X;q,t] \coloneqq \nabla^{-1} f[X-\epsilon].
\]
Then, for every partition $\mu \vdash n \geq m$ we have
\[
\langle e^*_{n-m}f,\widetilde{H}_\mu\rangle_* = \Pi'_f[D_\mu;q,t].
\]
\end{thm}

\subsection{Proof of Theorem~\ref{thm: main theorem} (a) and (b)}

Let $\mu$ be a partition of $n$. By applying \eqref{Eq: *inner}, and Theorem~\ref{thm: GHT main}, we have
\begin{align*}
    \Pi'_{f}[D_\mu;q,t] &= \langle e_{n-m}f[MX], \omega \widetilde{H}_\mu \rangle=\langle f[MX], e_{n-m}^\perp \omega\widetilde{H}_\mu \rangle\\
    &=\langle  f[X], \omega(e_{n-m}^\perp \omega \widetilde{H}_\mu) \rangle_*= \langle \omega f[X], e_{n-m}^\perp \omega \widetilde{H}_\mu \rangle_*.
\end{align*}

For a partition $\lambda\vdash m$, we let $f$ in the above equation be $\widetilde{H}_\lambda$. We apply the well-known symmetry of modified Macdonald polynomials (e.g., Theorem 1.5 in \cite{GT96}), which is
\[
    \omega\widetilde{H}_\mu=T_\mu \rev(\widetilde{H}_\mu),
\]
where $\rev$ is an involution on $\mathbf{Sym}$ defined by $q,t$-reversal:
\[
    \rev(f) = f|_{q\mapsto q^{-1}, t\mapsto t^{-1}}.
\]
Then we have
\begin{equation}\label{eq: before lemma 3.2}
\Pi'_{\widetilde{H}_\lambda}[D_\mu;q,t] 
    = \langle \omega \widetilde{H}_\lambda[X], e_{n-m}^\perp \omega \widetilde{H}_\mu \rangle_*
    = \langle T_\lambda \rev(\widetilde{H}_\lambda),T_\mu e^\perp_{n-m} \rev(\widetilde{H}_\mu)\rangle_*
    = T_\mu T_\lambda\langle  \rev(\widetilde{H}_\lambda), e^\perp_{n-m} \rev(\widetilde{H}_\mu)\rangle_*.
\end{equation}
To extract the operator $\rev$ outside of the $*$-inner product, we need the following lemma.

\begin{lem}\label{lem: rev}
For two symmetric functions $f,g \in \mathbf{Sym}$ of degree $m$, we have
\[ 
    \langle \rev(f), \rev(g) \rangle_*= (qt)^m \rev\left(\langle f, g \rangle_* \right) 
\]

\end{lem}
\begin{proof}
It suffices to show the claim for $p_{\rho},p_{\nu}$ for partitions $\rho,\nu\vdash m$. We have
\begin{align*}
    \langle \rev(p_{\rho}), \rev(p_\nu) \rangle_* &= \langle p_\rho, p_\nu \rangle_*
    = (-1)^{|\rho|-\ell(\rho)}z_\rho \delta_{\rho,\nu} \prod_{i\ge 1}(1-q^{\rho_i})(1-t^{\rho_i})
    \\ &= (qt)^m (-1)^{|\rho|-\ell(\rho)}z_\rho \delta_{\rho,\nu} \prod_{i\ge 1}(1-q^{-\rho_i})(1-t^{-\rho_i}) 
    = (qt)^m \rev\left(\langle p_\rho,p_\nu\rangle_* \right).
\end{align*}
\end{proof}

Applying Lemma \ref{lem: rev} to \eqref{eq: before lemma 3.2} gives

\begin{equation*}
    \dfrac{1}{(qt)^mT_\mu T_\lambda}\langle \widetilde{H}_\lambda, e_{n-m}^\perp \widetilde{H}_\mu\rangle_*= \rev\left(\Pi'_{\widetilde{H}_\lambda}[D_\mu;q,t]\right).
\end{equation*}
By using the orthogonality relation \eqref{Eq: orthogonality of Mac w.r.t <>*}, we see that
\begin{equation}\label{Eq: e perp to Mac}
    e_{n-m}^\perp \widetilde{H}_\mu
    = (qt)^m T_\mu \sum_{\lambda \vdash m} \rev\left(\Pi'_{\widetilde{H}_\lambda}[D_\mu;q,t]\right)\dfrac{T_\lambda \widetilde{H}_\lambda}{\tilde{h}_\lambda\tilde{h}'_\lambda}.
\end{equation}

Now, let $\mu\vdash n+1$ be a partition and $\mu^{(1)},\dots,\mu^{(k)}\subseteq \mu$ be $k$ distinct partitions such that 
\[
    |\mu\setminus\mu^{(1)}|=\cdots=|\mu\setminus\mu^{(k)}|=1.
\]
Then the Macdonald intersection polynomial can be written as
\[
    \I_{\mu^{(1)},\dots,\mu^{(k)}}[X;q,t] =\sum_{i=1}^k \prod_{j\neq i}\left(\frac{z_j}{z_j-z_i}\right) \widetilde{H}_{\mu^{(i)}},
\]
where we denote $z_i=T_{\mu^{(i)}}/T_\mu$. Note that $T_{\mu^{(i)}}=z_iT_\mu$ and $T_{\cap_{i=1}^k \mu^{(i)}}=z_1\cdots z_k T_\mu$. Then using \eqref{Eq: e perp to Mac} we can show that $e^\perp_{n-m}\I_{\mu^{(1)},\dots,\mu^{(k)}}[X;q,t]$ is equal to
\begin{align}
    &(qt)^m {z_1\cdots z_k T_\mu} \sum_{\lambda \vdash m} \left(\sum_{i=1}^k \frac{\rev\left(\Pi'_{\widetilde{H}_\lambda}[D_{\mu^{(i)}};q,t]\right)}{\prod_{j\neq i}(z_j-z_i)} \right)\dfrac{T_\lambda\widetilde{H}_\lambda}{\tilde{h}_\lambda\tilde{h}'_\lambda}\nonumber\\
    &=(qt)^m T_{\cap_{i=1}^k \mu^{(i)}} \sum_{\lambda \vdash m} \left(\sum_{i=1}^k \frac{\rev\left(\Pi'_{\widetilde{H}_\lambda}[D_{\mu^{(i)}};q,t]\right)}{\prod_{j\neq i}(z_j-z_i)} \right)\dfrac{T_\lambda\widetilde{H}_\lambda}{\tilde{h}_\lambda\tilde{h}'_\lambda}\label{Eq: e perp Mac intersection}.
\end{align}

We will make use of a well-known equality, which we state as a lemma.
\begin{lem}\label{Lemma: sum prod 1/(x_i-x_j)}
Let $z_1,\ldots,z_k$ be variables. Then, for a polynomial $g(z)$ (in one variable $z$) of degree less than $k-1$, we have
$$\sum_{i=1}^k \dfrac{g(z_i)}{\prod_{j \neq i} (z_j-z_i)}=0.$$
If $g(z)$ is a polynomial of degree $k-1$ with leading term $az^{k-1}$, then
$$\sum_{i=1}^k \dfrac{g(z_i)}{\prod_{j \neq i} (z_j-z_i)}=(-1)^{k-1}a.$$
\end{lem}
\begin{proof}
    By Lemma~\ref{Lemma: degree less than k-1 vanishes} and Lemma~\ref{Lemma: monomial symmetric sum}. 
\end{proof}

\begin{proof}[Proof of Theorem~\ref{thm: main theorem} (a) (vanishing identity )]
Consider the expression in \eqref{Eq: e perp Mac intersection} and assume $m<k-1$. Then, the term
\begin{align*}
    \rev\left(\Pi'_{\widetilde{H}_\lambda}[D_\mu^{(i)};q,t])\right)&=
    \rev\left(\nabla^{-1}\widetilde{H}_\lambda\left[D_\mu-\epsilon-Mz_i^{-1};q,t\right]\right)
    \\
    &=T_\lambda \widetilde{H}_\lambda \left[D_\mu(q^{-1},t^{-1}) -\epsilon - \dfrac{M}{qt}z_i;q^{-1},t^{-1}\right]
\end{align*}
is a polynomial in $z_i$ of degree $m$ which is less than $k-1$. As a consequence, using Lemma~\ref{Lemma: sum prod 1/(x_i-x_j)}, we have that \eqref{Eq: e perp Mac intersection} vanishes. This completes the proof.
\end{proof}

\begin{proof}[Proof of Theorem~\ref{thm: main theorem} (b)]
In \eqref{Eq: e perp Mac intersection}, let $m=k-1$. Then 
\[
    \rev\left(\Pi'_{\widetilde{H}_\lambda}[D_\mu;q,t])\right)=T_\lambda \widetilde{H}_\lambda \left[D_\mu(q^{-1},t^{-1}) -\epsilon - \dfrac{M}{qt}z_i;q^{-1},t^{-1}\right]
\]
is a polynomial of degree $k-1$ in $z_i$. Moreover, the leading term is equal to 
\begin{align*}
T_\lambda \widetilde{H}_\lambda\left[-\dfrac{M}{qt};q^{-1},t^{-1}\right] &= \omega \widetilde{H}_\lambda\left[-\dfrac{M}{qt};q,t\right]
\\&=(-qt)^{1-k} \widetilde{H}_\lambda\left[M;q,t\right]
\\&=(-qt)^{1-k} M \Pi_\lambda B_\lambda
\end{align*}
where $\Pi_{\lambda}$ is a constant associated to a partition $\lambda$ defined as
\begin{equation*}
    \Pi_{\lambda}:=\prod_{c\in \lambda\setminus\{(1,1)\}}(1-q^{\coarm_{\lambda}(c)}t^{\coleg_{\lambda}(c)}).
\end{equation*} Here, the last equality can be obtained from the generalized Macdonald-Koornwinder reciprocity \cite[Theorem 3.3]{GHT99} when $\mu =(1)$. Using Lemma~\ref{Lemma: sum prod 1/(x_i-x_j)}, \eqref{Eq: e perp Mac intersection} becomes
\begin{align*}
    e_{n+1-k}^\perp\I_{\mu^{(1)},\dots,\mu^{(k)}}[X;q,t] &= (qt)^{k-1}T_{\cap_{i=1}^k \mu^{(i)}} \sum_{\lambda \vdash k-1} \left(\sum_{i=1}^k \prod_{j\neq i}\frac{1}{z_j-z_i} \rev\left(\Pi'_{\widetilde{H}_\lambda}[D_\mu;q,t])\right)\right)\dfrac{T_\lambda\widetilde{H}_\lambda}{\tilde{h}_\lambda\tilde{h}'_\lambda}
    \\&=T_{\cap_{i=1}^k \mu^{(i)}} \sum_{\lambda \vdash k-1} \dfrac{M \Pi_\lambda B_\lambda T_\lambda \widetilde{H}_\lambda}{\tilde{h}_\lambda\tilde{h}'_\lambda}=T_{\cap_{i=1}^k \mu^{(i)}} \nabla e_{k-1}
\end{align*}
where the last equality is from \cite[Theorem 3.2]{Hai02}. 
\end{proof}

\section{The fermionic formula for the shuffle formula}\label{Sec: New fermionic formula}
\subsection{Fermionic formulas}
Several important families of $q$-polynomials that arise in combinatorics and representation theory can be expressed using \emph{fermionic formulas}. These formulas involve a sum of products of $q$-multinomials and nonnegative powers of $q$. The origin of these formulas can be traced back to the Bethe ansatz \cite{Bet93} and $q$-Kostka polynomial $K_{\lambda,\mu}(q)$ \cite{KKR88}. Similarly, we refer to a fermionic formula for the sums of products of $q$-multinomials and nonnegative powers of $q$ and $t$.

Garsia and Haiman \cite{GH96Catalan} introduced a rational function $C_n(q,t)$ called the \emph{$q,t$-Catalan numbers} and conjectured that $C_n(q,t)$ is the Hilbert series of the diagonal coinvariant alternants, or $\langle D_n[X;q,t], e_n\rangle$. This number $C_n(q,t)$ specializes to the well-known $q$-Catalan numbers:
\[
    q^{\binom{n}{2}}C_n(q,1/q) = C^{(1)}_n(q) \text{ and } C_n(1,q)=C_n(q,1)=C^{(2)}_n(q),
\]
where $C^{(1)}_n(q)$ is the major-index generating function over Dyck words, and $C^{(2)}_n(q)$ is the area generating function for Dyck paths, justifying the name $q,t$-Catalan numbers. Haglund \cite{Hag03} gave a fermionic formula for $C_n(q,t)$:
\begin{equation}\label{eq: fermionic formula for q,t-Catalan}
    C_n(q,t) =  \sum_{\alpha \models n}{t^{\sum_{i=1}^{\ell(\alpha)}(i-1)\alpha_i}q^{\sum_{i=1}^{\ell(\alpha)} \binom{\alpha_i}{2} } \prod_{1\le i\le \ell(\alpha)-1} \binom{\alpha_i + \alpha_{i+1} -1}{\alpha_{i+1}}}_q.
\end{equation} 

Meanwhile, Haglund and Loehr \cite{HL05} conjectured a fermionic formula for the Hilbert series of the diagonal coinvariant algebra, or $\mathcal{H}_n(q,t)=\langle D_n[X;q,t], e_{(1^n)} \rangle$. This conjecture was eventually proved in \cite{CM18, CO18}, yielding the following formula:
\begin{equation}\label{eq: fermionic formula for Hilbert polynomial}
\mathcal{H}_n(q,t) = \sum_{\sigma \in \mathfrak{S}_n}t^{\maj(\sigma)}\prod_{i=1}^{n} [u_i(\sigma)-1]_q,
\end{equation}
where $\maj(\sigma)=\sum_{i\in \Des(\sigma)}i$, and $u_i$ is defined as follows. For a permutation $\sigma$, define a sequence $\sigma'$ by setting $\sigma'_i=\sigma_i$ for $1\le i \le n$ and $\sigma'_{n+1}=0$. For each $1 \le i \le n$, $u_i(\sigma)$ is the length of the longest consecutive sequence $\sigma'_{i}\dots\sigma'_{j}$ satisfying either
\begin{itemize}
    \item $\sigma'_{i}\dots\sigma'_{j}$ has no descent, or
    \item $\sigma'_{i}\dots\sigma'_{j}$ has one descent and $\sigma'_i>\sigma'_j$.
\end{itemize} 
In Section~\ref{subsec: 4.2}, we review a generalization (Theorem~\ref{thm: fermionic formula}) of these two fermionic formulas \eqref{eq: fermionic formula for q,t-Catalan} and \eqref{eq: fermionic formula for Hilbert polynomial}.

\subsection{Modified labeled Dyck paths and the shuffle formula}\label{subsec: 4.2}

The shuffle theorem concerns a combinatorial formula $D_n[X;q,t]$ for $\nabla e_n$ as a generating function over parking functions. Several equivalent formulations of $D_n[X;q,t]$ exist in the literature. For our purpose, we will use the version presented in \cite[Theorem 14]{LN14} employing the concept of \emph{modified labeled Dyck paths}, which Haglund and Loehr have first investigated in \cite{HL05}.

\begin{definition}\label{def: mld definition} A Dyck path of length $n$ is a lattice path from $(0,0)$ to $(n,n)$ that does not go below the line $y=x$. We denote a Dyck path by a sequence of $E$ and $N$, where $E$ is an east step and $N$ is a north step. We say that a Dyck path has a \emph{corner} at $(i,j)$ if there is an east step from $(i-1,j)$ to $(i,j)$ followed by a north step from $(i,j)$ to $(i,j+1)$. Given a Dyck path $\pi$ and a permutation $w$, a \emph{labeled Dyck path} $L(\pi,w)$ is a Dyck path $\pi$ with its $i$-th east step and $i$-th north step are labeled by $w_{n+1-i}$. For example, on the left in Figure~\ref{fig: mld example nonexample}, $\pi=NNNENNENEEEE$ and $w=326514$, therefore we have
\begin{equation*}
L(\pi,w)=((N,4),(N,1),(N,5),(E,4),(N,6),(N,2),(E,1),(N,3),(E,5),(E,6),(E,2),(E,3)).    
\end{equation*}
The permutation $w$ can be obtained from $L(\pi,w)$ by reading labels of the north steps from top to bottom (or equivalently, by reading labels of the east steps from right to left). We call a pair $(\pi,w)$ is a \emph{modified labeled Dyck path} if for each corner of $\pi$ at $(i,j)$, the label of the $i$-th east step is less than the label of the $(j+1)$-th north step in $L(\pi,w)$. We present a modified labeled Dyck path $(\pi,w)$ by a Dyck path $\pi$ and a permutation $w$ assigned to the diagonal $y=x$ from top to bottom.

The set of modified labeled Dyck paths of length $n$ will be denoted as $\MLD(n)$. We call a cell by $(i, j)$-box if $i$ and $j$ are the $x-$ and $y-$ coordinates of
its upper-right corner. Then for a pair $(\pi,w)$ of a Dyck path $\pi$ of length $n$ and a permutation $w$ of length $n$, the quantity $\area'(\pi,w)$ is defined as the number of pairs $(i,j)$ such that $(i,j)$-box lies between $\pi$ and the diagonal $y=x$ and the label of the $i$-th east step is less than the label of the $j$-th north step in $L(\pi,w)$.

Given a Dyck path $\pi$ of length $n$, we define its \emph{bounce path} as follows. Begin at $(0,0)$ and move northward until the beginning of an east step is encountered. Then, travel eastward until the diagonal $y=x$ is reached. From this point, repeat the process by traveling northward until the next east step is encountered, then moving eastward until the diagonal $y=x$ is reached. Continue this manner until we reach $(n,n)$. The \emph{bounce vector} $\overrightarrow{b}(\pi)=(\alpha_1,\alpha_2,\dots)$ associated to the Dyck path is a composition of $n$, obtained by reading the lengths of the individual bounces of the path from bottom to top. Then we define $\bounce(\pi)$ to be $\sum_{i=1}^{\ell(\alpha)}(i-1)\alpha_i$.
\end{definition}

\begin{figure}[h]\centering
\begin{tikzpicture}[scale=1]

\draw[-] (0,0) -- (6,0);
\draw[-] (0,1) -- (6,1);
\draw[-] (0,2) -- (6,2);
\draw[-] (0,3) -- (6,3);
\draw[-] (0,4) -- (6,4);
\draw[-] (0,5) -- (6,5);
\draw[-] (0,6) -- (6,6);

\draw[-] (0,0) -- (0,6);
\draw[-] (1,0) -- (1,6);
\draw[-] (2,0) -- (2,6);
\draw[-] (3,0) -- (3,6);
\draw[-] (4,0) -- (4,6);
\draw[-] (5,0) -- (5,6);
\draw[-] (6,0) -- (6,6);

\filldraw [red] (0.5,2.5) circle (2pt);
\filldraw [red] (1.5,3.5) circle (2pt);
\filldraw [red] (1.5,2.5) circle (2pt);
\filldraw [red] (1.5,4.5) circle (2pt);
\filldraw [red] (2.5,3.5) circle (2pt);
\filldraw [red] (4.5,5.5) circle (2pt);

\draw[green] (0,0) -- (6,6);

\draw[blue, very thick] (0,0) -- (0,3);
\draw[blue, very thick] (0,3) -- (1,3);
\draw[blue, very thick] (1,3) -- (1,5);
\draw[blue, very thick] (1,5) -- (2,5);
\draw[blue, very thick] (2,5) -- (2,6);
\draw[blue, very thick] (2,6) -- (6,6);

\filldraw[black] (0.5,0.75) circle (0.000001pt) node[anchor=north] {$4$};
\filldraw[black] (1.5,1.75) circle (0.000001pt) node[anchor=north] {$1$};
\filldraw[black] (2.5,2.75) circle (0.000001pt) node[anchor=north] {$5$};
\filldraw[black] (3.5,3.75) circle (0.000001pt) node[anchor=north] {$6$};
\filldraw[black] (4.5,4.75) circle (0.000001pt) node[anchor=north] {$2$};
\filldraw[black] (5.5,5.75) circle (0.000001pt) node[anchor=north] {$3$};
\filldraw[black] (0.5,3.75) circle (0.000001pt) node[anchor=north] {$<$};
\filldraw[black] (1.5,5.75) circle (0.000001pt) node[anchor=north] {$<$};

\begin{scope}[shift={(8,0)}]

\draw[-] (0,0) -- (6,0);
\draw[-] (0,1) -- (6,1);
\draw[-] (0,2) -- (6,2);
\draw[-] (0,3) -- (6,3);
\draw[-] (0,4) -- (6,4);
\draw[-] (0,5) -- (6,5);
\draw[-] (0,6) -- (6,6);

\draw[-] (0,0) -- (0,6);
\draw[-] (1,0) -- (1,6);
\draw[-] (2,0) -- (2,6);
\draw[-] (3,0) -- (3,6);
\draw[-] (4,0) -- (4,6);
\draw[-] (5,0) -- (5,6);
\draw[-] (6,0) -- (6,6);

\draw[green] (0,0) -- (6,6);

\draw[blue, very thick] (0,0) -- (0,3);
\draw[blue, very thick] (0,3) -- (1,3);
\draw[blue, very thick] (1,3) -- (1,5);
\draw[blue, very thick] (1,5) -- (2,5);
\draw[blue, very thick] (2,5) -- (2,6);
\draw[blue, very thick] (2,6) -- (6,6);

\filldraw[black] (0.5,0.75) circle (0.000001pt) node[anchor=north] {$5$};
\filldraw[black] (1.5,1.75) circle (0.000001pt) node[anchor=north] {$1$};
\filldraw[black] (2.5,2.75) circle (0.000001pt) node[anchor=north] {$4$};
\filldraw[black] (3.5,3.75) circle (0.000001pt) node[anchor=north] {$3$};
\filldraw[black] (4.5,4.75) circle (0.000001pt) node[anchor=north] {$6$};
\filldraw[black] (5.5,5.75) circle (0.000001pt) node[anchor=north] {$2$};
\filldraw[black] (0.5,3.75) circle (0.000001pt) node[anchor=north] {$\nless$};
\filldraw[black] (1.5,5.75) circle (0.000001pt) node[anchor=north] {$<$};

\end{scope}     

\end{tikzpicture}
\caption{An example and a nonexample of modified labeled Dyck path.} 
\label{fig: mld example nonexample}
\end{figure}
\begin{example}
On the left in Figure~\ref{fig: mld example nonexample}, a modified labeled Dyck path is depicted. On the other hand, the labeled Dyck path on the right in Figure~\ref{fig: mld example nonexample} is not a modified labeled Dyck path because the label of the east step at the first corner is 5, which is greater than the label of the north step, which is 3.

Let $(\pi,w)$ be the modified labeled Dyck path shown on the left in Figure~\ref{fig: mld example nonexample}. The bounce vector associated with $\pi$ is $\overrightarrow{b}(\pi)=(3,3)$, and $\bounce(\pi)=3$. Additionally, six boxes (indicated with dots on the left in Figure \ref{fig: mld example nonexample}) are located between the Dyck path and the main diagonal, where the label below the box is smaller than the label to its right. Therefore, we have $\area'(\pi,w)=6$.
\end{example}

The following theorem reformulates the shuffle formula $D_n[X;q,t]$ in terms of modified labeled Dyck paths:
\begin{thm}\label{thm: Bounce formulation shuffle}\cite[Theorem 14]{LN14} We have
\begin{equation*}
    D_n[X;q,t]=\sum_{(\pi,w)\in \MLD(n)}t^{\bounce(\pi)}q^{\area'(\pi,w)}F_{\iDes(w)}.
\end{equation*}
In other words, the shuffle theorem is equivalent to the identity:
\begin{equation*}
\nabla e_n=\sum_{(\pi,w)\in \MLD(n)}t^{\bounce(\pi)}q^{\area'(\pi,w)}F_{\iDes(w)}.
\end{equation*}
\end{thm}

\subsection{The fermionic formula}

For a (quasi)symmetric function $f$ of degree $n$, and a composition $\beta \models n$, we denote $[F_{\beta}](f)$ for the coefficient of $F_{\beta}$ in the expansion of $f$ into fundamental quasisymmetric functions. Then the following lemma is direct.

\begin{lem}\label{lem: mathfrakF = <e, >}
For a composition $\beta=(\beta_1,\beta_2,\dots)$ of $n$, let $e_\beta=e_{\beta_1}e_{\beta_2}\cdots$ be the elementary symmetric function. Then for a symmetric function $f$ of homogeneous degree $n$, we have
\[
    \langle f, e_{\beta}\rangle= \sum_{\substack{\alpha \models n \\ [n-1]\setminus \Set(\beta) \subseteq \Set(\alpha)  }}[F_\alpha](f),
\]
where the map $\Set$ is the bijection between compositions of $n$ and $2^{[n-1]}$.
\end{lem}
\begin{proof}
This follows from subsequently applying Proposition \ref{prop: en perp gessel fundamental quasi}.
\end{proof}

It is easy to check that given a permutation $w\in \mathfrak{S}_n$, we have 
\begin{equation}\label{eq: easy inference}
    \langle F_{\Des(w)}),e_{\beta}\rangle=1 \text{ if and only if } \{i: w_{i}<w_{i+1}\}\subseteq \Set(\beta).
\end{equation}

Note that \eqref{eq: fermionic formula for q,t-Catalan} and \eqref{eq: fermionic formula for Hilbert polynomial} give fermionic formulas for $\langle D_n[X;q,t],e_{n}\rangle$ and  $\langle D_n[X;q,t],e_{(1^n)}\rangle$, respectively. We will present a fermionic formula for $\langle D_{n}[X;q,t],e_{\beta}\rangle$ for general compositions $\beta \models n$, which essentially follow from \cite[Equation (63)]{HHLRU05}. With that purpose in mind, we introduce some notations. 

For positive integers $i,j$, the \emph{lightning bolt set} $\LLB(i,j)$ is the set of positions that are weakly above $(i,j-1)$ together with the positions that are weakly below $(i,j)$:
\[
    \LLB(i,j) \coloneqq \{(a,j-1):a\le i\} \cup \{(a,j):a\ge i\}.
\]
Here, we used matrix coordinates for \((i, j)\). Illustratively, the set $\LLB(i,j)$ can be represented as in Figure~\ref{fig:the lightning bolt set}.
\begin{figure}[h]
    \centering
    \begin{tikzpicture}
    \draw[line width=2pt, line cap=round, dash pattern=on 0pt off 1cm](0,0) grid (1,6);
    \draw (1,3)node[below]{$(i,j)$}; 
    \draw (-1,3)node{$\cdots$};
    \draw (2,3)node{$\cdots$};
    \begin{scope}[on background layer]
    \draw[yellow!60, line width=6pt * 1.5] (1,3) -- (1, -0.5    );
    \draw[yellow!60, line width=6pt * 1.5] (-0.15, 3) -- (1.15,3);
    \draw[yellow!60, line width=6pt * 1.5] (0,3) -- (0, 6.5);
    \end{scope}
        \end{tikzpicture}
    \caption{The lightning bolt set $\LLB(i,j)$.}
    \label{fig:the lightning bolt set}
\end{figure}

We denote the set of nonnegative integer matrices by $\Matoo=\bigcup_{1\le r,\ell}{\Mat^{\ge0}_{r\times\ell}}$, where ${\Mat^{\ge0}_{r\times\ell}}$ is the set of nonnegative integer matrices with $r$ rows and $\ell$ columns. A row sum vector (respectively column sum vector) of a matrix $M$ is denoted by $\csum(M)$ (respectively $\rsum(M$)), and let $|M|$ be the sum of entries of $M$. For example, the matrix
\[
M=\begin{bmatrix}
a & b \\
c & d 
\end{bmatrix}
\]
gives $\csum(M)=(a+b,c+d)$, $\rsum(M)=(a+c,b+d)$ and $|M|=  a+b+c+d$. For a matrix $M\in\Mat^{\ge0}_{r\times\ell}$ and $1\le i \le r, 2 \le j \le \ell$, we define $\LB(M;(i,j))$ as
\[
    \LB(M;(i,j)) \coloneqq \sum_{(a,b)\in \LLB(i,j)}M_{a,b}.
\]
  We also define $\LB(M)$ as follows:
\[
    \LB(M) \coloneqq \binom{M_{1,1}+\dots+M_{r,1}}{M_{1,1},\dots,M_{r,1}}_q \prod_{\substack{1\le i \le r\\ 2\le j \le \ell}} \binom{\LB(M;(i,j))-1}{M_{i,j}}_q,
\]
where the first term on the right is the $q$-multinomial coefficient and the second term is the product of $q$-binomial coefficients. We regard $\binom{-1}{0}_q=\binom{0}{0}_q=1$. The following basic property of $\LB(M)$ is worth noting.
\begin{lem}\label{lem: basic property of lbm}
    For a matrix $M$ such that $\LB(M)\neq0$, there exists a composition $\alpha$ such that $\rsum(M)=(\alpha,0,\dots,0)$. For such $M$, let $M'$ be the matrix obtained from $M$ by restricting to the first $\ell(\alpha)$ columns. Then we have $\LB(M)=\LB(M')$.
\end{lem}
\begin{proof}
    Assume there exists $j$ such that $j$-th column of $M$ is a zero column while $(j+1)$-th column is not a zero column. Let $i$ be the maximal integer such that $M_{i,j+1}\neq0$. Then we have 
    $\LB(M;(i,j+1))=M_{i,j+1}$ thus $\binom{M_{i,j+1}-1}{M_{i,j+1}}_q=0$. We conclude $\LB(M)=0$. Therefore, to have $\LB(M)\neq0$, zero columns of $M$ must be on the right of non-zero columns of $M$ implying $\rsum(M)=(\alpha,0,\dots,0)$ for some composition $\alpha$.

   Note that $M$ can be obtained from $M'$ by appending zero columns to the right. It is easy to check that appending a zero column to the right does not change the value $\LB(-)$.
\end{proof}
We now state the fermionic formula for \(\langle D_n[X;q,t],e_{\beta}\rangle\) of Haglund, Haiman, Loehr, Remmel, and Ulyanov \cite{HHLRU05}. Their original formulation uses the concept of a `\(\beta\)-shuffle', whereas we employ a different formulation. Spefically, \ \(\langle D_n[X;q,t],e_{\beta}\rangle\) is expressed as a sum of products of \(\LB(M)\) and nonnegative powers of \(q\) and \(t\) over all nonnegative integer matrices \(M\) of size \(\ell(\beta) \times \ell(\alpha)\) that satisfy certain constraints.
\begin{thm}\cite{HHLRU05}\label{thm: fermionic formula}
For a composition $\beta\models n$, we have
\[
 \langle D_n[X;q,t],e_{\beta}\rangle = \sum_{\alpha \models n} t^{\sum_{i=1}^{\ell(\alpha)}(i-1)\alpha_i} \sum_{\substack{M \in \Matoo \\ \csum(M)=\beta \\ \rsum(M)=\alpha}}  q^{\sum_{i,j}\binom{M_{i,j}}{2}  } \LB(M).
\]
\end{thm}

For the sake of completeness, we provide a proof in Section~\ref{subsec: Proof of fermionic formula} in our formulation. We end this section with an explanation of how \eqref{eq: fermionic formula for q,t-Catalan} and \eqref{eq: fermionic formula for Hilbert polynomial} can be regarded as special cases of the theorem above. Observe that the fermionic formula for $q,t$-Catalan numbers \eqref{eq: fermionic formula for q,t-Catalan} can be derived directly from Theorem~\ref{thm: fermionic formula} by setting $\beta=(n)$. 

To demonstrate that \eqref{eq: fermionic formula for Hilbert polynomial} also follows from Theorem~\ref{thm: fermionic formula}, let $\beta=(1^n)$. In this scenario, the summation is conducted over all nonnegative integer matrices with $n$ rows, each of whose rows contains precisely one nonzero entry, which is 1.

We can establish a correspondence between such matrices $M$ and permutations $\sigma(M)$ as follows: read the row numbers of the ones from top to bottom, starting from the rightmost column and moving leftward. For instance, the matrix
\begin{equation}\label{eq: eg matrix M}
M=\begin{bmatrix}
1 & 0 & 0\\
0 & 1 & 0\\
1 & 0 & 0\\
0 & 0 & 1
\end{bmatrix}
\end{equation}
corresponds to the permutation $\sigma(M)=4213$. For a permutation $\sigma\in\mathfrak{S}_n$, let $\Des(\sigma)=\{d_1<\dots<d_{\ell}\}$. Then for $1\leq i \leq \ell+1$, we define a word $\sigma^{(i)}\coloneqq\sigma_{d_{i-1}+1}\cdots \sigma_{d_{i}}$ where we regard $d_0=0$ and $d_{\ell+1}=n$. In other words, $\sigma^{(i)}$ is the $i$-th increasing run of $\sigma$. For example, for $\sigma=4213$, we have $\sigma^{(1)}=4$, $\sigma^{(2)}=2$, $\sigma^{(3)}=13$. Now we define an $n\times (\ell+1)$ matrix $M(\sigma)$ by letting its entries be defined as follows:
\[
    M(\sigma)_{i,j}:=
    \begin{cases}
    1&\text{ if } i\text{ appears in } \sigma^{(\ell+2-j)}\\
    0&\text{ otherwise}.
    \end{cases}
\]
It is straightforward to see that \(M(4213)\) is the matrix \(M\) given in \eqref{eq: eg matrix M}.
We claim that $M(\sigma)$ is the unique matrix $M$ with $\LB(M)\neq 0$ among matrices satisfying $\sigma(M)=\sigma$. To see this, suppose $M'$ is a matrix with $\sigma(M')=\sigma$ and $M'\neq M(\sigma)$. Then, there exists \(\sigma^{(r)}\) such that there are distinct \(i, i' \in \sigma^{(r)}\) and distinct \(j, j'\) where \(M_{i,j} = 1\) and \(M_{i',j'} = 1\). Assume $\sigma^{(i)}$ starts in the $b$-th column and let $a$ be the maximal integer such that $M'_{a,b}=1$. Then we have $M'_{c,d}=0$ for all $(c,d)\in\LLB(a,b)\setminus\{(a,b)\}$. This implies $\binom{\LB(M';(a,b))-1}{M'_{a,b}}=\binom{0}{1}_q=0$, hence $\LB(M')=0$. For example,
\[
M'=\begin{bmatrix}
0&1 & 0 & 0\\
0&0 & 1 & 0\\
1&0 & 0 & 0\\
0&0 & 0 & 1
\end{bmatrix}
\]
yields $\sigma(M')=4213$ and $\sigma^{(3)}=13$ lies in the first and the second columns. We have $\LB(M')=0$ since $\binom{\LB(M';(1,2))-1}{M'_{1,2}}=\binom{0}{1}_q = 0$. 

Let $\sigma$ be a permutation with $\Des(\sigma) = \{d_1<\cdots<d_\ell\}$, and consider $M=M(\sigma)$. For $j>1$ pick any $i$ such that $M_{i,j}=1$ and let $m$ be the corresponding position in $\sigma(M)=\sigma$, i.e., $\sigma_m = i$. Then it is easy to check that $u_m(\sigma)=\LB(M;(i,j))$. Entries equal to 1 in the first column of $M$ correspond to $\sigma^{(\ell+1)}=\sigma_{d_{\ell}+1}\cdots \sigma_{n}$ and for each $d_{\ell}+1\leq m\leq n$, we have $u_m(\sigma)=n+2-m$. Therefore we have 
\begin{equation*}
    \prod_{m=d_{\ell}+1}^{n}[u_m(\sigma)-1]_q=[n-d_{\ell}]_q\cdots[1]_q = [n-d_\ell]_q!
\end{equation*}
which equals to  $\binom{M_{1,1}+\dots+M_{n,1}}{M_{1,1},\dots,M_{n,1}}_q=\binom{n-d_\ell}{1,\dots,1}_q$. This shows that $\LB(M(\sigma))=\prod_{m=1}^{n}[u_m(\sigma)-1]_q$. Since we have
 $\maj(\sigma)=\sum_{i\ge 1} (i-1)\alpha_i$, where $\alpha=\rsum(M(\sigma))$, letting $\beta=(1^n)$ in Theorem~\ref{thm: fermionic formula} gives \eqref{eq: fermionic formula for Hilbert polynomial}.


\subsection{Proof of Theorem~\ref{thm: fermionic formula}}\label{subsec: Proof of fermionic formula}

In this subsection, we only consider  $M\in\Mat^{\geq0}$ without a zero column or a zero row. We provide a combinatorial interpretation of $\LB(M)$ in terms of modified labeled Dyck paths (Proposition~\ref{prop: combinatorial interpretation of LB(M)}). To that end, we fix a matrix $M \in \Mat^{\ge0}_{r\times\ell}$. We associate a vector $m\in \mathbb{Z}_{\ge0}^{r\ell}$ obtained by reading the entries of $M$ from left to right, top to bottom, i.e., $m_{n(i,j)}= M_{i,j}$, where $n(i,j):= \ell (i-1)+j$ for $1 \leq j \leq \ell$. Then we set 
\[
    I(M;(i,j)) := \left[1 + \sum_{a=1}^{n(i,j)-1} m_a, \sum_{a=1}^{n(i,j)} m_a\right].
\]
Now we define $\MLD(M)$ to be the set of all modified labeled Dyck paths $(\pi,w)$ that satisfy the following conditions:
\begin{enumerate}
    \item The bounce vector $\overrightarrow{b}(\pi)=(\alpha_1,\alpha_2,\dots)=\rsum(M)$,
    \item the set of entries of $w$ on the bounce corresponding to column $(\ell-j+1)$ is equal to $\cup_{i=1}^{r} I(M;(i,j))$, \text{ and }
    \item elements of $I(M;(i,j))$ are decreasing in $w$.
\end{enumerate}  

\begin{example}
    Let $M=\left(\begin{bmatrix}
1 & 2\\
2 & 1
\end{bmatrix}
\right)$. Then the associated vector $m = (1,2,2,1)$ and intervals $I(M;(i,j))$ are
\[
    I(M;(1,1))=\{1\}, \quad I(M;(1,2))=\{2,3\}, \quad I(M;(2,1))=\{4,5\}, \quad I(M;(2,2))=\{6\}.
\]
By definition, a modified labeled Dyck path $(\pi,w)\in \MLD(M)$ satisfies the following conditions.
\begin{enumerate}
    \item The bounce vector $\overrightarrow{b}(\pi)=(3,3)$, 
    \item $I(M;(1,1))\cup I(M;(2,1))=\{1,4,5\}=\{w_4,w_5,w_6\}$, and $I(M;(1,2))\cup I(M;(2,2))=\{2,3,6\}=\{w_1,w_2,w_3\}$, and
    \item 4 and 5 are placed in a decreasing order, and 2 and 3 are placed in a decreasing order in $w$.
\end{enumerate}
The modified labeled Dyck path on the left in Figure~\ref{fig: mld example nonexample} belongs to $\MLD(M)$. For the same Dyck path $\pi$, and $w=632541$, one can check that $(\pi,w)\in \MLD(M)$.
\end{example} 

For $(\pi,w)\in \MLD(M)$, we have $\langle F_{\iDes(w)},e_{\beta}\rangle=1$ if $\beta=\csum(M)$. Conversely, given $(\pi,w)\in \MLD(n)$ such that $\langle F_{\iDes(w)},e_{\beta}\rangle=1$ for some $\beta\models n$, we associate a matrix $M$ as follows. Let $\alpha=\overrightarrow{b}(\pi)$ be the bounce vector. We also let $s_j=\sum_{j<b} \alpha_b$ and $t_i=\sum_{b<i} \beta_b$. Then we define $\ell(\beta) \times \ell(\alpha)$ matrix $M$ by letting its entries be defined as follows: 
\begin{equation*}
    M_{i,j}=\left|\left\{w_{s_j+1}, \dots, w_{s_{j-1}}\right\}\ \cap \left[t_{i}+1,t_{i+1}\right]\right|.
\end{equation*}
Note that we have $(\pi,w)\in \MLD(M)$ and $\csum(M)=\beta$. Furthermore, $M$ is the unique matrix with such property. Therefore, together with Theorem~\ref{thm: Bounce formulation shuffle}, we obtain
\begin{align}\label{eq: mld mld m}
    \langle D_{n}[X;q,t],e_{\beta}\rangle=1= \sum_{\substack{M\in\Matoo \\ \csum(M)=\beta}}\sum_{{(\pi,w)\in\MLD(M)}}t^{\bounce(\pi)}q^{\area'(\pi,w)}.
\end{align}

\begin{proposition}\label{prop: combinatorial interpretation of LB(M)}
Let $M\in\Mat^{\ge0}_{r\times\ell}$. Then $\LB(M)$ is a generating function for $\area'$ over all modified labeled Dyck paths $\MLD(M)$:
\begin{equation*}
    \LB(M) = q^{-\sum_{i,j}\binom{M_{i,j}}{2}}\sum_{(\pi,w)\in\MLD(M)} q^{\area'(\pi,w)}.
\end{equation*}
\end{proposition}

\begin{proof}
Let $\ell = 1$. If $(\pi, w) \in \MLD(M)$ then the bounce vector of $\pi$ is of length 1, so $\pi$ must be the full Dyck path, i.e., $\pi = N \cdots N E \cdots E$. For a permutation $w \in \mathfrak{S}_{|M|}$, $(\pi, w) \in \MLD(M)$ if and only if the elements in $I(M; (i, 1))$ are in decreasing order in $w$ for all $i$. Then there is a natural bijection
\[
    \phi: \MLD(M) \rightarrow W_{M_{1,1}, M_{2,1}, \dots, M_{r,1}},
\]
where $W_{M_{1,1}, M_{2,1}, \dots, M_{r,1}}$ is the set of words of consisting of $M_{i,1}$-many $i$'s, and $\phi(\pi,w)$ is defined as the word obtained from the permutation $w$ by replacing entries in the interval $I(M; (i, 1))$ with $i$. Then, note that 
\begin{equation}\label{eq: area = sum binom + inv}
    \area'(\pi, w) = \inv(w) = \sum_{i=1}^r \binom{M_{i,1}}{2} + \inv(\phi(\pi,w)).
\end{equation}
One can see this by dividing the contributions to $\inv(w)$ into two cases. The first case is the pair of entries in the same interval $I(M; (i, 1))$ for some $i$, contributing $\sum_{i=1}^r \binom{M_{i,1}}{2}$. The second case is the pair of integers in different intervals, $I(M; (i, 1))$ and $I(M; (i', 1))$, which are out of order. The contribution of the latter case is given by the usual inversions of $\phi(\pi,w)$. By \eqref{eq: area = sum binom + inv}, we have
\[
  \sum_{(\pi, w) \in \MLD(M)} q^{\area'(\pi, w)} = q^{\sum_{i=1}^r \binom{M_{i,1}}{2}} \binom{M_{1,1} + \dots + M_{r,1}}{M_{1,1}, \dots, M_{r,1}}_q = q^{\sum_{i=1}^r \binom{M_{i,1}}{2}} \LB(M).
\]
Here, we used the fact that the inversion generating function of $W_{M_{1,1}, M_{2,1}, \dots, M_{r,1}}$ is the $q$-multinomial coefficient $\binom{M_{1,1} + \dots + M_{r,1}}{M_{1,1}, \dots, M_{r,1}}_q$.

Suppose $\ell>1$, and let us proceed by induction on the sum $n=|M|$. We are dealing with $M$ without a zero column so the base case $|M|=1$ has been covered. Let $c$ be the smallest index $i$ such that $M_{i,\ell}\neq0$, and let $M'$ be the matrix obtained from $M$ by replacing $M_{c,\ell}$ with zero. If the rightmost column of $M'$ is a zero column, we simply erase that column.

Let $s=\LB(M;(c,\ell))$ and let $W^{0}_{s,M_{c,\ell}}$ be the set of 0-1 sequence of length $s$ with $M_{c,\ell}$-many ones starting with a zero. Suppose we construct a bijection
\begin{align*}
    \Phi:\MLD(M) &\rightarrow \MLD(M')\times W^{0}_{s,M_{c,\ell}}\\
    (\pi,w) &\mapsto (\Phi_1(\pi,w), \Phi_2(\pi,w))
\end{align*}
such that 
\begin{equation}\label{Eq: area' = area' + inv}
\area'(\pi,w)=\area'\left(\Phi_1(\pi,w)\right) + \inv\left(\Phi_2(\pi,w)\right) + \binom{M_{c,\ell}}{2}.
\end{equation}
Then we have
\begin{align*}
    \sum_{(\pi,w)\in\MLD(M)} q^{\area'(\pi,w)} & =
    q^{\binom{M_{c,\ell}}{2}}
    \sum_{(\pi,w)\in\MLD(M')} q^{\area'(\pi,w)}
    \sum_{\sigma\in W^{0}_{s,M_{c,\ell}}}q^{\inv(\sigma)}\\
    &=q^{\sum_{i,j}\binom{M_{i,j}}{2}}\LB(M') \binom{\LB(M;(c,\ell))-1}{M_{c,\ell}}_q\\
    &=q^{\sum_{i,j}\binom{M_{i,j}}{2}}\LB(M).
    \end{align*}
Here, the first equality follows from \eqref{Eq: area' = area' + inv}. The second equality follows from the induction hypothesis, and the last equality follows from the definition of $\LB(M)$. Therefore, it suffices to construct such a bijection $\Phi$ satisfying \eqref{Eq: area' = area' + inv}. 

Given a modified labeled Dyck path $(\pi,w)\in\MLD(M)$, recall that each step in $\pi$ has an associated label which we encoded in $L(\pi,w)$ (see Definition \ref{def: mld definition}). Let $\rsum(M)=\alpha$ then $\pi$ must pass through specific points $X_0$ and $X_1$ given by
\begin{equation}\label{eq: starting and ending point of P(pi,w)}
    X_{0} = \left(\sum_{1\le j \le \ell-2}\alpha_j,\sum_{1\le j \le \ell-1}\alpha_j\right)
    \quad \text{and} \quad    
    X_{1} = \left(\sum_{1\le j \le \ell-1}\alpha_j,n\right).
\end{equation}
In other words, the subpath from $X_0$ to $X_1$ corresponds to the last bounce of $\pi$ and we let $P(\pi,w)$ to be a subsequence of $L(\pi,w)$ for this subpath denoted by $P(\pi,w)=(P_1,\dots,P_{\alpha_{\ell-1}+\alpha_{\ell}})$. If $P_j=(E,a)$ for some label $a$ then we have $a\in\bigcup_{i=1}^r I(M;(i,\ell-1))$ and if $P_j=(N,a)$ we have  $a\in\bigcup_{i=1}^r I(M;(i,\ell))$. Define sets $A$, $B$ and $C$ given as
\begin{align*}
     A&=\{(E,a): a\in\bigcup_{i=1}^c I(M;(i,\ell-1))\}\cup \{(N,a): a\in\bigcup_{i=c+1}^r I(M;(i,\ell))\}\\
     B&=\{(N,a): a\in I(M;(c,\ell))\}\\
      C&=\{(E,a): a\in\bigcup_{i=c+1}^r I(M;(i,\ell-1))\}.
\end{align*}
Then each $P_i$ belongs to exactly one of $A$, $B$ or $C$. We define $\Phi_1(\pi,w)$ to be $(\pi',w')$ whose $L(\pi',w')$ is obtained from $L(\pi,w)$ by deleting $P_i$'s in $B$ and replacing each label $\gamma\in\bigcup_{c<i\le r,1\le j \le \ell}I(M;(i,j))$ by $\gamma-M_{c,\ell}$ to standardize the labels. From $P(\pi,w)$ we replace $P_i$'s in $A$ by 0, $P_i$'s in $B$ by 1 and delete $P_i$'s in $C$. The resulting 0-1 sequence is defined to be $\Phi_2(\pi,w)$. 

Given that $(N,a)\in A$ and $(N,b)\in B$ imply $a>b$, it follows that $(\pi',w')$ remains a modified labeled Dyck path. Let $j$ be the minimal number such that $P_j$ is in $A$ or $B$. If $P_j \in B$ then $j>1$ since $P_1$ must be of the form $(E,a)$. Thus, $P_{j-1}\in C$, which would violate the condition for a modified labeled Dyck path. Therefore $P_j\in A$ implying that $\Phi_2(\pi,w)$ starts with 0. This confirms that $\Phi$ is well-defined. 

Now we describe an inverse map from $(\pi',w')\in \MLD(M')$ and $v\in  W^{0}_{s,M_{c,\ell}}$. As before, let $P(\pi,w')=(P'_1,\dots)$ to be a subsequence of $L(\pi',w')$ corresponding to the last bounce. The following observation is crucial:
\begin{itemize}
    \item Given a sequence $\bar{v}$ of $0$, $1$ and $2$ where there are no consecutive 2 and 1 (in this order), we can uniquely determine $\bar{v}$ if we know the sequences $v$ (obtained from $\bar{v}$ by deleting 2's) and $v'$ (obtained from $\bar{v}$ by deleting 1's). 
\end{itemize}
Using this observation, proceed as follows:
\begin{enumerate}
    \item Generate a 0-2 Sequence: From \(P(\pi', w')\), create a 0-2 sequence \(v'\) by replacing \(P'_i\)s in
    \[
    C' = \{(E, a) \mid a \in \bigcup_{i=c+1}^r I(M'; (i, \ell-1))\}
    \]
    with 2 and replacing \(P'_i\)s not in \(C'\) with 0.
    
    \item Construct \(\bar{v}\): Using the above observation, combine \(v\) and \(v'\) to reconstruct \(\bar{v}\), which consists of 0, 1, and 2.
    
    \item Insert Elements: Insert elements of \(B\) into the sequence \(P(\pi', w')\) according to the positions of 1s in \(\bar{v}\).
    
    \item Standardize Labels: Shift the labels in \(L(\pi', w')\) properly to standardize them, yielding \(L(\pi, w)\) for some \((\pi, w) \in \MLD(M)\).
\end{enumerate}

Thus, we have demonstrated that \(\Phi\) is a bijection. The equation \(\eqref{Eq: area' = area' + inv}\) follows straightforwardly from the construction.

\end{proof}

\begin{example}\label{Ex: the bijection Phi}
    Let $(\pi,w)$ be the modified labeled Dyck path on the left in Figure~\ref{fig: mld example nonexample} in $\MLD\left(\begin{bmatrix}  1 & 2\\2 & 1\\ \end{bmatrix}\right).$ Or equivalently, we could describe $(\pi,w)$ as a modified labeled Dyck path such that 
    \[
        L(\pi,w)=((N,4),(N,1),(N,5),(E,4),(N,6),(N,2),(E,1),(N,3),(E,5),(E,6),(E,2),(E,3)),
    \]
    where its $i$-th component presents the $i$-th labeled step. 
    
    Let $\Phi$ be the bijection $\Phi:\MLD\left(\begin{bmatrix}
        1 & 2 \\ 2 & 1
    \end{bmatrix}\right) \rightarrow \MLD\left(\begin{bmatrix}
        1 & 0 \\ 2 & 1
    \end{bmatrix}\right)\times W^{0}_{4,2}$ defined in the proof of Proposition~\ref{prop: combinatorial interpretation of LB(M)}. First, by deleting steps labeled with numbers in the 
    interval $I(M;(1,2))=\{2,3\}$ gives 
    \[
    ((N,4),(N,1),(N,5),(E,4),(N,6),(E,1),(E,5),(E,6)),
    \]
    or equivalently
    
    
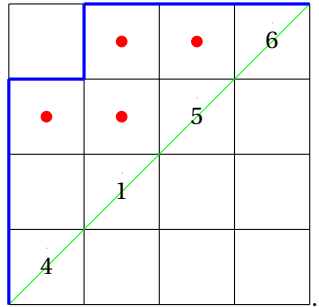
\begin{center}
    \begin{tikzpicture}[scale=1]
    
    \draw[-] (0,0) -- (4,0);
    \draw[-] (0,1) -- (4,1);
    \draw[-] (0,2) -- (4,2);
    \draw[-] (0,3) -- (4,3);
    \draw[-] (0,4) -- (4,4);
    
    \draw[-] (0,0) -- (0,4);
    \draw[-] (1,0) -- (1,4);
    \draw[-] (2,0) -- (2,4);
    \draw[-] (3,0) -- (3,4);
    \draw[-] (4,0) -- (4,4);

    \filldraw [red] (0.5,2.5) circle (2pt);
    \filldraw [red] (1.5,2.5) circle (2pt);
    \filldraw [red] (1.5,3.5) circle (2pt);
    \filldraw [red] (2.5,3.5) circle (2pt);

    \draw[green] (0,0) -- (4,4);
    
    \draw[blue, very thick] (0,0) -- (0,3);
    \draw[blue, very thick] (0,3) -- (1,3);
    \draw[blue, very thick] (1,3) -- (1,4);
    \draw[blue, very thick] (1,4) -- (4,4);
    
    \filldraw[black] (0.5,0.75) circle (0.000001pt) node[anchor=north] {$4$};
    \filldraw[black] (1.5,1.75) circle (0.000001pt) node[anchor=north] {$1$};
    \filldraw[black] (2.5,2.75) circle (0.000001pt) node[anchor=north] {$5$};
    \filldraw[black] (3.5,3.75) circle (0.000001pt) node[anchor=north] {$6$};
    \end{tikzpicture}.
    \end{center}
    Then, by replacing each label $\gamma\in I(M;(2,1))\cup I(M;(2,2)) = \{4,5,6\}$ by $\gamma-M_{c,\ell}=\gamma-2$, we have the following modified labeled Dyck path $\Phi_1(\pi,w)$

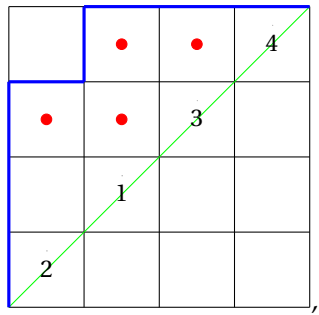
    \begin{center}
    \begin{tikzpicture}[scale=1]
    
    \draw[-] (0,0) -- (4,0);
    \draw[-] (0,1) -- (4,1);
    \draw[-] (0,2) -- (4,2);
    \draw[-] (0,3) -- (4,3);
    \draw[-] (0,4) -- (4,4);
    
    \draw[-] (0,0) -- (0,4);
    \draw[-] (1,0) -- (1,4);
    \draw[-] (2,0) -- (2,4);
    \draw[-] (3,0) -- (3,4);
    \draw[-] (4,0) -- (4,4);
    
    \filldraw [red] (0.5,2.5) circle (2pt);
    \filldraw [red] (1.5,2.5) circle (2pt);
    \filldraw [red] (1.5,3.5) circle (2pt);
    \filldraw [red] (2.5,3.5) circle (2pt);
    
    \draw[green] (0,0) -- (4,4);
    
    \draw[blue, very thick] (0,0) -- (0,3);
    \draw[blue, very thick] (0,3) -- (1,3);
    \draw[blue, very thick] (1,3) -- (1,4);
    \draw[blue, very thick] (1,4) -- (4,4);
    
    \filldraw[black] (0.5,0.75) circle (0.000001pt) node[anchor=north] {$2$};
    \filldraw[black] (1.5,1.75) circle (0.000001pt) node[anchor=north] {$1$};
    \filldraw[black] (2.5,2.75) circle (0.000001pt) node[anchor=north] {$3$};
    \filldraw[black] (3.5,3.75) circle (0.000001pt) node[anchor=north] {$4$};
    \end{tikzpicture},
    \end{center}
    which is in $\MLD\left(\begin{bmatrix}  1 & 0\\2 & 1\\ \end{bmatrix} \right)$.
    
    Let $\alpha=\rsum(M)=(3,3)$. The labeled subpath of $L(\pi,w)$ starting from 
    \[
    \left(\sum_{1\le j \le \ell-2}\alpha_j,\sum_{1\le j \le \ell-1}\alpha_j\right)=(0,3)
    \]
    and ending at 
    \[
    \left(\sum_{1\le j \le \ell-1}\alpha_j,n\right)=(3,6)
    \]
    is $P(\pi,w)=((E,4),(N,6),(N,2),(E,1),(N,3),(E,5))$. We get $\sigma(\pi,w)=462135$ by reading its labels. The 0--1 word obtained from $\sigma(\pi,w)$ by replacing numbers in 
    \[
    \bigcup_{(a,b)\in\LB(1,2)\setminus\{(1,2)\}} I(M;(a,b))=I(M;(1,1))\cup I(M;(2,2))=\{1,6\}
    \]
    with zeros, numbers in $I(M;(c,\ell))=\{2,3\}$ with ones, and deleting the other numbers, is $\Phi_2\left(\pi,w\right)=0101$. We can also check that
    \[
        \area'(\pi,w) = 6 = 4 + 1 + 1 = \area'(\Phi_1(\pi,w)) + \inv(\Phi_2(\pi,w)) + \binom{M_{1,2}}{2}.
    \]
    \end{example}

\begin{proof}[Proof of Theorem~\ref{thm: fermionic formula}]
Partitioning the sum on the right-hand side of \eqref{eq: mld mld m} by the bounce vector, we have
\begin{align*}
\langle D_n[X;q,t],e_{\beta}\rangle &=\sum_{\alpha \models n} t^{\sum_{i=1}^{\ell(\alpha)}(i-1)\alpha_i} \sum_{\substack{M \in \Matoo \\ \rsum(M)=\alpha \\ \csum(M)=\beta }} \sum_{(\pi,w)\in\MLD(M)} 
q^{\area'(\pi,w)}\\&=\sum_{\alpha \models n} t^{\sum_{i=1}^{\ell(\alpha)}(i-1)\alpha_i} \sum_{\substack{M \in \Matoo \\ \rsum(M)=\alpha \\ \csum(M)=\beta }}  q^{\sum_{i,j}\binom{M_{i,j}}{2}}  \LB(M),
\end{align*}
where the last equality follows from Proposition~\ref{prop: combinatorial interpretation of LB(M)}.
\end{proof}

\section{Macdonald polynomials for general 
 diagram and column exchange rule}\label{Sec: generalized Macdonald polynomials}
\subsection{A generalization of modified Macdonald polynomials}\label{subsec: generalization of Macdonald}

In \cite{KLO22}, the authors introduce a generalized version of the modified Macdonald polynomials. A diagram $D$ consists of a collection of cells, with a specific cell at the bottom of each column referred to as a \emph{bottom cell}. For a diagram $D$, the size $|D|$ of $D$ is the total number of cells it contains. In this paper, we only consider diagrams that are connected and finite. As in the partition case, we index rows of $D$ from the bottom and columns of $D$ from the left. Denote the cell of $D$ in the $i$-th row and $j$-th column by $(i,j)$. We represent a diagram $D$ by
\[
    D = [D^{(1)},D^{(2)},\dots],
\]
where $D^{(j)}$ denotes the set of $i$ such that $(i,j)$ is a cell of $D$. Throughout the paper, we assume that each column $D^{(j)}$ of a diagram $D$ forms an interval. 

\begin{figure}[h]
    \centering
    \begin{ytableau}
     \none & \none & q^2 & q^3 \\
     $ $ & t & qt & \\
     \none & $ $ & $ $ 
    \end{ytableau}
    \caption{An example of a filled diagram.}
    \label{fig:my_label}
\end{figure}

Let $\mathbb{F}$ be a field that includes $\mathbb{C}(q,t)$, the field of rational functions in $q$ and $t$. A \emph{filled diagram} $(D,f)$ consists of a diagram $D$ and a filling function
\[
    f:D\setminus\{\text{bottom cells of $D$}\}\rightarrow \mathbb{F}
\] 
which assigns a scalar from $\mathbb{F}$ to each cell of $D$ excluding the bottom cells. We represent a filled diagram $(D,f)$ by displaying the corresponding values of $f$ for each cell in the diagram, see Figure~\ref{fig:my_label} for the example. For a subdiagram $E$ of $D$, we consider a sub-filled diagram $(E,f\vert_{E})$ of $(D,f)$ where the filling $f\vert_{E}$ is inherited from $f$.

To establish a total order $N_D:D\rightarrow [|D|]$ on cells in $D$, we order them from left to right in each row and proceed row by row from top to bottom. We define two functions, $\inv_D:\mathfrak{S}\rightarrow \mathbb{F}$ and $\maj_{(D,f)}\rightarrow \mathbb{F}$, in the spirit of the functions $\inv_\mu$ and $\maj_\mu$ defined in \cite{HHL05}. For cells $u=(i,j)$ and $v=(i',j')$ of $D$, we say that a pair $(u,v)$ is an \emph{attacking pair} if either 
\begin{center}
    (1) $i=i'$ and $j<j'$ $\quad$ or$\quad$ (2) $i=i'+1$ and $j>j'$. 
\end{center}
For a permutation $w\in\mathfrak{S}_{|D|}$ we consider a pair $(u,v)$ of cells in $D$ to be an \emph{inversion pair} of $w$ if $(u,v)$ is an attacking pair and $w_{N_D(u)}>w_{N_D(v)}$. We denote the set of inversion pairs of $w$ by $\Inv_D(w)$ and define $\inv_D(w)$ to be 
\[
    \inv_D(w) \coloneqq q^{|\Inv_D(w)|}.
\]
We say that a cell $(i,j)$ in $D$ is a \emph{descent} of $w$ if $w_{N_D(i,j)}>w_{N_D(i-1,j)}$. We denote the set of descents of $w$ by $\Des_D(w)$ and define $\maj_{(D,f)}(w)$ as the product of $f(u)$ over all cells $u$ which are descents of $w$, i.e.,
\[
    \maj_{(D,f)}(w) \coloneqq \prod_{u\in\Des_D(w)} f(u).
\]
Finally, we define $\stat_{(D,f)}(w)$ by
\[
    \stat_{(D,f)}(w) \coloneqq \inv_D(w)\maj_{(D,f)}(w).
\]
Now, we define the modified Macdonald polynomial of a filled diagram $(D,f)$.

\begin{definition}\label{Def: generalized modified Macdoanld}
Let $(D,f)$ be a filled diagram of size $n$. The \emph{(generalized) modified Macdonald polynomials} of $(D,f)$ is defined by
\[
    \widetilde{H}_{(D,f)}[X]\coloneqq\sum_{w\in \mathfrak{S}_{|D|}} \stat_{(D,f)}(w)F_{\iDes(w)}.
\]
\end{definition}

If we define the \emph{standard filling} $f^{\st}_\mu$ of a partition $\mu$ by assigning $f^{\st}_\mu(u)=q^{-\arm_\mu(u)}t^{\leg_\mu(u)+1}$ to each cell $u$ in $\mu$, then the modified Macdonald polynomial of the filled diagram $(\mu,f^{\st}_\mu)$ is the usual modified Macdonald polynomial $\widetilde{H}_\mu$ by the celebrated Haglund--Haiman--Loehr formula \cite{HHL05}:
\[
    \widetilde{H}_{(\mu,f^{\st}_\mu)}[X] = \widetilde{H}_\mu[X;q,t].
\]

\subsection{Column exchange rule and cycling}

There can be two distinct filled diagrams $(D,f)$ and $(D',f')$ with the same modified Macdonald polynomials. In this subsection, we review such examples studied in \cite{KLO22}. The first one is the \emph{column exchange rule}. We first recall some terminology.

\begin{definition}\label{def: the operator S}
    For positive integers $n>m$,  We define $\mathcal{V}(n,m)$ to be the set of all filled diagrams $(\mu,f_{\mu})$ such that   $\mu=[[n],[m]]$ and a filling $f_\mu$ on it satisfies the following condition:
    \begin{equation}\label{eq: column exchange condition}
   f_\mu(i,1) = q^{-1}f_\mu(m+1,1) f_\mu(i,2), \quad \text{for  $1<i\leq m$}.  
    \end{equation}
     For $(\mu,f_{\mu})\in \mathcal{V}(n,m)$, we let $S(\mu,f_{\mu})$ to be the filled diagram $(\lambda,f_{\lambda})$ such that $\lambda=[[m],[n]]$ and:
    \begin{align*}
        &f_\lambda(i,2)=f_\mu(i,1), \hspace{30mm}\text{ for  $i>m+1$}\\
        &f_\lambda(m+1,2)=q^{-1}f_\mu(m+1,1)\\
        &f_\lambda(m+1,2) f_\lambda(i,1) = f_\lambda(i,2)=f_\mu(i,1) , \text{ for  $1<i\leq m$}.  \nonumber
    \end{align*}
See Figure~\ref{fig: generic filled diagrams column exchange lemma} for the example.

\begin{figure}[h]
\[
    \ytableausetup{boxsize=3em,aligntableaux=center}
    (\mu,f_{\mu}) = \begin{ytableau}
     b_{1} \\
     \vdots  \\
     b_{n-m-1} \\
    q\alpha  \\
    \alpha a_1 & a_1 \\
    \vdots & \vdots \\
    \alpha  a_{m-1} & a_{m-1} \\
     & 
    \end{ytableau}
    \qquad \qquad S(\mu,f_{\mu} ) = \begin{ytableau}
     \none &b_{1} \\
     \none &\vdots  \\
     \none &b_{n-m-1} \\
    \none & \alpha  \\
    a_1 & \alpha a_1 \\
    \vdots & \vdots \\
     a_{m-1} & \alpha a_{m-1} \\
     & 
    \end{ytableau}
\]
\caption{The filled diagram $(\mu,f_{\mu})\in\mathcal{V}(n,m)$ is of the form 
 in the left figure for some $a_i$'s, $b_i$'s and $\alpha$ in $\mathbb{F}$. The right figure shows the corresponding $S(\mu,f_{\mu})$.}\label{fig: generic filled diagrams column exchange lemma}
\end{figure}

Let $(D,f)$ be a filled diagram such that the restriction to the $j$ and $j+1$-th columns is in $\mathcal{V}(n,m)$ for some $n>m$. We define $S_j(D,f)$ to be the filled diagram obtained from $(D,f)$ by applying the map $S$ to the two columns ($j$ and $j+1$-th columns). Clearly, the map $S_j$ is injective, so we may consider $S_j^{-1}(D,f)$ if exists. Given two filled diagrams $(D,f)$ and $(D',f')$, we write $(D,f)\cong (D',f')$ if they are related by a sequence of operators of the form $S_j$ or $S_j^{-1}$.

\end{definition}

\begin{proposition}\cite[Proposition 4.3, rephrased]{KLO22}\label{lem: column exchange rephrased} Let $(D,f)$ and $(D',f')$ be filled diagrams of size $n$ such that $(D,f)\cong (D',f')$. We let $P_{i}$ be the set of $N_{D}(u)$ for the cells $u$ on the $i$-th row of $D$. Then there is a bijection $\phi:\mathfrak{S}_n\rightarrow\mathfrak{S}_n$ satisfying the following three conditions:
\begin{align*}
    (\phi1)\quad& \stat_{(D,f)}(w) = \stat_{(D',f')}(\phi(w)), \\
    (\phi2)\quad& \iDes(w) = \iDes(\phi(w)), \text{ and}\\
    (\phi3)\quad& \{w_{j}: j \in P_i\} = \{\phi(w)_{j}: j \in P_i\} \text{ for all $i$.}
\end{align*}
\end{proposition}
With the assumptions in Proposition \ref{lem: column exchange rephrased}, conditions $(\phi1)$ and $(\phi2)$ are enough to imply $\widetilde{H}_{(D,f)}[X]=\widetilde{H}_{(D',f')}[X]$. However, we need a more refined identity \eqref{eq: refined identity} coming from the condition $(\phi3)$ that will be used in the proof of Lemma \ref{lem: focus on the bottom k-1 rows} and Lemma \ref{lem: symmetric function!}. We define some terminologies to state \eqref{eq: refined identity}.

\begin{definition}
For diagram $D$ and a composition $\alpha=(\alpha_1,\dots,\alpha_{r})$ of size less than or equal to $|D|$, we define $\mathrm{OP}(D;\alpha)$ to be the set of tuples $(A_1,\dots,A_{r})$ of disjoint subsets $A_i\subseteq D$ such that $|A_i|=\alpha_i$. For a tuple $(A_1, \dots, A_{r})\in \mathrm{OP}(D;\alpha)$, let $B_i=\{N_D(u): u\in A_i\}$ where $N_D$ is the total order on cells of $D$ defined in Section~\ref{subsec: generalization of Macdonald}. We define $\sigma(D;(A_1, \dots, A_{r}))$ to be the permutation obtained by first listing elements $[|D|]\setminus\cup_{i=1}^{r}B_i$ in decreasing order, followed by elements in $B_1$ in decreasing order, then elements in $B_2$ in decreasing order, and so on.

For instance, let $D = [\{1,2,3,4,5\},\{1,2,3,4\},\{2,3\}]$ be a diagram and 
\[
    (A_1,A_2,A_3)=(\{(2,2),(3,1),(3,3) \},\{(2,1),(3,2)\}, \{ (2,3),(4,2),(5,1)\})\in \mathrm{OP}(D;(3,2,3)).
\]
Pictorially, $(A_1,A_2,A_3)$ is displayed by labeling cells in $A_i$ by $i$ as on the left in Figure~\ref{fig: example of P(D,alpha)}.
\begin{figure}
    \[
    \begin{ytableau}
    3 \\
      & 3 \\
    1 & 2 & 1 \\
    2 & 1 & 3\\
    &
    \end{ytableau} \qquad
    \begin{ytableau}
    1 \\
    2 & 3 \\
    4 & 5 & 6 \\
    7 & 8 & 9\\
    10 & 11
    \end{ytableau}
    \]
\caption{An element of $\mathrm{OP}(D,(3,2,3))$ and a natural ordering $N_{D}$ of $D$.}
\label{fig: example of P(D,alpha)}
\end{figure}
According to the total order $N_D$ (given on the right in Figure~\ref{fig: example of P(D,alpha)}), the permutation associated to $(A_1,A_2,A_3)$ is
\begin{equation*}
    \sigma(D;(A_1,A_2,A_3))=(11,10,2,\underbrace{8,6,4}_{3},\underbrace{7,5}_{2},\underbrace{9,3,1}_{3}).
\end{equation*}
\end{definition}

By \eqref{eq: easy inference}, letting $m=|D|-\sum_{i=1}^{r}\alpha_i$, we have 
\begin{equation}\label{eq: 11}\langle\widetilde{H}_{(D,f)}[X]),e_{(m,\alpha)}\rangle=\sum_{L\in \mathrm{OP}(D;\alpha)}\stat_{(D,f)}(\sigma(D;L)^{-1}).
\end{equation}

Now we further refine the set $\mathrm{OP}(D;\alpha)$.

\begin{definition}
Let $D$ be a diagram with $\ell$ rows. For $L=(A_1,\dots,A_r)\in \mathrm{OP}(D;\alpha)$, we define an $r\times\ell$ matrix $\type(D;L)$ by letting its entries be defined as follows:
\begin{align*}
\type(D;L)_{i,j}=\text{the number of cells in the $j$-th row of $D$ that are in $A_i$}.
\end{align*}
Now we define $\mathbf{OP}(D;M)$ to be the set of tuples $L$ such that $\type(D;L)=M$. Note that $\csum(\type(D;L))=\alpha$, if $L\in \mathrm{OP}(D;\alpha)$. With this notion, we partition $\mathrm{OP}(D;\alpha)$ into the following union
\[
    \mathrm{OP}(D;\alpha)=\bigcup_{\csum(M)=\alpha}\mathbf{OP}(D;M).
\]
\end{definition}

 Given $(D,f)\cong(D',f')$, let $\phi$ be the map in Proposition \ref{lem: column exchange rephrased}. By the conditions $(\phi1)$ and $(\phi2)$, $\phi$ gives a stat-preserving bijection between 
\[
    \{\sigma(D;L)^{-1}:L\in \mathrm{OP}(D;\alpha)\} \qquad \text{   and   } \qquad \{\sigma(D';L)^{-1}:L\in \mathrm{OP}(D';\alpha)\}.
\]
Moreover, the condition $(\phi3)$ implies that if $\phi(\sigma(D;L)^{-1})=\sigma(D';L')^{-1}$, we have $\type(D;L)=\type(D';L')$. In other words, $\phi$ gives a bijection between 
\[
    \{\sigma(D;L)^{-1}:L\in\mathbf{OP}(D;M)\} \qquad\text{   and   }\qquad \{\sigma(D';L)^{-1}:L\in\mathbf{OP}(D';M)\}.
\]
We conclude that for any $M\in\Mat^{\geq0}$
\begin{equation}\label{eq: refined identity}
\sum_{L\in\mathbf{OP}(D;M)}\stat_{(D,f)}(\sigma(D;L)^{-1})=\sum_{L\in\mathbf{OP}(D';M)}\stat_{(D',f')}(\sigma(D';L)^{-1}).
\end{equation}

We end this section by recalling the \emph{cycling rule}. Given a diagram $D=[D^{(1)},\dots,D^{(\ell)}]$ consider a diagram
\[
    D'=[D^{(2)},\dots,D^{(\ell)},D^{(1)}+1],
\]
where $I+1=\{a+1:a\in I\}$ for an interval $I$. In other words, $D'$ is a diagram obtained by moving the leftmost column of $D$ to the end of the right and placing it one cell up. Cells in $D$ and $D'$ are naturally in bijection. For a filling $f$ of $D$ we can associate a filling $f'$ of $D'$ inherited from $f$ via this natural bijection. Then we define $\cycling (D,f)\coloneqq(D',f')$. Then the following lemma is well-known and its proof is straightforward (See \cite[Lemma 3.6]{KLO22} for example)

\begin{lem}\label{lem: cycling}(Cycling rule)
Let $(D,f)$ be a filled diagram and $(D',f')=\cycling(D,f)$. Then we have
\[
    \widetilde{H}_{(D,f)}[X] = \widetilde{H}_{(D',f')}[X].
\]
\end{lem}

\section{Lightning bolt formula for Macdonald intersection polynomials}\label{Sec: proof of main theorem}\label{Sec: Proof of main theorem}
In this section, we present the \emph{lightning bolt formula} for Macdonald intersection polynomials, which plays a key role in the proof of Theorem~\ref{thm: main theorem} (c). Thanks to the shape independence stated in Theorem~\ref{thm: main theorem}~(b), it is sufficient to compute $e^{\perp}_{n+1-k}\I_{\mu^{(1)},\dots,\mu^{(k)}}$ for any tuple of distinct partitions $(\mu^{(1)},\dots,\mu^{(k)})$ where there exists a partition $\mu \vdash n+1$ satisfying $|\mu\backslash \mu^{(i)}|=1$. In this section, we fix $\mu$ as the augmented staircase $\overline{\delta}_k\coloneqq(\underbrace{k,\dots,k}_{k},k-1,\dots,1)$ and $\mu^{(i)}=\overline{\delta}_k\backslash c_i$, where $c_1,\dots,c_k$ represent removable corners of $\overline{\delta}_k$ indexed from top to bottom.

Here, we outline the contents of this section. In Section~\ref{subsec: deforming the filled diagram}, we apply the column exchange rule to $(\mu^{(i)},f^{\st}_{\mu^{(i)}})$ to obtain a filled diagram $\left(\overline{\delta}_{k,i}, \overline{f}_{k,i}\right)$.
In Section~\ref{subsec: lightning bolt formula for Macdonald intersection polynomials}, we present the lightning bolt formula for the Macdonald intersection polynomials (Proposition~\ref{prop: main key sf to shuffle}), which serves as a key component in proving Theorem~\ref{thm: main theorem} (c).
In Section~\ref{subsec: Proof of Proposition main key to SF to shuffle}, we demonstrate the reduction of Proposition~\ref{prop: main key sf to shuffle} to Proposition~\ref{prop: toward lb tilde}. Section~\ref{subsec: proof of Proposition 6.18} provides a proof of Proposition~\ref{prop: toward lb tilde}, which finalizes the proof of Theorem~\ref{thm: main theorem} (c).

\subsection{Deforming the filled diagram $(\mu^{(i)},f^{\st}_{\mu^{(i)}})$}\label{subsec: deforming the filled diagram}
In this subsection, we introduce a filled diagram $(\overline{\delta}_{k,i},\overline{f}_{k,i})$ that possesses the property 
\[
\widetilde{H}_{(\mu^{(i)},f^{\st}_{\mu^{(i)}})}[X]=\widetilde{H}_{(\overline{\delta}_{k,i},\overline{f}_{k,i})}[X].
\]
For the computation of $\I_{\mu^{(1)},\dots,\mu^{(k)}}$, we will utilize $(\overline{\delta}_{k,i},\overline{f}_{k,i})$ in place of $(\mu^{(i)},f^{\st}_{\mu^{(i)}})$.

\begin{definition}\label{def: deforming the standard filled diagram}
We define
\begin{equation*}
(\overline{\delta}_{k,i},\overline{f}_{k,i})\coloneqq\cycling\left(S_1\dots S_{i-1}(\mu^{(i)},f^{\st}_{\mu^{(i)}})\right).
\end{equation*}
Refer to Figure \ref{fig: delta k i process} for an illustrative example. In other words, we perform the following procedure: we shift the $i$-th column of $(\mu^{(i)},f^{\st}_{\mu^{(i)}})$ all the way to the left by applying $S_1\cdots S_{i-1}$, and then we take the first column (previously the $i$-th column) and move it to the far right using the operator $\cycling$. It can be verified that the local condition \eqref{eq: column exchange condition} is always satisfied during the application of the operators $S_1\cdots S_{i-1}$. For a detailed proof, we refer to \cite[Lemma 5.2]{KLO22}.

\end{definition}

\begin{figure}[h] 
\[
    \begin{ytableau}
    t \\
    q^{-1}t^2 & t \\
    q^{-1}t^3 & t^2 \\
    q^{-3}t^4 & q^{-2}t^3 & q^{-1}t & t\\
    q^{-3}t^5 & q^{-2}t^4 & q^{-1}t^2 & t^2\\
    q^{-3}t^6 & q^{-2}t^5 & q^{-1}t^3 & t^3\\
    &&&
    \end{ytableau}
    \qquad 
    \begin{ytableau}
    \none & t  \\
    \none & q^{-1}t^2 & t \\
    \none & q^{-2}t^3 & q^{-1}t^2 \\
    q^{-1}t & q^{-3}t^4 & q^{-2}t^3 & t\\
    q^{-1}t^2 & q^{-3}t^5 & q^{-2}t^4 & t^2\\
    q^{-1}t^3 & q^{-3}t^6 & q^{-2}t^5 & t^3\\
    &&&
    \end{ytableau}
    \qquad
    \begin{ytableau}
    t  \\
    q^{-1}t^2 & t \\
    q^{-2}t^3 & q^{-1}t^2 & \none & q^{-1}t\\
    q^{-3}t^4 & q^{-2}t^3 & t & q^{-1}t^2\\
    q^{-3}t^5 & q^{-2}t^4 & t^2 & q^{-1}t^3\\
    q^{-3}t^6 & q^{-2}t^5 & t^3 & \\
    &&
    \end{ytableau}
\]
\caption{The figure on the left depicts the filled diagram $\left(\mu^{(3)},f^{\st}_{\mu^{(3)}}\right)$, while the one in the center represents the filled diagram $S_1 S_2\left(\mu^{(3)},f^{\st}_{\mu^{(3)}}\right)$. Finally, the figure on the right illustrates the filled diagram $\left(\overline{\delta}_{4,3}, \overline{f}_{4,3}\right)$, which is obtained by applying the operator $\cycling$ to the figure on the center.}
\label{fig: delta k i process}
\end{figure}
\begin{lem}\label{lem: deformed partitions filling description} We have $\overline{\delta}_{k,i}=[I_1,I_2,\dots,I_k]$, where $I_j$ is the interval given by
\[
    I_j = \begin{cases}
        [2k-j] &\text{if } j<i \\
        [2k-j-1] &\text{if } i\le j<k \\
        [2,2k-i] &\text{if } j=k.
    \end{cases}
\]
The filling $\overline{f}_{k,i}$ on $\overline{\delta}_{k,i}$ is given by
\[
    \overline{f}_{k,i}(a,b) =
    \begin{cases}
        q^{b-k}t^{2k-a-b+1}&\text{if } a\le k, b<i\\
        q^{b-k+1}t^{2k-a-b}&\text{if } a\le k, i\le b< k\\
        q^{i-k} t^{2k-i-a+1}&\text{if } a\le k, b=k\\
        q^{a+b-2k}t^{2k-a-b+1}&\text{if } a>k, b<i\\
        q^{a+b-2k+1}t^{2k-a-b}&\text{if } a>k,i\le b <k\\
        q^{i+a-2k-1}t^{2k-i-a+1}&\text{if } a>k, b=k.
    \end{cases}
\]
\end{lem}
\begin{proof}
    The standard filling of $\mu^{(i)}$ is given by
    \[
        f^{\st}_{\mu^{(i)}}(a,b)=
        \begin{cases}
            q^{b-k}t^{2k-a-b+1}&\text{if } a\le k, b\neq i\\
            q^{i-k}t^{2k-a-i}&\text{if } a\le k, b= i\\
            q^{a+b-2k}t^{2k-a-b+1}&\text{if } k<a<2k-i, b\neq i\text{ or }a>2k-i\\
            q^{a+i-2k}t^{2k-a-i}&\text{if } a>k, b=i\\ 
            q^{b-i+1}t^{i-b+1}&\text{if } a=2k-i.
        \end{cases}
    \]
    While applying the operator $S_1\dots S_{i-1}$ the only changes of the filling occur in the $(2k-i)$-th row by multiplying $q^{-1}$. Denoting $(D,f)=S_1\dots S_{i-1}\left(\mu^{(i)},f^{\st}_{\mu^{(i)}}\right)$, we have
    \[
    D = [[2k-i-1],[2k-1],\dots,[2k-i+1],[2k-i-1],\dots,[k]],
    \]
    and
    \[
        f(a,b)=
        \begin{cases}
            q^{i-k} t^{2k-a-i}&\text{if } a\le k, b=1\\
            q^{b-k-1}t^{2k-a-b+2}&\text{if } a\le k, 1<b\le i\\
            q^{b-k}t^{2k-a-b+1}&\text{if } a\le k, b>i\\
            q^{a+i-2k}t^{2k-a-i}&\text{if } a>k, b=1\\
            q^{a+b-2k-1}t^{2k-a-b+2}&\text{if } a>k, 1<b\le i\\
            q^{a+b-2k}t^{2k-a-b+1}&\text{if } a>k, b>i.
        \end{cases}
    \]
Applying the operator $\cycling$ completes the proof.
\end{proof}

\begin{rmk}
    Let $\lambda$ be a partition and $(D,f)$ be a filled diagram satisfying $(D,f)\cong(\lambda,f_{\lambda}^{\st})$. Then \cite[Theorem 5.1.1]{HHL08} implies 
    $\widetilde{H}_{(D,f)}[X]=\widetilde{H}_{(\lambda,f^{\st}_{\lambda})}[X]$. In particular, letting  $(D,f)=S_{1}\dots S_{i-1}(\mu^{(i)},f^{\st}_{\mu^{(i)}})$ we obtain 
    \[
    \widetilde{H}_{(D,f)}[X]=\widetilde{H}_{(\mu^{(i)},f^{\st}_{\mu^{(i)}})}[X]
    \]
    without using Proposition \ref{lem: column exchange rephrased}. However \cite[Theorem 5.1.1]{HHL08} is not enough for our purpose later as we deal with $z$-deformed filled diagrams in Section~\ref{subsec: Proof of Proposition main key to SF to shuffle}. Moreover, we need a more refined identity \eqref{eq: refined identity}.
\end{rmk}

\subsection{Lightning bolt formula for Macdonald intersection polynomials}\label{subsec: lightning bolt formula for Macdonald intersection polynomials}
In this subsection, we state the lightning bolt formula for Macdonald intersection polynomials (Proposition~\ref{prop: main key sf to shuffle}) which essentially proves Theorem~\ref{thm: main theorem} (c).

By \eqref{eq: 11}, for a composition $\beta\models k-1$ we have
\begin{align}\label{eq: Sf coefficient equ}
\langle e^{\perp}_{n+1-k}\I_{\mu^{(1)},\dots,\mu^{(k)}}[X;q,t],e_{\beta}\rangle&=\sum_{i=1}^{k} \left(\prod_{j\neq i}\dfrac{T_{\mu^{(j)}}}{T_{\mu^{(j)}}-T_{\mu^{(i)}}}\right)\sum_{L\in \mathrm{OP}(\overline{\delta}_{k,i};\beta)}\stat_{(\overline{\delta}_{k,i},\overline{f}_{k,i})}\left(\sigma(\overline{\delta}_{k,i};L)^{-1}\right)\\ &=\sum_{\substack{M\in \Mat^{\ge0}_{r\times(2k-1)} \\ \csum(M)=\beta}}\sum_{i=1}^{k} \left(\prod_{j\neq i}\dfrac{T_{\mu^{(j)}}}{T_{\mu^{(j)}}-T_{\mu^{(i)}}}\right)\sum_{L\in \mathbf{OP}(\overline{\delta}_{k,i};M)}\stat_{(\overline{\delta}_{k,i},\overline{f}_{k,i})}\left(\sigma(\overline{\delta}_{k,i};L)^{-1}\right)  \nonumber
\end{align}

Now we present the remarkable observation, which we refer to as the "lightning bolt formula" for Macdonald intersection polynomials. This formula establishes a connection between the Macdonald intersection polynomial and the fermionic formula for the shuffle formula (Theorem~\ref{thm: fermionic formula}). The subsequent subsections are devoted to a proof of Proposition \ref{prop: main key sf to shuffle}.
\begin{proposition}\label{prop: main key sf to shuffle}
For $M\in \Mat^{\ge0}_{r\times(2k-1)} $ such that $|M|=k-1$, denoting $\alpha=\rsum(M)$, we have 
\begin{align*}    
\sum_{i=1}^{k} \left(\prod_{j\neq i}\dfrac{T_{\mu^{(j)}}}{T_{\mu^{(j)}}-T_{\mu^{(i)}}}\right)\sum_{L\in \mathbf{OP}(\overline{\delta}_{k,i};M)}\stat_{(\overline{\delta}_{k,i},\overline{f}_{k,i})}\left(\sigma(\overline{\delta}_{k,i};L)^{-1}\right)=T_{\bigcap_{i=1}^k \mu^{(i)}}\left(t^{\sum_{i=1}^{\ell(\alpha)}(i-1)\alpha_i}q^{\sum_{i,j} \binom{M_{i,j}}{2}}\right) \LB(M).
\end{align*}

\end{proposition}
We end this subsection with a proof of Theorem~\ref{thm: main theorem} (c) assuming Proposition~\ref{prop: main key sf to shuffle}.
\begin{proof}[Proof of Theorem~\ref{thm: main theorem} (c)]
For any $\beta\models k-1$, by Proposition \ref{prop: main key sf to shuffle} and \eqref{eq: Sf coefficient equ}, we have 
\begin{equation*}
    \langle\dfrac{1}{T_{\bigcap_{i=1}^k \mu^{(i)}}}e_{n+1-k}^\perp\left(\I_{\mu^{(1)},\dots,\mu^{(k)}}[X;q,t]\right),e_{\beta}\rangle=\sum_{\alpha }\sum_{\substack{M\in \Mat^{\ge0}_{\ell(\beta)\times(2k-1)}\\ \rsum(M)=\alpha \\ \csum(M)=\beta}}\left(t^{\sum_{i=1}^{\ell(\alpha)}(i-1)\alpha_i}q^{\sum_{i,j} \binom{M_{i,j}}{2}}\right) \LB(M),
\end{equation*}
where the sum on the right-hand side is over vectors $\alpha$ of length $2k-1$ whose entries are nonnegative integers.
By applying Lemma~\ref{lem: basic property of lbm}, we can deduce that if $\rsum(M)$ is not in the form of $(\alpha',0,\dots,0)$ for a composition $\alpha'\models k-1$, then $\LB(M)=0$. Consequently, our focus narrows down to matrices where $\rsum(M)$ takes the form $(\alpha',0,\dots,0)$ for some $\alpha'\models k-1$. Using Theorem~\ref{thm: fermionic formula}, we can finally conclude that
\begin{equation*}
\langle\dfrac{1}{T_{\bigcap_{i=1}^k \mu^{(i)}}}e_{n+1-k}^\perp\left(\I_{\mu^{(1)},\dots,\mu^{(k)}}[X;q,t]\right),e_{\beta}\rangle=\langle D_{k-1}[X;q,t],e_{\beta}\rangle
\end{equation*}
for any $\beta\models k-1$, which is enough to deduce
\begin{equation*}
\dfrac{1}{T_{\bigcap_{i=1}^k \mu^{(i)}}}e_{n+1-k}^\perp\left(\I_{\mu^{(1)},\dots,\mu^{(k)}}[X;q,t]\right)=D_{k-1}[X;q,t].
\end{equation*}
\end{proof}

Our current objective is to prove Proposition~\ref{prop: main key sf to shuffle}. In Section~\ref{subsec: Proof of Proposition main key to SF to shuffle}, we will demonstrate that it is sufficient to establish Proposition~\ref{prop: toward lb tilde}, the proof of which will be provided in Section~\ref{subsec: proof of Proposition 6.18}.

\subsection{Reducing to the rectangular case}\label{subsec: Proof of Proposition main key to SF to shuffle}

We first introduce a $z$-deformation $(\overline{\delta}_{k,i},\overline{f}^{\mathbf{z}}_{k,i})$ of the filled diagram $(\overline{\delta}_{k,i},\overline{f}_{k,i})$, where the filling $\overline{f}_{k,i}^{\textbf{z}}$ on $\overline{\delta}_{k,i}$ is defined as follows: 
    \begin{equation}\label{Eq: definition of z-deformed f_k,i}
        \overline{f}_{k,i}^{\textbf{z}}(a,b) =
        \begin{cases}
            q z_{2k+1-a}/z_b &\text{if } b<i \\ 
            q z_{2k+1-a}/z_{b+1} &\text{if } i\le b<k\\
            q z_{2k+1-a}/z_i &\text{if } a\le k, b=k\\
            z_{2k+1-a}/z_i &\text{if } a>k, b=k.
        \end{cases}
    \end{equation}
Figure~\ref{Fig: z-deformed filled diagram} shows $(\overline{\delta}_{4,3},\overline{f}_{4,3}^\textbf{z})$ which is a $z$-deformation of $(\overline{\delta}_{4,3},\overline{f}_{4,3})$ shown on the right of Figure \ref{fig: delta k i process}.
\begin{figure}[h]
    \centering
    \begin{ytableau}
    q\dfrac{z_2}{z_1}  \\
    q\dfrac{z_3}{z_1} & q\dfrac{z_3}{z_2}\\
    q\dfrac{z_4}{z_1} & q\dfrac{z_4}{z_2} & \none & \dfrac{z_4}{z_3}\\
    q\dfrac{z_5}{z_1} & q\dfrac{z_5}{z_2} & q\dfrac{z_5}{z_4} & q\dfrac{z_5}{z_3}\\
    q\dfrac{z_6}{z_1} & q\dfrac{z_6}{z_2} & q\dfrac{z_6}{z_4} & q\dfrac{z_6}{z_3}\\
    q\dfrac{z_7}{z_1} & q\dfrac{z_7}{z_2} & q\dfrac{z_7}{z_4} & \\
    &&
    \end{ytableau}
    \caption{A $z$-deformed filled diagram $(\overline{\delta}_{4,3},\overline{f}_{4,3}^\textbf{z})$.}
    \label{Fig: z-deformed filled diagram}
\end{figure}

By Lemma~\ref{lem: deformed partitions filling description}, it is direct to see that substituting 
\begin{equation}\label{eq: z_i values}
    z_j = 
    \begin{cases}
        q^{1-j}t^{j-1}&\text{if } j\le k\\
        q^{-k}t^{j-1}&\text{if } j > k
    \end{cases}
\end{equation}
recovers $(\overline{\delta}_{k,i}\overline{f}_{k,i})$. We now introduce $z$-deformed version of Proposition~\ref{prop: main key sf to shuffle}. 

\begin{proposition}\label{prop: main key sf to shuffle z-deformed}
For $M\in \Mat^{\ge0}_{r\times(2k-1)} $ such that $|M|=k-1$, denoting $\alpha=\rsum(M)$, we have
\begin{align*}    
\sum_{i=1}^{k} \left(\prod_{j\neq i}\dfrac{z_j}{z_j-z_i}\right)\sum_{L\in \mathbf{OP}(\overline{\delta}_{k,i};M)}\stat_{(\overline{\delta}_{k,i},\overline{f}^{\mathbf{z}}_{k,i})}\left(\sigma(\overline{\delta}_{k,i};L)^{-1}\right)=q^{\frac{4k^3-6k^2-k}{3}+\sum_{i,j}\binom{M_{i,j}}{2}}\left(\frac{z_k\prod_{i=1}^{k-1}z_{2k-i}^{k-\alpha_i}}{z_{2k-1}\prod_{i=1}^{k-1}z_i^{2k-2i-1} }\right) \LB(M).
\end{align*}
\end{proposition}
On the right-hand side, we may only consider $\alpha$ such that $\alpha_\ell=0$ for $\ell\geq k$ (otherwise $\LB(M)=0$), and a simple computation shows
\begin{align*}
    \frac{T_{\mu^{(i)}}}{T_{\mu^{(j)}}}=\frac{z_i}{z_j},\qquad q^{\frac{4k^3-6k^2-k}{3}}\left(\frac{z_k\prod_{i=1}^{k-1}z_{2k-i}^{k-\alpha_i}}{z_{2k-1}\prod_{i=1}^{k-1}z_i^{2k-2i-1} }\right)=T_{\bigcap_{i=1}^k \mu^{(i)}}\left(t^{\sum_{i=1}^{k-1}(i-1)\alpha_i}\right)
\end{align*}
under the specialization \eqref{eq: z_i values}. Therefore Proposition \ref{prop: main key sf to shuffle} follows from Proposition \ref{prop: main key sf to shuffle z-deformed}. We will reduce Proposition~\ref{prop: main key sf to shuffle z-deformed} to Proposition~\ref{prop: toward lb tilde}  which analyzes the rectangular case. Here is the outline of the strategy.

\begin{enumerate}
    \item We first establish the equivalence between Proposition~\ref{prop: main key sf to shuffle z-deformed} and \eqref{eq: z deformed prop rephrased} by introducing a complementary statistic $\overline{\stat}$ (Definition~\ref{def: combinatorial description of stat}).
    \item  We can restrict our attention to the almost rectangular shape $\overline{R}_{k}$ in the first $k$ rows (shown on the left in Figure~\ref{Fig: filled diagrams R_k, R_k bar}), utilizing Lemma~\ref{lem: sf degree less k-1 vanishes} and Lemma~\ref{lem: focus on the bottom k-1 rows}. This reduction transforms \eqref{eq: z deformed prop rephrased} into \eqref{eq: z deformed prop rephrased2}.
    \item We establish a connection between the rectangular shape $R_k$ (on the right in Figure~\ref{Fig: filled diagrams R_k, R_k bar}) and $\overline{R_k}$ in Lemma~\ref{lem: bar Rk to Rk reduction}. 
    \item We present Proposition~\ref{prop: toward lb tilde} and demonstrate that when combined with Lemma~\ref{lem: bar Rk to Rk reduction} and Lemma~\ref{Lem: recurrence tilde LV and LV}, it completes the proof of Proposition~\ref{prop: main key sf to shuffle z-deformed}.

\end{enumerate}

\begin{definition}\label{def: combinatorial description of stat}
Given a filled diagram $(D,f)$, let $L=(A_1,\dots,A_r)\in \mathbf{OP}(D;M)$ for some $M\in\Mat^{\geq0}$. We set $A_0=D\setminus \bigcup_{i=1}^r A_i$, so that for every cell $u\in D$, there exists unique $\ell$ such that $u\in A_{\ell}$. We define $\overline{\Inv}_D(L)$ as the set of attacking pairs $(u,v)$ that do not belong to $\Inv_D(\sigma(D;L)^{-1})$. In other words, $\overline{\Inv}_D(L)$ consists of pairs $(u,v)$ of cells in $D$ such that $(u,v)$ is an attacking pair, and $u$ belongs to $A_{\ell}$ while $v$ belongs to $A_{\ell'}$ with $ \ell<\ell'$. We define $\overline{\inv}_D(L)$ as the product of $q$ over all pairs $(u,v)\in \overline{\Inv}_D(L)$.

Next, we define $\overline{\Des}_D(L)$ as the set of cells in $D$ that are not bottom cells and are not in $\Des_D(\sigma(D;L)^{-1})$. In other words, $\overline{\Des}_D(L)$ consists of cells $(i,j)$ in $D$ such that $(i,j)\in A_{\ell}$ and $(i-1,j)\in A_{\ell'}$ for $\ell<\ell'$, where $A_0$ is defined as before. We let $\overline{\maj}_{(D,f)}(L)$ denote the product of $f(u)$ over all cells $u\in \overline{\Des}_D(L)$. Finally, we define $\overline{\stat}_{(D,f)}$ as follows:
\[
    \overline{\stat}_{(D,f)}(L) \coloneqq \overline{\inv}_{D}(L)\overline{\maj}_{(D,f)}(L)
\]
By definition, we have
\begin{equation*}
\stat_{(D,f)}(\sigma(D;L)^{-1})=\frac{\stat_{(D,f)}(w_0)}{\overline{\stat}_{(D,f)}(L)}=\frac{\stat_{(D,f)}(w_0)}{\overline{\inv}_D(L)\overline{\maj}_{D,f}(L)},
\end{equation*}
where $w_0$ is the longest permutation in $\mathfrak{S}_{|D|}$.
\end{definition}

\begin{lem}\label{lem: longest permutation stat}
For each $1\le i\le k$,
    \begin{equation*}
    \stat_{(\overline{\delta}_{k,i},\overline{f}_{k,i}^{\textbf{z}})}(w_0)= q^{\frac{4k^3-3k^2-4k}{3}}\dfrac{z_i\prod_{j=1}^{k-1}z_{2k-j}^k}{z_{2k-1}\prod_{j=1}^{k-1}z_j^{2k-2j}}
    \end{equation*}
\end{lem}
for the longest permutation $w_0$.
\begin{proof}
For simplicity, let $(D,f)=(\overline{\delta}_{k,i},\overline{f}_{k,i})$. We claim that
\begin{equation}\label{Eq: maj_D(w_0)} 
    \maj_{(D,f)}(w_0) = q^{\frac{3k^2-5k}{2} + i-1}\dfrac{z_i\prod_{j=1}^{k-1}z_{2k-j}^k}{z_{2k-1}\prod_{j=1}^{k-1}z_j^{2k-2j}}
\end{equation}
and 
\begin{equation}\label{Eq: inv_D(w_0)}
    \inv_{D}(w_0) = q^{\frac{8k^3-15k^2+7k}{6}+1-i},
\end{equation}
which completes the proof.

The first claim \eqref{Eq: maj_D(w_0)} is direct from definition of $\overline{f}^{\mathbb{z}}_{k,i}$ in \eqref{Eq: definition of z-deformed f_k,i} and by the fact that $\maj_{(D,f)}(w_0)=\prod_{c\in D}\overline{f}^{\mathbb{z}}_{k,i}(c)$. To see \eqref{Eq: inv_D(w_0)}, note that for given cell $u=(a,b)$, the number of pair $(u,v)$ in attacking position in $D$ is
\[
    \begin{cases}
       k-1 &\text{ if } 1<a\le k, 1\le b \le k\\
       k-b-1 &\text{ if } a=1, 1\le b<k\\
       2k-a-1 &\text{ if } k<a\le 2k-i, 1\le b\le 2k-a-1 \text{ or } ,2k-i<a<2k, 1\le b\le 2k-a,\\
       2k-a &\text{ if } k<a\le2k-i, b=k.\\
    \end{cases}
\]
Therefore, the total number of such attacking pairs, i.e., the total number of inversion pairs of $w_0$, is 
\[
    k(k-1)^2 + \sum_{1\le b <k}(k-b-1) + \sum_{k<a\le2k-i}(2k-a-1)^2 + \sum_{2k-i<a<2k}(2k-a-1)(2k-a) + \sum_{k<a\le 2k-i}(2k-a)  =  \frac{8k^3-15k^2+7k}{6}+1-i.
\]
\end{proof}

By the above lemma, the term $(\prod_{j\neq i}z_j)\stat_{(\overline{\delta}_{k,i},\overline{f}_{k,i}^{\textbf{z}})}(w_0)$ is independent of $i$, thus the left-hand side of Proposition \ref{prop: main key sf to shuffle z-deformed} becomes 
\begin{align*}
 q^{\frac{4k^3-3k^2-4k}{3}}\dfrac{z_k\prod_{j=1}^{k-1}z_{2k-j}^k}{z_{2k-1}\left(\prod_{j=1}^{k-1}z_j^{2k-2j-1}\right) }\sum_{i=1}^{k} \left(\prod_{j\neq i}\dfrac{1}{z_j-z_i}\right)\sum_{L\in \mathbf{OP}(\overline{\delta}_{k,i};M)}\frac{1}{\overline{\stat}_{(\overline{\delta}_{k,i},\overline{f}^{\mathbf{z}}_{k,i})}(L)}.
\end{align*}

Therefore, after the cancellation, Proposition \ref{prop: main key sf to shuffle z-deformed} is equivalent to the following: For $M\in \Mat^{\ge0}_{r\times(2k-1)} $ such that $|M|=k-1$, denoting $\alpha=\rsum(M)$, we have
\begin{equation}\label{eq: z deformed prop rephrased}
    \sum_{i=1}^{k} \left(\prod_{j\neq i}\dfrac{1}{z_j-z_i}\right)\sum_{L\in \mathbf{OP}(\overline{\delta}_{k,i};M)}\frac{1}{\overline{\stat}_{(\overline{\delta}_{k,i},\overline{f}^{\mathbf{z}}_{k,i})}(L)}=\left(\frac{q^{\sum_{i,j}\binom{M_{i,j}}{2}}}{q^{k(k-1)}\prod_{i=1}^{k-1}z_{2k-i}^{\alpha_i}}\right)\LB(M).
\end{equation}

Now we analyze \eqref{eq: z deformed prop rephrased}. Before that, we define several terminologies and prove preparatory lemmas (Lemma~\ref{lem: sf degree less k-1 vanishes} and \ref{lem: focus on the bottom k-1 rows}). 

For a diagram $D$, note that $\mathbf{OP}(D;M)$ makes sense if $M$ has $\ell$ columns where $\ell$ is the number of rows in $D$. For a matrix $M$ with $\ell'$ columns such that $\ell'<\ell$, we consider $\overline{M}$ obtained by appending $(\ell-\ell')$ many zero columns to $M$. From now on, abusing the notation, we refer to $\mathbf{OP}(D;\overline{M})$ as $\mathbf{OP}(D;M)$.

Let $D$ be a diagram and $E$ be a subdiagram of $D$. Given $L=(A_1,\dots,A_r)\in \mathbf{OP}(D;M)$, we define $L \vert_E$ as $(A_1 \cap E,\dots,A_r \cap E)$.

\begin{definition}\label{def: R_k and R_k bar}
    We define the diagram $\overline{R}_k$ to be a subdiagram of $\overline{\delta}_{k,i}$ obtained by restricting to the first $k$ rows and the diagram $R_k$ to be the subdiagram of $\overline{R}_k$ obtained by restricting to the first $(k-1)$ columns. In other words, $\overline{R}_k=[I_1,I_2,\dots,I_k]$ and $R_k=[I_1,I_2,\dots,I_{k-1}]$, where $I_j$ is the interval given by
    \[
        I_j = \begin{cases}
            [k] & \text{if } j<k\\
            [2,k] & \text{if } j=k.
        \end{cases}
    \]
    We define $(\overline{R}_{k},\overline{g}^{\mathbf{z}}_{k,i})$ and $(R_{k},g_{k,i}^{\mathbf{z}})$ to be corresponding sub-filled diagrams of $(\overline{\delta}_{k,i},\overline{f}^{\mathbf{z}}_{k,i})$.

\end{definition}
To describe the filling $\overline{g}^{\mathbf{z}}_{k,i}$, it is convenient to introduce the following operator.
\begin{definition}\label{def: operator eta}
For a function $P(z_1,\dots,z_{k-1})$ of variables $z_1,\dots,z_{k-1}$, let $\eta_{k,i}$ be the operator that shifts the indices of variables by
\begin{equation*}
    \eta_{k,i}\left(P(z_1,\dots,z_{k-1})\right)\coloneqq P(z_1,\dots,z_{i-1},z_{i+1},\dots,z_{k}).
\end{equation*}
Using the operator $\eta_{k,i}$, we can write the filling $\overline{g}^{\mathbf{z}}_{k,i}$ of $\overline{R}_k$ by
\[
    \overline{g}^{\mathbf{z}}_{k,i}(a,b)= \begin{cases}
        \eta_{k,i}(q\frac{z_{2k+1-a}}{z_b}) &\text{if } b<k \\
        q\frac{z_{2k+1-a}}{z_i} &\text{if } b=k.
    \end{cases}
\]
\end{definition}
For example, recall the filled diagram $(\overline{\delta}_{4,3},\overline{f}_{4,3}^\textbf{z})$ in Figure~\ref{Fig: z-deformed filled diagram}. Taking the sub-filled diagram by restricting to the first 4 rows, we obtain the filled diagram $(\overline{R}_{4},\overline{g}^{\mathbf{z}}_{4,3})$ (on the left in Figure~\ref{Fig: filled diagrams R_k, R_k bar}). Again by restricting to the first 3 columns, we obtain the filled diagram $({R}_{4},g^{\mathbf{z}}_{4,3})$ (on the right in Figure~\ref{Fig: filled diagrams R_k, R_k bar}).
\begin{figure}[h]
    \centering
    \begin{ytableau}
    q\dfrac{z_5}{z_1} & q\dfrac{z_5}{z_2} & q\dfrac{z_5}{z_4} & q\dfrac{z_5}{z_3}\\
    q\dfrac{z_6}{z_1} & q\dfrac{z_6}{z_2} & q\dfrac{z_6}{z_4} & q\dfrac{z_6}{z_3}\\
    q\dfrac{z_7}{z_1} & q\dfrac{z_7}{z_2} & q\dfrac{z_7}{z_4} & \\
    &&
    \end{ytableau}
    \qquad 
    \begin{ytableau}
    q\dfrac{z_5}{z_1} & q\dfrac{z_5}{z_2} & q\dfrac{z_5}{z_4}\\
    q\dfrac{z_6}{z_1} & q\dfrac{z_6}{z_2} & q\dfrac{z_6}{z_4}\\
    q\dfrac{z_7}{z_1} & q\dfrac{z_7}{z_2} & q\dfrac{z_7}{z_4}\\
    &&
    \end{ytableau}
    \caption{The filled diagrams $(\overline{R}_{4},\overline{g}^{\mathbf{z}}_{4,3})$ and $({R}_{4},{g}_{4,3}^{\mathbf{z}})$.}
    \label{Fig: filled diagrams R_k, R_k bar}
\end{figure}

\begin{definition}\label{def: rk rkbar def}
Let $D$ be a diagram which is $\overline{R}_k$ or $R_k$ and consider $L=(A_1,\dots,A_r)\in \mathbf{OP}(D;M)$, where $ M\in\Mat^{\ge0}_{r\times(k-1)}$. We define a vector $\gamma(D;L)$ of length $k-1$ , whose $i$-th component is the number of cells $c\in\bigcup_{1\le j \le r}A_j$ that belong to the $i$-th row of $D$ and the cell $c'$ right above $c$ does not belong to $\overline{\Des}_D(L)$. In other words, $c\in A_j$ and $c'\in A_{j'}$ for $j \leq j'$. It is worth noting that there is no cell $c\in\bigcup_{1\le j \le r}A_j$ that is in the $k$-th row of $D$ as the matrix $M$ has $(k-1)$ columns. We say that $L$ is \emph{reduced} with respect to $D$ if $\gamma(D;L)=0$. The set of reduced elements in $\mathbf{OP}(D;M)$ is denoted by $\mathbf{OP}^{\mathrm{red}}(D;M)$. 

For instance, consider the diagram $D=\overline{R}_4$ and $L$ and $L'$ depicted  in Figure~\ref{Fig: A reduced tuple and a non-reduced tuple}. In this example, $L$ is reduced, while $L'$ is not reduced since $\gamma(D;L')=(1,1,0)$ (the cells contributing to $\gamma(D;L')$ are shaded in gray).
\end{definition}

The following two lemmas will be used to transform \eqref{eq: z deformed prop rephrased} to \eqref{eq: z deformed prop rephrased2}, which only concerns the reduced elements in $\mathbf{OP}(D;M)$.

\begin{figure}[h]
    \centering
    $L=$\begin{ytableau}
    $ $ &  &  &  \\
    $ $ & 2 & 1 & 2 \\
    $ $ & 3 &  & 4 \\
    1 &  & 2
    \end{ytableau}
    \qquad 
    $L'=\begin{ytableau}
    $ $ &  &  &  \\
    2 & {3} &  & 2 \\
    $ $ & *(gray!30)2 & 1 & \\
    1 &  & *(gray!30)1
    \end{ytableau}$
    \caption{On the left shows $L$ which is reduced, while $L'$ on the right is not reduced.}
    \label{Fig: A reduced tuple and a non-reduced tuple}
\end{figure}

\begin{lem}\label{lem: sf degree less k-1 vanishes}
Let $M\in\Mat^{\ge0}_{r\times(k-1)}$ and $\gamma$ be a vector of nonnegative integers of length $k-1$. Furthermore, we assume that $|M|-|\gamma|<k-1$. Then we have 
\begin{equation*}
    \sum_{i=1}^{k}\left(\prod_{j\neq i}\frac{1}{z_j-z_i}\right)\sum_{\substack{L\in \mathbf{OP}(\overline{R}_k;M)\\ \gamma(\overline{R}_k;L)=\gamma}}\frac{1}{\overline{\stat}_{(\overline{R}_k,\overline{g}^{\mathbf{z}}_{k,i})}(L)}=0.
\end{equation*}
\end{lem}
\begin{proof}
    Let $\alpha=\rsum(M)$ and $L\in \mathbf{OP}(\overline{R}_k;M)$ such that $\gamma(\overline{R}_k;L)=\gamma$.  For each $1\leq j\leq k-1$, there are exactly $(\alpha_j-\gamma_j)$ many cells in the $(j+1)$-th of $\overline{R}_k$ that are in $\overline{\Des}_{\overline{R}_k}(L)$. Note that if $c$ is the cell in the $(j+1)$-th row of $\overline{R}_k$, we have $\overline{g}^{\mathbf{z}}_{k,i}(c)=q\frac{z_{2k-j}}{z_{\ell}}$ for some $1\leq \ell\leq k$. Therefore we can write 
    \begin{equation}\label{Eq: stat bar = P/z}
        \sum_{\substack{L\in \mathbf{OP}(\overline{R}_k;M) \\ \gamma(\overline{R}_k;L)=\gamma}}\frac{1}{\overline{\stat}_{(\overline{R}_k,\overline{g}_{k,i})}(L)}=\frac{P_i(z_1,\dots,z_k)}{\prod_{j=1}^{k-1}z_{2k-j}^{\alpha_j-\gamma_j}}
    \end{equation}
    for some polynomial $P_i(z_1,\dots,z_k)$ of degree $\sum(\alpha_i-\gamma_i)=|M|-|\gamma|<k-1$. As we have $P_i=P_j$ when $z_i=z_j$ due to the description of fillings $\overline{g}^{\mathbf{z}}_{k,i}$ and $\overline{g}^{\mathbf{z}}_{k,j}$, Lemma~\ref{Lemma: degree less than k-1 vanishes} finishes the proof.
\end{proof}

\begin{lem}\label{lem: focus on the bottom k-1 rows}
    Let $M\in \Mat^{\ge0}_{r\times(2k-1)}$ and denote $\rsum(M)=\alpha$. If there exists $\ell\leq k-1$ such that 
    $\alpha_{\ell}=0$ and $\sum_{i=1}^{\ell-1}\alpha_i<k-1$, we have
  \begin{equation}\label{eq: lem 6.13 eq}
    \sum_{i=1}^{k}\left(\prod_{j\neq i}\frac{1}{z_j-z_i}\right)\sum_{\substack{L\in \mathbf{OP}(\overline{\delta}_{k,i};M)}}\frac{1}{\overline{\stat}_{(\overline{\delta}_{k,i},\overline{f}^{\mathbf{z}}_{k,i})}(L)}=0.
\end{equation}  
\end{lem}

\begin{proof} 
Let $(D^{(1)},f^{(1)}_i)$ be the sub-filled diagram of $(\overline{\delta}_{k,i},\overline{f}^{\mathbf{z}}_{k,i})$ obtained by restricting to the first $\ell$ rows and $(D^{(2)}_i,f^{(2)}_i)$ be the remaining sub-filled diagram from the $(\ell+1)$-th row and above. Note that the shape $D^{(1)}$ is independent of $i$. For $L=(A_1,\dots,A_r)\in \mathbf{OP}(\overline{\delta}_{k,i};M)$, there is no cell in the $\ell$-th row that is in $\cup_{t=1}^{r}A_t$ as $\alpha_{\ell}=0$. Therefore we have 
\begin{align*}
    \overline{\Inv}_{\overline{\delta}_{k,i}}(L)=\overline{\Inv}_{D^{(1)}}(L\vert_{D^{(1)}}) \cup \overline{\Inv}_{D^{(2)}_i}(L\vert_{D^{(2)}_i}) \qquad \overline{\Des}_{\overline{\delta}_{k,i}}(L)=\overline{\Des}_{D^{(1)}}(L\vert_{D^{(1)}}) \cup \overline{\Des}_{D^{(2)}_i}(L\vert_{D^{(2)}_i}),
\end{align*}
which gives $\overline{\stat}_{(\overline{\delta}_{k,i},\overline{f}^{\mathbf{z}}_{k,i})}(L)=\overline{\stat}_{(D^{(1)},f^{(1)}_i)}(L\vert_{D^{(1)}})\overline{\stat}_{(D^{(2)}_i,f^{(2)}_i)}(L\vert_{D^{(2)}_i})$.
Note that we have $\type(D^{(1)};L\vert_{D^{(1)}})=M^{(1)}$, and $\type(D^{(2)}_i;L\vert_{D^{(2)}_i})=M^{(2)}$. Here, $M^{(1)}$ represents the sub-matrix of $M$ obtained by restricting it to the first $\ell$ columns. Similarly, $M^{(2)}$ is the sub-matrix of $M$ obtained by restricting it from the $(\ell+1)$-th column to the $(2k-1)$-th column. We conclude
\begin{align*}    \sum_{L\in \mathbf{OP}(\overline{\delta}_{k,i};M)}\frac{1}{\overline{\stat}_{(\overline{\delta}_{k,i},\overline{f}^{\mathbf{z}}_{k,i})}(L)}
&=\sum_{L^{(1)}\in \mathbf{OP}(D^{(1)};M^{(1)})}\frac{1}{ \overline{\stat}_{(D^{(1)},f^{(1)}_i)}(L^{(1)})}\sum_{L^{(2)}\in \mathbf{OP}(D^{(2)}_i;M^{(2)})}\frac{1}{ \overline{\stat}_{(D^{(2)}_i,f^{(2)}_i)}(L^{(2)})}.
\end{align*}

For $1\leq i,j\leq k$, we have 
\begin{equation*}
    S_{i}^{-1}\dots S_{k-1}^{-1}(D_{i}^{(2)},f_{i}^{(2)})=S_{j}^{-1}\dots S_{k-1}^{-1}(D_{j}^{(2)},f_{j}^{(2)})
\end{equation*}
which gives $(D_{i}^{(2)},f_{i}^{(2)})\cong (D_{j}^{(2)},f_{j}^{(2)})$. From the identity $\widetilde{H}_{(D_{i}^{(2)},f_{i}^{(2)})}[X]=\widetilde{H}_{(D_{j}^{(2)},f_{j}^{(2)})}[X]$, taking the coefficient of $F_{(1,\dots,1)}$ gives
\begin{equation}\label{eq: temp1}
    \stat_{(D_{i}^{(2)},f_{i}^{(2)})}(w_0)=\stat_{(D_{j}^{(2)},f_{j}^{(2)})}(w_0)
\end{equation}
where $w_0$ is the longest permutation of the size $|D^{(2)}_i|$. We also have 
\begin{equation}\label{eq: temp2}
    \sum_{L\in\mathbf{OP}(D^{(2)}_i;M^{(2)})}\stat_{(D_{i}^{(2)},f_{i}^{(2)})}(\sigma(D^{(2)}_i;L)^{-1})=\sum_{L\in\mathbf{OP}(D^{(2)}_j;M^{(2)})}\stat_{(D_{j}^{(2)},f_{j}^{(2)})}(\sigma(D^{(2)}_j;L)^{-1})
\end{equation}
by \eqref{eq: refined identity}. 


By \eqref{eq: temp1} and \eqref{eq: temp2}, we deduce the term $\sum_{L^{(2)}\in \mathbf{OP}(D^{(2)}_i;M^{(2)})}\frac{1}{ \overline{\stat}_{(D^{(2)}_i,f^{(2)}_i)}(L^{(2)})}$ is independent of $i$, i.e. for $1\le i \le k$, we have
\[
    \sum_{L^{(2)}\in \mathbf{OP}(D^{(2)}_1;M^{(2)})}\frac{1}{ \overline{\stat}_{(D^{(2)}_1,f^{(2)}_1)}(L^{(2)})} = \sum_{L^{(2)}\in \mathbf{OP}(D^{(2)}_i;M^{(2)})}\frac{1}{ \overline{\stat}_{(D^{(2)}_i,f^{(2)}_i)}(L^{(2)})}.
\]
Therefore, we conclude that the left-hand side of \eqref{eq: lem 6.13 eq} is equal to 
\begin{align*}    
\sum_{L^{(2)}\in \mathbf{OP}(D^{(2)}_1;M^{(2)})}\frac{1}{ \overline{\stat}_{(D^{(2)}_1,f^{(2)}_1)}(L^{(2)})}\sum_{i=1}^{k}\left(\prod_{j\neq i}\frac{1}{z_j-z_i}\right)\sum_{L^{(1)}\in \mathbf{OP}(D^{(1)};M^{(1)})}\frac{1}{ \overline{\stat}_{(D^{(1)},f^{(1)}_i)}(L^{(1)})}.
\end{align*}
Finally, we have
\[
    \sum_{i=1}^{k}\left(\prod_{j\neq i}\frac{1}{z_j-z_i}\right)\sum_{L^{(1)}\in \mathbf{OP}(D^{(1)};M^{(1)})}\frac{1}{ \overline{\stat}_{(D^{(1)},f^{(1)}_i)}(L^{(1)})}= \sum_{i=1}^{k}\left(\prod_{j\neq i}\frac{1}{z_j-z_i}\right)\sum_{L^{(1)}\in \mathbf{OP}(\overline{R}_k;M^{(1)})}\frac{1}{ \overline{\stat}_{(\overline{R}_k,\overline{g}_{k,i})}(L^{(1)})} = 0,
\]
where first equality is obvious, and the last equality follows from Lemma~\ref{lem: sf degree less k-1 vanishes} since $|M^{(1)}|=\sum_{i=1}^{\ell}\alpha_i<k-1$.
\end{proof}

For $M\in \Mat^{\ge0}_{r\times(2k-1)}$ such that $|M|=k-1$, let $\alpha=\rsum(M)$. If there exists $j>k-1$ such that $\alpha_j>0$, there is $\ell\leq k-1$ with $\alpha_{\ell}=0$ since $\sum_{i=1}^{2k-1}\alpha_i=k-1$. In that case the left-hand side of \eqref{eq: z deformed prop rephrased} vanishes by Lemma~\ref{lem: focus on the bottom k-1 rows} and the right-hand side vanishes by  Lemma~\ref{lem: basic property of lbm}. Therefore we may only consider $M$ such that the $j$-th column is a zero column for all $j>k-1$. For such $M$,  we trivially have 
\begin{equation*}
\overline{\stat}_{(\overline{\delta}_{k,i},\overline{f}_{k,i}^{\mathbf{z}})}(L)=\overline{\stat}_{(\overline{R}_{k},\overline{g}_{k,i}^{\mathbf{z}})}(L\vert_{\overline{R}_{k}})
\end{equation*}
for $L\in \mathbf{OP}(\overline{\delta}_{k,i};M)$. Therefore, to show \eqref{eq: z deformed prop rephrased}, it remains to show the following:  For $M\in \Mat^{\ge0}_{r\times(k-1)}$ such that $|M|=k-1$, denoting $\alpha=\rsum(M)$, we have 
\begin{equation}\label{eq: z deformed prop rephrased11}
    \sum_{i=1}^{k} \left(\prod_{j\neq i}\dfrac{1}{z_j-z_i}\right)\sum_{L\in \mathbf{OP}(\overline{R}_{k};M)}\frac{1}{\overline{\stat}_{(\overline{R}_{k},\overline{g}^{\mathbf{z}}_{k,i})}(L)}=\left(\frac{q^{\sum_{i,j}\binom{M_{i,j}}{2}}}{q^{k(k-1)}\prod_{i=1}^{k-1}z_{2k-i}^{\alpha_i}}\right)\LB(M).
\end{equation}

Applying Lemma~\ref{lem: sf degree less k-1 vanishes} to the left-hand side, \eqref{eq: z deformed prop rephrased11} becomes

\begin{equation}\label{eq: z deformed prop rephrased2}
    \sum_{i=1}^{k} \left(\prod_{j\neq i}\dfrac{1}{z_j-z_i}\right)\sum_{L\in \mathbf{OP}^{\mathrm{red}}(\overline{R}_{k};M)}\frac{1}{\overline{\stat}_{(\overline{R}_{k},\overline{g}^{\mathbf{z}}_{k,i})}(L)}=\left(\frac{q^{\sum_{i,j}\binom{M_{i,j}}{2}}}{q^{k(k-1)}\prod_{i=1}^{k-1}z_{2k-i}^{\alpha_i}}\right)\LB(M).
\end{equation}

To prove \eqref{eq: z deformed prop rephrased2}, we first replace $\overline{\stat}_{(\overline{R}_k,\overline{g}^{\mathbf{z}}_{k,i})}(L)$ in a more simpler term (Lemma~\ref{lem: rectangle stat is simple}) exploiting the following property: for each cell $v$ (other than cells in the top row) of the diagram $\overline{R}_k$ there are precisely $(k-1)$ cells $u$ such that $(u,v)$ is an attacking pair.

\begin{definition}
For a diagram $D$ and $L=(A_1,\dots,A_r)\in \mathbf{OP}(D;M) $ for some matrix $M$, let $N$ be the total number of attacking pairs $(u,v)$ such that $u\in A_j$ and $v\in A_i$ for $1 \leq i\leq j$. We define the partial inversion $\pinv_{D}(L)\coloneqq q^{N}$. Then we define $c(D;L)$ to be a vector whose $i$-th entry is the number of cells in the $i$-th column of $D$ that are in $\cup_{t=1}^{r}A_t$. For example, consider $L$ and $L'$ in Figure \ref{Fig: A reduced tuple and a non-reduced tuple}. We have $c(\overline{R}_4;L)=(1,2,2,2)$ and $c(\overline{R}_4;L')=(2,2,2,1)$. 
\end{definition}

\begin{lem}\label{lem: rectangle stat is simple}
 Let $M \in \Mat^{\ge0}_{r\times(k-1)}$ and for $L\in \mathbf{OP}^{\mathrm{red}}(\overline{R}_k;M)$, denote $\alpha=\rsum(M)$ and $c=c(\overline{R}_k;L)$. We have
\begin{equation*}
    \frac{1}{\overline{\stat}_{(\overline{R}_k,\overline{g}^{\mathbf{z}}_{k,i})}(L)}=\frac{z_i^{c_k}\eta_{k,i}(\prod_{j=1}^{k-1}z_j^{c_j})}{q^{k |M|}\prod_{j=1}^{k-1}z_{2k-j}^{\alpha_j}}\pinv_{\overline{R}_k}(L).
\end{equation*}
\end{lem}
\begin{proof}
Denote $L=(A_1,\dots,A_r)$, then for each cell $v\in \cup_{t=1}^{r}A_t$, there are exactly $(k-1)$ cells $u$ such that $(u,v)$ is an attacking pair. Since pairs $(u,v)$ that contribute to $\pinv_{\overline{R}_k}(L)$ are precisely those do not contribute to $\overline{\inv}_{\overline{R}_k}(L)$, we have
\begin{equation*}
    \overline{\inv}_{\overline{R}_k}(L)=\frac{q^{(k-1)|M|}}{\pinv_{\overline{R}_k}(L)}.
\end{equation*}
As $L$ is reduced, every cell that is located right above $v\in \cup_{t=1}^{r}A_t$ is in $\overline{\Des}_{\overline{R}_k}(L)$. From the explicit description of the filling $\overline{g}^{\mathbf{z}}_{k,i}$, we have 
\begin{equation*}
    \overline{\maj}_{(\overline{R}_k,\overline{g}^{\mathbf{z}}_{k,i})}(L)=\frac{q^{|M|}\prod_{j=1}^{k-1}z_{2k-j}^{\alpha_j}}{z_i^{c_k}\eta_{k,i}(\prod_{j=1}^{k-1}z_j^{c_j})}.
\end{equation*}

\end{proof}

\begin{definition}
A matrix $E$ is a \emph{Fibonacci matrix} if it satisfies the following conditions:
\begin{enumerate}
    \item the first column is a zero column. 
    \item in each column, there is at most one nonzero entry, and if there is, the entry is 1.
    \item if $E_{a,j}=E_{b,j+1}=1$, we have $a>b$.
\end{enumerate}
The name "Fibonacci" is justified by the fact that the number of Fibonacci matrices of size $1 \times \ell$ is given by the Fibonacci sequence, i.e., $\{f_{\ell}\}_{\ell\ge1}=1,2,3,5,8,13,\dots$. We denote the set of Fibonacci matrices of size $r\times \ell$ by $\mathrm{Fib}_{r,\ell}$. For example, $\mathrm{Fib}_{2,3}$ has the following six elements,
\[
    \begin{bmatrix}
0 & 0 & 0\\
0 & 0 & 0
\end{bmatrix}, \begin{bmatrix}
0 & 1 & 0\\
0 & 0 & 0
\end{bmatrix}, \begin{bmatrix}
0 & 0 & 1\\
0 & 0 & 0
\end{bmatrix}, \begin{bmatrix}
0 & 0 & 0\\
0 & 1 & 0
\end{bmatrix}, \begin{bmatrix}
0 & 0 & 0\\
0 & 0 & 1
\end{bmatrix}, \text{ and } \begin{bmatrix}
0 & 0 & 1\\
0 & 1 & 0
\end{bmatrix}.
\]

\end{definition}

Now we state a lemma that describes the left-hand side of \eqref{eq: z deformed prop rephrased2} in terms of the filled diagram $(R_k,g_{k,i})$ and the Fibonacci matrices. For a 0-1 matrix $E$, we abuse the notation and denote $\LB(M;E)$ to be the sum of $\LB(M;(i,j))$ over all $(i,j)$ such that  $E_{i,j}=1$,
\[
    \LB(M;E) = \sum_{\substack{(i,j)\\E_{i,j}=1}} \LB(M;(i,j)).
\]

\begin{lem}\label{lem: bar Rk to Rk reduction}
For a matrix $M\in\Mat^{\ge0}_{r\times(k-1)}$ with $\rsum(M)=\alpha$, we have
\begin{align*}
    &\sum_{L\in \mathbf{OP}^{\mathrm{red}}(\overline{R_k};M)}\frac{1}{\overline{\stat}_{(\overline{R}_k,\overline{g}_{k,i})}(L)}\\&=\frac{1}{q^{k |M|}\prod_{j=1}^{k-1}z_{2k-j}^{\alpha_j}}\sum_{E\in \mathrm{Fib}_{r,k-1}}z_i^{|E|}q^{\LB (M;E)-|E|}\left(
    \sum_{L'\in \mathbf{OP}^{\mathrm{red}}(R_k;M-E)}\eta_{k,i}\left(\prod_{j=1}^{k-1}z_j^{c(R_k;L')_j}\right)\pinv_{R_k}(L')\right).
\end{align*}
\end{lem}
\begin{proof}
    For a given $L=(A_1,A_2,\dots,A_r)\in\boldOP^{\rred}(\overline{R_k};M)$, we associate a  $r\times(k-1)$ matrix E by letting $E_{i,j}=1$ if the cell $(j,k)$ is in $A_i$, otherwise $E_{i,j}=0$. The diagram $\overline{R}_k$ does not have a cell $(1,k)$ thus the first column of $E$ is a zero column. Moreover, since $L$ is reduced, we conclude that $E$ is the Fibonacci matrix. It is direct to see that  $L'=L\vert_{R_k}\in \mathbf{OP}^{\mathrm{red}}(R_k;M-E)$. We claim that 
    \begin{equation}\label{eq: additional pinv}
    \pinv_{\overline{R}_k}(L)=q^{\LB(M;E)-|E|}\pinv_{R_k}(L').
    \end{equation}
    Pick an entry of $E$ such that $E_{i,j}=1$, then $(j,k)\in A_i$. Let us denote $v=(j,k)$. For a cell $u\neq v$, the pair $(u,v)$ contributes to $\pinv_{\overline{R}_k}(L)$ if and only if either
    \begin{itemize}
        \item if $u$ is in the $j$-th row and $u\in A_{i'}$ for $i'\geq i$, 
        \item if $u$ is in the $(j-1)$-th row such that $u\in A_{i'}$ for $i\geq i'$. 
    \end{itemize}
    We conclude that for each $E_{i,j}=1$, there are total $\LB(M;(i,j))-1$ pairs that give an additional contribution to $\pinv_{\overline{R}_k}(L)$ and \eqref{eq: additional pinv} follows.

    It is straightforward to check that $c(R_k;L')_j=c(\overline{R}_k;L)_j$ for $1\leq j\leq k-1$ and $c(\overline{R}_k;L)_k=|E|$. Together with Lemma~\ref{lem: rectangle stat is simple} and \eqref{eq: additional pinv},
    partitioning $\boldOP(\overline{R_k};M)$ with respect to the associated Fibonacci matrix $E$ completes the proof. 
\end{proof}

Using Lemma~\ref{lem: bar Rk to Rk reduction}, we reduce \eqref{eq: z deformed prop rephrased2} to the following: For $M\in \Mat^{\ge0}_{r\times(k-1)}$ such that $|M|=k-1$, we have 
\begin{equation}\label{eq: z deformed prop rephrased3}
    \sum_{E\in\mathrm{Fib}_{r,k-1}}q^{\LB (M;E)-|E|}\left(\sum_{i=1}^{k}\left(\prod_{j\neq i}\frac{1}{z_j-z_i}\right)z_i^{|E|}\sum_{L\in \mathbf{OP}^{\mathrm{red}}(R_k;M-E)}\eta_{k,i}\left(\prod_{j=1}^{k-1}z_j^{c(R_k;L)_j}\right)\pinv_{R_k}(L)\right)=q^{\sum_{i,j}\binom{M_{i,j}}{2}}\LB(M).
\end{equation}

Now we state Proposition \ref{prop: toward lb tilde}, which is the last step of proof of Proposition~\ref{prop: main key sf to shuffle}. The proof of Proposition \ref{prop: toward lb tilde} will be given in Section~\ref{subsec: proof of Proposition 6.18}. Before the statement, we introduce $\widetilde{\LB}(M)$ which slightly modifies $\LB(M)$. For a matrix $M\in\Mat^{\ge0}_{r\times\ell}$, we define 
\[
    \widetilde{\LB}(M)\coloneqq \binom{\sum_{1\le a\le r}M_{a,1}} {M_{1,1},M_{2,1},\dots,M_{r,1}}_q\prod_{\substack{1\le i \le r\\2\leq j\le \ell}}\binom{\LB(M;(i,j))}{M_{i,j}}_q.
\]

\begin{proposition}\label{prop: toward lb tilde}
For $M\in \Mat^{\ge0}_{r\times(k-1)}$ such that $|M|\leq k-1$, we have 

\begin{align*}
    \sum_{i=1}^{k}\left(\prod_{j\neq i}\frac{1}{z_j-z_i}\right)z_i^{k-1-|M|}\sum_{L\in \mathbf{OP}^{\mathrm{red}}(R_k;M)}\eta_{k,i}\left(\prod_{j=1}^{k-1}z_j^{c(R_k;L)_j}\right)\pinv_{R_k}(L)=(-1)^{k-1-|M|}q^{\sum_{i,j}\binom{M_{i,j}}{2}} \widetilde{\LB}(M).
\end{align*}
\end{proposition}

We end this subsection with a proof of Proposition~\ref{prop: main key sf to shuffle}, assuming Proposition~\ref{prop: toward lb tilde}.

\begin{proof}[Proof of Proposition~\ref{prop: main key sf to shuffle}]
We have shown that Proposition~\ref{prop: main key sf to shuffle z-deformed} specializes to Proposition~\ref{prop: main key sf to shuffle}, and is equivalent to \eqref{eq: z deformed prop rephrased3}. Applying Proposition~\ref{prop: toward lb tilde}, the left hand side of \eqref{eq: z deformed prop rephrased3} becomes
\begin{align*}
    &\sum_{E\in\mathrm{Fib}_{r,k-1}}(-1)^{k-1-|M-E|}q^{\LB (M;E)+\sum_{i,j}\binom{M_{i,j}-E_{i,j}}{2}-|E|} \widetilde{\LB}(M-E)\\&=\sum_{E\in\mathrm{Fib}_{r,k-1}}(-1)^{|E|}q^{\LB (M;E)-\sum_{(i,j):E_{i,j}=1}M_{i,j} +\sum_{i,j}\binom{M_{i,j}}{2}} \widetilde{\LB}(M-E).
\end{align*}
By Lemma~\ref{Lem: recurrence tilde LV and LV}, this equals to $q^{\sum_{i,j}\binom{M_{i,j}}{2}} \LB(M),$
which is the right hand side of \eqref{eq: z deformed prop rephrased3}. 
\end{proof}

\subsection{Proof of Proposition~\ref{prop: toward lb tilde}}\label{subsec: proof of Proposition 6.18}
Now, all that remains is to prove Proposition~\ref{prop: toward lb tilde}. We outline our strategy briefly. First, we establish the symmetry in Lemma~\ref{lem: symmetric function!}, which allows us to utilize Corollary~\ref{Cor: symmetric polynomial version of Lemma A3}. Then we can rewrite the left-hand side of Proposition~\ref{prop: toward lb tilde} as a signed sum of $\pinv_{R_k}(L)$ over specific elements called `packed' in $\mathbf{OP}(R_k;M)$ (Definition~\ref{def: packed}), as shown in \eqref{eq: before involution}. Next, we introduce a sign-reversing involution $\tau_{R_k}$ (Definition~\ref{def: involution}). Finally, we interpret the generating function of $\pinv_{R_k}(L)$ over the fixed points of $\tau_{R_k}$ as $\widetilde{\LB}(M)$ (Lemma~\ref{Lem: fixed points = tilde lb}), which completes the proof.

\begin{lem}\label{lem: symmetric function!}
For $M\in \Mat^{\ge0}_{r\times(k-1)}$, the following term
\begin{equation*}
    \sum_{L\in \mathbf{OP}^{\mathrm{red}}(R_k;M)}\left(\prod_{j=1}^{k-1}z_j^{c(R_k;L)_j}\right)\pinv_{R_k}(L)
\end{equation*} is symmetric in variables $z_1,\dots,z_{k-1}$.
\end{lem}
\begin{proof}
We will show the symmetry in variables $z_1,z_2$. The proof for the rest variables follows with the same argument. Let $(D,f)$ be the sub-filled diagram of $(\overline{\delta}_{k,k},\overline{f}_{k,k})$ obtained by restricting to the first $(k-1)$ columns and let $(D',f')=S_1(D,f)$. Then by \eqref{eq: refined identity}, we have
\begin{equation}\label{eq: wath}
    \sum_{L\in \mathbf{OP}(D;M)}\frac{1}{\overline{\stat}_{(D,f)}(L)}=\sum_{L\in \mathbf{OP}(D';M)}\frac{1}{\overline{\stat}_{(D',f')}(L)}.
\end{equation}

Note that $(R_k,g_{k,k})$ is the sub-filled diagram of $(D,f)$ obtained by restricting to the first $k$ rows, and we denote $(R_k,g')$ to be the sub-filled diagram of $(D',f')$ obtained by restricting to the first $k$ rows. Since the matrix $M$ has $(k-1)$ columns, \eqref{eq: wath} becomes 
\begin{equation}\label{eq: wath1}
    \sum_{L\in \mathbf{OP}(R_k;M)}\frac{1}{\overline{\stat}_{(R_k,g_{k,k})}(L)}=\sum_{L\in \mathbf{OP}(R_k;M)}\frac{1}{\overline{\stat}_{(R_k,g')}(L)}.
\end{equation}
Let $\alpha=\rsum(M)$. By \eqref{Eq: stat bar = P/z}, taking the coefficients of $\prod_{j=1}^{k-1}z_{2k-j}^{-\alpha_j}$ of \eqref{eq: wath1} gives
\begin{equation*}
    \sum_{L\in \mathbf{OP}^{\mathrm{red}}(R_k;M)}\frac{1}{\overline{\stat}_{(R_k,g_{k,k})}(L)}=\sum_{L\in \mathbf{OP}^{\mathrm{red}}(R_k;M)}\frac{1}{\overline{\stat}_{(R_k,g')}(L)}.
\end{equation*}
As $g'$ is the filling obtained from $g_{k,k}$ by exchanging the variables $z_1$ and $z_2$, we conclude that 
\begin{equation*}
    \sum_{L\in \mathbf{OP}^{\mathrm{red}}(R_k;M)}\frac{1}{\overline{\stat}_{(R_k,g_{k,k})}(L)}
\end{equation*}
is symmetric in variables $z_1,z_2$. Finally, by the same argument in Lemma~\ref{lem: rectangle stat is simple}, we have
\begin{equation*}
    \frac{1}{\overline{\stat}_{(R_k,g_{k,k})}(L)}=\frac{\prod_{j=1}^{k-1}z_j^{c(R_k;L)_j}}{q^{(k-1) |M|}\prod_{j=1}^{k-1}z_{2k-j}^{\alpha_j}}\pinv_{R_k}(L).
\end{equation*}
for $L\in \mathbf{OP}^{\mathrm{red}}(R_k;M)$ and the proof is complete.
\end{proof}

\begin{definition}\label{def: packed}
For $L\in \mathbf{OP}(R_k;M)$ for some matrix $M$, we say that $L$ is \emph{packed} with respect to $R_k$ if the vector $c(R_k;M)=(\alpha,0,\dots,0)$ for some composition $\alpha$. For such $L$, we let $\ell(L)$ to be $\ell(\alpha)$. We define $\mathbf{OP}^{\mathrm{red,packed}}(R_k;M)$ to be the set of all reduced and packed $L\in \mathbf{OP}(R_k;M)$. 
\end{definition}

\begin{figure}[h]
    \centering
    $L=$ \begin{ytableau}
    $ $ & $ $ & $ $ & $ $ & $ $ \\
    $ $ & $ $ & $ $ & $ $ & $ $ \\
    2 &  &  & & \\
    $ $ & 4  & 1 & 3 &\\
    $ $ &  &  & 4 & \\
    1 &  & 2 & & 
    \end{ytableau}
    \qquad 
    $L'=$ \begin{ytableau}
    $ $ & $ $ & $ $ & $ $ & $ $ \\
    $ $ & $ $ & $ $ & 2 & $ $ \\
    $ $ & 1 &  & {4} &\\
    2 & {3} &  & &\\
    $ $ & 4 &  & &\\
    1 &  & & &
    \end{ytableau}
    \caption{On the left shows $L$ which is packed while $L'$ on the right is not packed.}
    \label{Fig: A packed tuple and a non-packed tuple}
\end{figure}

By Lemma~\ref{lem: symmetric function!} and Corollary~\ref{Cor: symmetric polynomial version of Lemma A3}, Proposition \ref{prop: toward lb tilde} is equivalent to
\begin{align}\label{eq: before involution}
      \sum_{L\in \mathbf{OP}^{\mathrm{red,packed}}(R_k;M)}(-1)^{|M|-\ell(L)}\pinv_{R_K}(L)=q^{\sum_{i,j}\binom{M_{i,j}}{2}} \widetilde{\LB}(M).
\end{align}

From now on we fix $M\in \Mat^{\ge0}_{r\times(k-1)}$ such that $|M|\leq k-1$. We construct a sign-reversing involution on $\mathbf{OP}^{\mathrm{red,packed}}(R_k;M)$. 

\begin{definition}\label{def: involution}
For $L=(A_1,\dots,A_r) \in \mathbf{OP}^{\mathrm{red,packed}}(R_k;M)$, we say that a cell $(a,b)\in \cup_{i=1}^{r}A_i$ is \emph{mergeable} if it satisfies the following conditions:
    \begin{itemize}
        \item if $(a',b)\in \cup_{i=1}^{r}A_i$ then $a'<a$ and there is a cell in  $\cup_{i=1}^{r}A_i$ in the $(b+1)$-th column
        \item if $(a',b+1)\in \cup_{i=1}^{r}A_i$ then $a'>a$
        \item if $(a+1,b+1)\in \cup_{i=1}^{r}A_i$ then $(a+1,b+1)\in A_{\ell}$ and $(a,b)\in A_{\ell'}$ for $\ell<\ell'$.
    \end{itemize}
    We define the cell $(a,b)\in \cup_{i=1}^{r}A_i$ is \emph{splittable} if there exists $(a',b)\in \cup_{i=1}^{r}A_i$ such that $a'>a$.
    We denote the set of mergeable cells with $\Merge(R_k;L)$ and the set of splittable cells with $\Split(R_k;L)$.    
\end{definition}

\begin{definition} We first give a total order on cells in $R_k$ by letting $u=(a,b) <v=(a',b')$ if $b<b'$ or $b=b'$ and $a<a'$.
For $L=(A_1,\dots,A_r) \in \mathbf{OP}^{\mathrm{red,packed}}(R_k;M)$ if there is no mergeable nor splittable cell, then we define $\tau_{R_k}(L)=L$. Otherwise, let $u=(a,b)$ be the minimal (in terms of the total order) cell in $\Merge(R_k;L)\cup \Split(R_k;L)$. We perform the following process to $L$:

\begin{itemize}
    \item if $u\in \Merge(R_k;L)$ then replace each cell $v=(c,d)$ in $\cup_{i=1}^r A_i$ satisfying $u< v$ by $(c,d-1)$.
    \item if $u\in \Split(R_k;L)$ then replace each cell $v=(c,d)$ in $\cup_{i=1}^r A_i$ satisfying $u< v$ by $(c,d+1)$.
\end{itemize}
We define $\tau_{R_k}(L)$ to be the result. We simply write $\tau(L)$ when $k$ is obvious from the context. For the first case, it is straightforward to see $\tau(L)\in \mathbf{OP}^{\mathrm{red,packed}}(R_k;M)$. For the second case, $|\cup_{i=1}^r A_i|=|M|\leq k-1$ and $b$-th column of $R_k$ has at least two cells in $\cup_{i=1}^r A_i$. Since $L$ is packed the largest $d$ such that $d$-th column has a cell in $\cup_{i=1}^r A_i$ is at most $(k-2)$. Therefore $(c,d+1)$ is a cell in $R_k$ and we have $\tau(L)\in \mathbf{OP}^{\mathrm{red,packed}}(R_k;M)$.

\end{definition}

\begin{figure}[h]
    \centering
   
    $L=$ \begin{ytableau} 
    $ $ & $ $ & $ $ & $ $ & $ $ \\
    $ $ & $ $ & $ $ & $ $ & $ $ \\
    $ $ & $ $ & 1 & $ $ & $ $ \\
    1 & $ $ & *(gray!30)2 & $ $ & $ $ \\
    $ $ & *(black)\textcolor{white}{3} & $ $ & $ $ & $ $ \\
    *(gray!30) 2 & $ $ & $ $ & $ $ & $ $ \\
    \end{ytableau}
    \qquad 
    $\tau(L)=$ \begin{ytableau}
    $ $ & $ $ & $ $ & $ $ & $ $ \\
    $ $ & $ $ & $ $ & $ $ & $ $ \\
    $ $ & $ $ & $ $ & 1 & $ $ \\
    $ $ & 1 & $ $ & *(gray!30) 2 & $ $ \\
    $ $ & $ $ & *(black)\textcolor{white}{3} & $ $ & $ $ \\
    *(black)\textcolor{white}{2} & $ $ & $ $ & $ $ & $ $ \\
    \end{ytableau}

    \caption{An example of $L$ and its image $\tau(L)$ is shown. Splittable cells are shaded in gray, and the mergeable cells are shaded in black with letters in white.}
    \label{fig: involution tau}
\end{figure}

It is easy to check that $\tau$ is an involution on the set $\boldOP^{\mathrm{red},\mathrm{packed}}({R_k;M})$. Furthermore, we have
$\pinv(L)=\pinv(\tau(L))$ and $\ell(L)=\ell(\tau(L))\pm 1$ if $L$ is not fixed by $\tau$: $\tau$ is sign-reversing. Moreover if $L=(A_1,\dots,A_r)$ is fixed by $\tau$ then there is at most one cell in $\cup_{i=1}^{r}A_i$ in each column of $R_k$ otherwise $\Split(R_k;L)$ is not empty. We get $\ell(L)=|M|$ for $\tau$-fixed $L$ and therefore, via this involution $\tau$, the left-hand side of \eqref{eq: before involution} becomes the sum over $\tau$-fixed $L\in \boldOP^{\mathrm{red},\mathrm{packed}}({R_k;M})$. The following lemma provides a formula for the sum. 

\begin{lem}\label{Lem: fixed points = tilde lb} For a matrix $M\in\Mat^{\ge0}_{r\times(k-1)}$ such that $|M|\leq k-1$, we have
\[
    \sum_{\substack{L\in\boldOP^{\mathrm{red},\mathrm{packed}}(R_k;M)\\ \tau(L) = L}} {\mathrm{pinv}_{R_k}(L)}=q^{\sum_{i,j}\binom{M_{i,j}}{2}}\widetilde{\LB}(M).
\]
\end{lem}
\begin{proof}
Firstly, we describe $L=(A_1,\dots,A_r)$ that is fixed by $\tau$. Since there is at most one cell in $\cup_{i=1}^{r}A_i$ in each column of $R_k$ we may identify such $L$ with a biword $\bw(L)$ whose $a$-th letter is $\bw(L)_a =(i,j)$ if the cell $(j,a)\in A_i$. Then $\bw(L)$ has $M_{i,j}$-many $(i,j)$ and denoting $\bw(L) = ((i_1,j_1),(i_2,j_2),\dots)$, for $1 \leq r\leq |M|-1$ satisfies either:
\begin{itemize}
    \item $j_r \ge j_{r+1}$, or
    \item $j_r = j_{r+1}-1$ and $i_r \le i_{r+1}$
\end{itemize}
where two conditions above are equivalent to the condition $\Merge(R_k;L)=\emptyset$. We denote the set of such biwords by $\BW(M)$. Let $\bw(L) = ((i_1,j_1),(i_2,j_2),\dots)$. Under this correspondence, we have $\pinv_{R_{k}}(L)=q^N$, where $N$ is the number of pairs $(x,y)$ of indices satisfying either:
\begin{itemize} 
    \item $j_x=j_y$, $x<y$, and $i_x\ge i_y$, or
    \item $j_x=j_y+1$, $x>y$, and $i_x\ge i_y$.
\end{itemize}
We abuse our notation to denote this by $\pinv(\bw(L))$.

Let $\ell$ be the maximum of $m$ for which the $m$-th column of matrix $M$ contains a non-zero element. If $\ell=1$, the claim is straightforward. Suppose $\ell>1$, and let us proceed by induction on the sum $n=|M|$. The base case $M=O$ is trivial. Without loss of generality, we assume that the rightmost column of $M$ is nonzero. Let $c$ be the smallest index $i$ such that $M_{i,\ell}\neq0$, and let $M'$ be the matrix obtained from $M$ by replacing $M_{c,\ell}$ with zero.

Let $s=\LB(M;(c,\ell))$, and let $W_{s,M_{c,\ell}}$ be the set of 0-1 sequences of length $s$ with $M_{c,\ell}$ ones. Suppose we construct a bijection
\begin{align*}
    \widetilde{\Phi}:\BW(M) &\rightarrow \BW(M')\times W_{s,M_{c,\ell}}\\
    \bw &\mapsto (\widetilde{\Phi}_1(\bw),\widetilde{\Phi}_2(\bw)) 
\end{align*}
such that 
\begin{equation}\label{Eq: pinv = pinv + inv}
\pinv(\bw) = \pinv\left(\Phi_1(\bw)\right) q^{\inv\left(\Phi_2(\bw)\right)}q^{\binom{M_{c,\ell}}{2}}.
\end{equation}
Then we have
\begin{align*}
\sum_{\bw\in \BW(k,M)} {\pinv(\bw)}&=q^{\binom{M_{c,\ell}}{2}}\sum_{\bw \in \BW(k,M')} \pinv(\bw(L))\sum_{\sigma \in W_{s,M_{c,\ell}}}q^{\inv(\sigma)}\\
&=q^{\sum_{i,j}\binom{M_{i,j}}{2}}\widetilde{\LB}(M')\binom{\LB(M;(c,\ell))}{M_{c,\ell}}_q\\
&=q^{\sum_{i,j}\binom{M_{i,j}}{2}}\widetilde{\LB}(M).
\end{align*}
Here, The first equality follows from \eqref{Eq: pinv = pinv + inv} of the bijection $\widetilde{\Phi}$. The second equality follows from the induction hypothesis, and the last equality follows from the definition of $\widetilde{\LB}(M)$. Therefore, it suffices to construct such a bijection $\widetilde{\Phi}$.

Given $((i_1,j_1),(i_2,j_2),\dots)\in\BW(M)$, we define $\widetilde{\Phi}_1(L)$ as the biword obtained by deleting the letters $(c,\ell)$. We define $\widetilde{\Phi}_2(L)$ as the 0--1 word obtained from $L$ by replacing letters in $\LLB(c,\ell)\setminus \{(c,\ell)\}$ with zeros, $(c,\ell)$ with ones, and deleting the others. Finally, we define
\[
    \widetilde{\Phi}(\bw)\coloneqq\left(\widetilde{\Phi}_1(\bw),\widetilde{\Phi}_2(\bw)\right).
\]
An example of the bijection $\widetilde{\Phi}$ is given in Example~\ref{Ex: the bijection tilde Phi}.
 
In $\bw(L)$, the letters $(c,\ell)$ are always right next to the letter $(a,b)\in\{(a,b)\in\LLB(c,\ell)\}$ due to the characterization of $\BW(M)$. This guarantees that we can recover $\widetilde{\Phi}(\bw)$ from $\widetilde{\Phi}_1(\bw)$ and $\widetilde{\Phi}_2(\bw)$. Therefore the map $\widetilde{\Phi}$ is a bijection. Furthermore, the additional contribution to $\pinv(\bw)$ due to cells in the $\ell$-th row that are in $A_c$ is given by $q^{\binom{M_{c,\ell}}{2}}q^{\inv(\widetilde{\Phi}_2(\bw))}$. Therefore, we conclude that \eqref{Eq: pinv = pinv + inv} holds.
\end{proof}

\begin{example}\label{Ex: the bijection tilde Phi}
Let $M=\begin{bmatrix}
1 & 2 & 0 & 0 & 0 & 0\\
2 & 1 & 0 & 0 & 0 & 0
\end{bmatrix}$ and let $L\in\boldOP^{\mathrm{red},\mathrm{packed}}(R_7;M)$ depicted on the left in Figure~\ref{Fig: tilde Phi image}.

\begin{figure}[h]
    \centering
    \begin{ytableau}
    $ $ & $ $ & & $ $ & $ $ & $ $ \\
    $ $ & $ $ & & $ $ & $ $ & $ $ \\
    $ $ & $ $ & &$ $ & $ $ & $ $ \\
    $ $ &  &  & & &\\
    $ $ & $ $  & &$ $ & $ $ &\\
    1 &  & 2 &  &1 &\\
    $ $ & 2 & $ $ & 1 & &2 
    \end{ytableau} \qquad \qquad
    \begin{ytableau}
    $ $ & $ $ & & $ $ & $ $ & $ $ \\
    $ $ & $ $ & & $ $ & $ $ & $ $ \\
    $ $ & $ $ & &$ $ & $ $ & $ $ \\
    $ $ &  &  & & &\\
    $ $ & $ $  & &$ $ & $ $ &\\
    $ $ & 2 &   & & & \\
    2 & $ $ & 1 & 2 & & 
    \end{ytableau}
    \caption{An example of $L\in\boldOP^{\mathrm{red},\mathrm{packed}}(R_7;M)$ and its image $\widetilde{\Phi}(\bw(L))$.}
    \label{Fig: tilde Phi image}
\end{figure}

One can see that there is no splittable or mergeable cell, so $L$ is fixed by the involution $\tau$. We can identify $L$ with 
\[
    \bw(L) = ((1,2),(2,1),(2,2),(1,1),(1,2),(2,1)) \in \BW(M).
\]
Then we have $\pinv(L) = \pinv(\bw(L)) = q^6$.

Let $M'=\begin{bmatrix}
1 & 0 & 0 & 0 & 0 & 0\\
2 & 1 & 0 & 0 & 0 & 0
\end{bmatrix}$. 
By deleting letters $(1,2)$ from $\bw(L)$, we obtain
\[
    \widetilde{\Phi}_1(\bw(L)) = ((2,1),(2,2),(1,1),(2,1))\in\BW(M'),
\]
which is depicted on the right in Figure~\ref{Fig: tilde Phi image}. Then, by replacing letters in $\LLB(1,2)\setminus\{(1,2)\}=\{(1,1),(2,2)\}$ with zeros and the letter $(1,2)$ with ones and deleting others from $\bw(L)$, we obtain 
\[
    \widetilde{\Phi}_2(\bw(L)) = 1001.
\]
By direct calculation, we can check \eqref{Eq: pinv = pinv + inv} for this example as
\[
    \pinv(L) = \pinv(\bw(L)) = q^6 = q^{3}q^{2}q^{1} = \pinv(\widetilde{\Phi}_1(\bw(L)))q^{\inv(\widetilde{\Phi}_2(\bw(L))}q^{\binom{M_{1,2}}{2}}.
\]

\end{example}

\section{Macdonald intersection polynomials at $q=t=1$}\label{Sec: Macdonald intersection polynomials at q=t=1}
In this section, we prove Theorem~\ref{thm: h-expansion, Kreweras}, which gives a combinatorial formula for the (whole) Macdonald intersection polynomials at $q=t=1$. Before providing a proof, we review some backgrounds starting from Kreweras numbers $\Krew(\lambda)$. The \emph{Kreweras number} for a partition $\lambda$ is defined by
\[
    \Krew(\lambda)\coloneqq\dfrac{1}{n+1}\binom{n+1}{n+1-\ell(\lambda),m_1(\lambda),\dots,m_n(\lambda)},
\]
where $n=|\lambda|$ and $m_j(\lambda)$ denotes the multiplicity of $j$ among the parts of $\lambda$. Kreweras \cite{Kre72} originally interpreted $\Krew(\lambda)$ in terms of the number of \emph{noncrossing partitions} of $1,2,\dots,n$. Equivalently, Kreweras numbers count the number of Dyck paths whose lengths of consecutive horizontal steps are listed in $\lambda$. To be more precise, for a Dyck path $\pi$ of size $n$, we associate a composition $\alpha(\pi)$, defined as the composition where the number of consecutive East steps gives each part. For example, for a Dyck path $\pi=NNENEE$, $\alpha(\pi)=(1,2)$. We denote by $\Dyck(\alpha)$ the set of Dyck paths $\pi$ of size $n$ with $\alpha(\pi)=\alpha$. Then the Kreweras number $\Krew(\lambda)$ counts the size of $\bigcup_{\alpha\sim\lambda} \Dyck(\alpha)$, where the union is over all rearrangements $\alpha$ of $\lambda$ \cite{Deu99}. For example, $\Krew(2,1)=\frac{1}{4}\binom{4}{4-2,1,1}=3$, and the set $\Dyck(1,2)\cup\Dyck(2,1)$ consists of the following three Dyck paths.

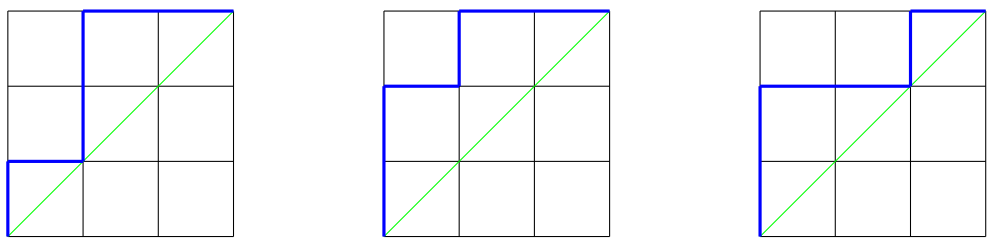
\begin{figure}[ht]\centering
\begin{tikzpicture}[scale=1]
\begin{scope}[shift={(0,0)}]
\draw[-] (0,0) -- (3,0);
\draw[-] (0,1) -- (3,1);
\draw[-] (0,2) -- (3,2);
\draw[-] (0,3) -- (3,3);

\draw[-] (0,0) -- (0,3);
\draw[-] (1,0) -- (1,3);
\draw[-] (2,0) -- (2,3);
\draw[-] (3,0) -- (3,3);

\draw[green] (0,0) -- (3,3);

\draw[blue, very thick] (0,0) -- (0,1);
\draw[blue, very thick] (0,1) -- (1,1);
\draw[blue, very thick] (1,1) -- (1,3);
\draw[blue, very thick] (1,3) -- (3,3);
\end{scope}

\begin{scope}[shift={(5,0)}]
\draw[-] (0,0) -- (3,0);
\draw[-] (0,1) -- (3,1);
\draw[-] (0,2) -- (3,2);
\draw[-] (0,3) -- (3,3);

\draw[-] (0,0) -- (0,3);
\draw[-] (1,0) -- (1,3);
\draw[-] (2,0) -- (2,3);
\draw[-] (3,0) -- (3,3);

\draw[green] (0,0) -- (3,3);

\draw[blue, very thick] (0,0) -- (0,2);
\draw[blue, very thick] (0,2) -- (1,2);
\draw[blue, very thick] (1,2) -- (1,3);
\draw[blue, very thick] (1,3) -- (3,3);
\end{scope}
\begin{scope}[shift={(10,0)}]
\draw[-] (0,0) -- (3,0);
\draw[-] (0,1) -- (3,1);
\draw[-] (0,2) -- (3,2);
\draw[-] (0,3) -- (3,3);

\draw[-] (0,0) -- (0,3);
\draw[-] (1,0) -- (1,3);
\draw[-] (2,0) -- (2,3);
\draw[-] (3,0) -- (3,3);

\draw[green] (0,0) -- (3,3);

\draw[blue, very thick] (0,0) -- (0,2);
\draw[blue, very thick] (0,2) -- (2,2);
\draw[blue, very thick] (2,2) -- (2,3);
\draw[blue, very thick] (2,3) -- (3,3);
\end{scope}

\end{tikzpicture}
\end{figure}

We now recall the notion of \emph{$\lambda$-brick tabloids}, which appears as the entries of the transition matrix between elementary and homogeneous symmetric functions. For a partition $\lambda\vdash n$ and a composition $\alpha=(\alpha_1,\dots,\alpha_\ell)\models n$, a $\lambda$-brick tabloid of shape $\alpha$ is a tuple of compositions $(\alpha^{(1)},\alpha^{(2)},\dots,\alpha^{(\ell)})$ satisfying
\begin{itemize}
\item $\alpha^{(i)}\models \alpha_i$, and
\item $\sum_{i=1}^\ell m_j(\alpha^{(i)}) = m_j(\lambda)$ for all $j$,
\end{itemize}
where $m_j(\beta)$ is the multiplicity of $j$ in the composition $\beta$. We denote by $B_{\lambda,\alpha}$ for the set of $\lambda$-brick tabloid of shape $\alpha$. For example, \[B_{(2,1,1,1),(3,2)}=\{((2,1),(1,1)),((1,2),(1,1)),((1,1,1),(2))\}.\]
Using $\lambda$-brick tabloids, Egecioglu and Remmel gave a combinatorial formula for the transition matrix for the elementary symmetric functions and homogeneous symmetric functions.
\begin{proposition}\cite{ER91}\label{prop: h-expansion of e}
For a partition $\mu\vdash n$, we have
\[
    e_\mu[X] = \sum_{\lambda\vdash n}(-1)^{n-\ell(\lambda)}|B_{\lambda,\mu}|h_\lambda[X].
\]
\end{proposition}

Using the previous proposition, the following lemma provides an $h$-expansion for $D_n[X;1,1]$. 
\begin{lem}\label{lemma: h-expansion for Nabla e_n}
For a given positive integer $n$,
\[
    D_n[X;1,1] = \sum_{\lambda\vdash n} (-1)^{n-\ell(\lambda)}\Krew(\lambda+(1^{\ell(\lambda)})) h_\lambda[X].
\]
\end{lem}
\begin{proof}
In \cite[Section 4.1]{HHLRU05}, the authors provide the $e$-expansion for $D_n[X;1,1]$ as
\begin{equation*}
    D_n[X;1,1] =\sum_{\lambda\vdash n} \Krew(\lambda)e_\lambda[X].
\end{equation*}
By Proposition~\ref{prop: h-expansion of e} we have
\[
    D_n[X;1,1]=\sum_{\lambda\vdash n} (-1)^{n-\ell(\lambda)}\left(\sum_{\mu\vdash n} \Krew(\mu)|B_{\lambda,\mu}|\right)h_\lambda
\]
Since we have $|B_{\lambda,\alpha}|=|B_{\lambda,\beta}|$ for a rearrangement $\beta$ of $\alpha$, it suffices to construct a bijection between 
\[
    \bigcup_{\beta\models n} \Dyck(\beta)\times B_{\lambda,\beta}\qquad \text{ and } \qquad \bigcup_{\alpha\sim\lambda+(1^{\ell(\lambda)})}\Dyck(\alpha).
\]

For given a composition $\beta=(\beta_1,\dots,\beta_\ell)$ of $n$, let $\pi\in \Dyck(\beta)$. Then $\pi$ is of the form 
\[N^{h_0}E^{\beta_1}N^{h_1}\cdots E^{\beta_\ell}\]
for some $h_0,\dots,h_{\ell-1}\ge 1$. Let $\tilde{\beta}=(\beta^{(1)},\dots,\beta^{(\ell)})$ be a $\lambda$-brick tabloid of shape $\beta$. We associate a path $P_\gamma$ to each composition $\gamma=(\gamma_1,\dots,\gamma_\ell)$ by
\[
    P_\gamma \coloneqq NE^{\gamma_1+1}NE^{\gamma_2+1}\cdots NE^{\gamma_\ell+1}.
\]
Then we define $\Psi_\lambda\left(\pi,\tilde{\beta}\right)$ to be the Dyck path
\[
    \Psi_\lambda\left(\pi,\tilde{\beta}\right)\coloneqq N^{h_0}P_{\beta^{(1)}}N^{h_1}\cdots P_{\beta^{(\ell)}}.
\]
An example of the bijection $\Psi_\lambda$ is given in Example~\ref{Example: the bijection Psi}.
Conversely, any Dyck path $\pi$ in $\Dyck(\alpha)$ for $\alpha\sim\lambda+(1^{\ell(\lambda)})$ is uniquely written as a form of
\[
\pi = N^{h_0}P_{\beta^{(1)}}N^{h_1}\cdots P_{\beta^{(\ell)}}
\]
for some $h_0,\dots,h_{\ell-1}\ge 1$, a composition $\beta\models n$, and $(\beta^{(1)},\dots,\beta^{(\ell)})\in B_{\lambda,\beta}$. Then the inverse map is given by replacing each $P_{\beta^{(i)}}$ by $E^{\beta_i}$. 
\end{proof}

\begin{example}\label{Example: the bijection Psi}
Let $\beta=(3,4,3)$
and $\lambda= (3,2,1,1,1,1,1)$. Pick an element $\widetilde{\beta} = ((2,1),(1,3),(1,1,1)) \in
B_{\lambda,\beta}.$ 
We have 
\[
    P_{(2,1)} = \begin{tikzpicture}[scale=0.5,baseline={([yshift=-.5ex]current bounding box.center)}]

    \draw[-] (0,0) -- (0,2);
    \draw[-] (1,0) -- (1,2);
    \draw[-] (2,0) -- (2,2);
    \draw[-] (3,0) -- (3,2);
    \draw[-] (4,0) -- (4,2);
    \draw[-] (5,0) -- (5,2);
    
    \draw[-] (0,0) -- (5,0);
    \draw[-] (0,1) -- (5,1);
    \draw[-] (0,2) -- (5,2);
    
    \draw[-,red,very thick] (0,0) -- (0,1);
    \draw[-,red,very thick] (0,1) -- (3,1);
    \draw[-,red,very thick] (3,1) -- (3,2);
    \draw[-,red,very thick] (3,2) -- (5,2);
    \end{tikzpicture}, \quad P_{(1,3)} = \begin{tikzpicture}[scale=0.5,baseline={([yshift=-.5ex]current bounding box.center)}]

    \draw[-] (0,0) -- (0,2);
    \draw[-] (1,0) -- (1,2);
    \draw[-] (2,0) -- (2,2);
    \draw[-] (3,0) -- (3,2);
    \draw[-] (4,0) -- (4,2);
    \draw[-] (5,0) -- (5,2);
    \draw[-] (6,0) -- (6,2);
    
    \draw[-] (0,0) -- (6,0);
    \draw[-] (0,1) -- (6,1);
    \draw[-] (0,2) -- (6,2);

    \draw[-,blue,very thick] (0,0) -- (0,1);
    \draw[-,blue,very thick] (0,1) -- (2,1);
    \draw[-,blue,very thick] (2,1) -- (2,2);
    \draw[-,blue,very thick] (2,2) -- (6,2);
    \end{tikzpicture}, \quad P_{(1,1,1)} =  \begin{tikzpicture}[scale=0.5,baseline={([yshift=-.5ex]current bounding box.center)}]
    \draw[-] (0,0) -- (0,3);
    \draw[-] (1,0) -- (1,3);
    \draw[-] (2,0) -- (2,3);
    \draw[-] (3,0) -- (3,3);
    \draw[-] (4,0) -- (4,3);
    \draw[-] (5,0) -- (5,3);
    \draw[-] (6,0) -- (6,3);
    \draw[-] (0,0) -- (6,0);
    \draw[-] (0,1) -- (6,1);
    \draw[-] (0,2) -- (6,2);
    \draw[-] (0,3) -- (6,3);

    \draw[-,green,very thick] (0,0) -- (0,1);
    \draw[-,green,very thick] (0,1) -- (2,1);
    \draw[-,green,very thick] (2,1) -- (2,2);
    \draw[-,green,very thick] (2,2) -- (4,2);
    \draw[-,green,very thick] (4,2) -- (4,3);
    \draw[-,green,very thick] (4,3) -- (6,3);
    \end{tikzpicture}.
\]
Let $\pi$ be a Dyck path $\pi=N^5E^3N^3E^4N^2E^3$ depicted as below 
\[
    \begin{tikzpicture}[scale=0.5,baseline={([yshift=-.5ex]current bounding box.center)}]

    \draw[-] (0,0) -- (0,10);
    \draw[-] (1,0) -- (1,10);
    \draw[-] (2,0) -- (2,10);
    \draw[-] (3,0) -- (3,10);
    \draw[-] (4,0) -- (4,10);
    \draw[-] (5,0) -- (5,10);
    \draw[-] (6,0) -- (6,10);
    \draw[-] (7,0) -- (7,10);
    \draw[-] (8,0) -- (8,10);
    \draw[-] (9,0) -- (9,10);
    \draw[-] (10,0) -- (10,10);

    \draw[-] (0,0) -- (10,0);
    \draw[-] (0,1) -- (10,1);
    \draw[-] (0,2) -- (10,2);
    \draw[-] (0,3) -- (10,3);
    \draw[-] (0,4) -- (10,4);
    \draw[-] (0,5) -- (10,5);
    \draw[-] (0,6) -- (10,6);
    \draw[-] (0,7) -- (10,7);
    \draw[-] (0,8) -- (10,8);
    \draw[-] (0,9) -- (10,9);
    \draw[-] (0,10) -- (10,10);

    \draw[-] (0,0) -- (10,10);
    
    \draw[-,very thick] (0,0) -- (0,5);
    \draw[-,very thick,red] (0,5) -- (3,5);
    \draw[-,very thick] (3,5) -- (3,8);
    \draw[-,very thick,blue] (3,8) -- (7,8);
    \draw[-,very thick] (7,8) -- (7,10);
    \draw[-,very thick,green] (7,10) -- (10,10);
    \end{tikzpicture}.
\]
By replacing each $E^{\beta_i}$ by $P_{\beta^{(i)}}$, we have
\[
    \Psi_\lambda(\pi)=\begin{tikzpicture}[scale=0.3,baseline={([yshift=-.5ex]current bounding box.center)}]

    \draw[-] (0,0) -- (0,17);
    \draw[-] (1,0) -- (1,17);
    \draw[-] (2,0) -- (2,17);
    \draw[-] (3,0) -- (3,17);
    \draw[-] (4,0) -- (4,17);
    \draw[-] (5,0) -- (5,17);
    \draw[-] (6,0) -- (6,17);
    \draw[-] (7,0) -- (7,17);
    \draw[-] (8,0) -- (8,17);
    \draw[-] (9,0) -- (9,17);
    \draw[-] (10,0) -- (10,17);
    \draw[-] (11,0) -- (11,17);
    \draw[-] (12,0) -- (12,17);
    \draw[-] (13,0) -- (13,17);
    \draw[-] (14,0) -- (14,17);
    \draw[-] (15,0) -- (15,17);
    \draw[-] (16,0) -- (16,17);
    \draw[-] (17,0) -- (17,17);
    
    \draw[-] (0,0) -- (17,0);
    \draw[-] (0,1) -- (17,1);
    \draw[-] (0,2) -- (17,2);
    \draw[-] (0,3) -- (17,3);
    \draw[-] (0,4) -- (17,4);
    \draw[-] (0,5) -- (17,5);
    \draw[-] (0,6) -- (17,6);
    \draw[-] (0,7) -- (17,7);
    \draw[-] (0,8) -- (17,8);
    \draw[-] (0,9) -- (17,9);
    \draw[-] (0,10) -- (17,10);
    \draw[-] (0,11) -- (17,11);
    \draw[-] (0,12) -- (17,12);
    \draw[-] (0,13) -- (17,13);
    \draw[-] (0,14) -- (17,14);
    \draw[-] (0,15) -- (17,15);
    \draw[-] (0,16) -- (17,16);
    \draw[-] (0,17) -- (17,17);

    \draw[-] (0,0) -- (17,17);
    
    \draw[-,very thick] (0,0) -- (0,5);
    \draw[-,very thick,red] (0,5) -- (0,6);    \draw[-,very thick,red] (0,6) -- (3,6);
    \draw[-,very thick,red] (3,6) -- (3,7);
    \draw[-,very thick,red] (3,7) -- (5,7);
    \draw[-,very thick] (5,7) -- (5,10);
    \draw[-,very thick,blue] (5,10) -- (5,11);
    \draw[-,very thick,blue] (5,11) -- (7,11);
    \draw[-,very thick,blue] (7,11) -- (7,12);
    \draw[-,very thick,blue] (7,12) -- (11,12);
    \draw[-,very thick] (11,12) -- (11,14);
    \draw[-,very thick,green] (11,14) -- (11,15);
    \draw[-,very thick,green] (11,15) -- (13,15);
    \draw[-,very thick,green] (13,15) -- (13,16);
    \draw[-,very thick,green] (13,16) -- (15,16);
    \draw[-,very thick,green] (15,16) -- (15,17);
    \draw[-,very thick,green] (15,17) -- (17,17);
    \end{tikzpicture}.
\]
\end{example}

\begin{rmk}
In \cite[Section 4.1]{HHLRU05}, the authors also give the $e$-expansion not only for $D_n[X;1,1]$, but also for $D_n[X;q,1]$ as
\[
     D_n[X;q,1] =\sum_{\pi\in \Dyck(n)} q^{\coarea(\pi)} e_{\alpha(\pi)}[X].
\]
The same bijection $\Psi_\lambda$ gives a $h$-expansion for $D_n[X;q,1]$ as a sum of Dyck paths in $\bigcup_{\alpha\sim\lambda+(1^{\ell(\lambda)})}\Dyck(\alpha)$ with an appropriate $q$--statistic modifying $\coarea$.
\end{rmk}

\begin{proof}[Proof of Theorem~\ref{thm: h-expansion, Kreweras}]
For a symmetric function $f$, let $[h_\lambda]\left(f\right)$ denote the coefficient of $h_\lambda$ in the $h$-expansion of $f$. We may rewrite the statement as
\[
    [h_\nu]\left(\I_{\mu^{(1)},\dots,\mu^{(k)}}[X;1,1]\right)=
    \begin{cases}
    (-1)^{k-1-\ell(\lambda)}\Krew(\lambda +(1^{\ell(\mu)})) &\text{ if } \nu = \lambda+(1^{n+1-k}) \text{ for some } \lambda\vdash k-1\\
    0 &\text{ otherwise}.
    \end{cases}
\]
The proof is threefold according to each condition $\ell(\nu) > n - k + 1$, $\ell(\nu) < n - k + 1$, and $\ell(\nu) = n - k + 1.$

First, suppose $\ell(\nu)>n+1-k$. We proceed by induction on $a=n-\ell(\nu)$. It is useful to recall that for a partition $\nu$ and a positive integer $N$, we have
\begin{equation}\label{eq: e perp homogeneous}
    e^\perp_N h_\nu = \begin{cases}
        0 &\text{ if } \ell(\nu) < N \\
        h_{\nu-(1^N)} &\text{ if } \ell(\nu) = N
    \end{cases}.
\end{equation}
For the base case $a=0$, by the vanishing identity (Theorem \ref{thm: main theorem} (a)) for $m=0$, we have 
\[
e_{n}^\perp \I_{\mu^{(1)},\dots,\mu^{(k)}}[X;q,t] = [h_{(1^n)}]\left(\I_{\mu^{(1)},\dots,\mu^{(k)}}[X;q,t]\right)=0.
\]
Now, suppose we have $[h_\nu]\left(\I_{\mu^{(1)},\dots,\mu^{(k)}}[X;q,t]\right)=0$ for any partition $\nu$ of length $\ell(\lambda)>n-a$. It is equivalent to saying that we can write
\[
\I_{\mu^{(1)},\dots,\mu^{(k)}}[X;1,1] = \sum_{\substack{\nu\vdash n\\ \ell(\nu)\le n-a}} c_\nu h_\nu[X]
\]
for some coefficients $c_\nu$'s. Then the vanishing identity for $m=a<k-1$ implies that
\begin{align*}
    e_{n-a}^\perp \I_{\mu^{(1)},\dots,\mu^{(k)}}[X;1,1] = \sum_{\substack{\nu\vdash n\\ \ell(\nu)= n-a}} c_\nu h_{\nu-(1^{n-a})}[X] = 0,
\end{align*}
thus 
\[
[h_\nu]\left(\I_{\mu^{(1)},\dots,\mu^{(k)}}[X;q,t]\right)=0 \text{ for }\ell(\nu)=n-a.
\]
This proves the claim.

For the proof of the second case, suppose that $\ell(\nu)<n+1-k$. Recall that for a partition $\mu\vdash n$, we have,
\[
    \widetilde{H}_\mu[X;q,1]=\sum_{\nu\vdash n} (q-1)^{n-\ell(\nu)}A_{\mu,\nu}(q)h_\nu[X],
\]
for some $A_{\mu,\nu}(q)\in\mathbb{Z}[q]$ \cite[Proposition 1.1]{GHQR19}. By the definition of $\I_{\mu^{(1)},\dots,\mu^{(k)}}[X;q,t]$, we have,
\[
    [h_\nu]\left(\I_{\mu^{(1)},\dots,\mu^{(k)}}[X;q,1]\right)=
    \sum_{i=1}^{k} \left(\prod_{j\neq i}\dfrac{T^{(j)}}{T^{(j)}-T^{(i)}}\right)(q-1)^{n-\ell(\nu)}A_{\mu^{(i)},\nu}(q),
\]
where we abuse our notation to write $T^{(i)}=T_{\mu^{(i)}}\vert_{t=1}$. Note that $(\prod_{j\neq i}\dfrac{T^{(j)}}{T^{(j)}-T^{(i)}})$ has a pole of degree $k-1$ at $q=1$. Since $n-\ell(\nu)>k-1$, each term 
\[
    \left(\prod_{j\neq i} \dfrac{T^{(j)}}{T^{(j)}-T^{(i)}}\right)(q-1)^{n-\ell(\nu)}
\]
vanishes at $q=1$.

So far, we have proved that $h$-coefficients vanish for $\ell(\nu)\neq n+1-k$:
\[
    \I_{\mu^{(1)},\dots,\mu^{(k)}}[X;1,1] = \sum_{\substack{\nu\vdash n\\ \ell(\nu)= n+1-k}} c_\nu h_\nu[X]
\]
for some coefficients $c_\nu$'s.
Finally, by Theorem~\ref{thm: main theorem} (c) and \eqref{eq: e perp homogeneous} 
\begin{align*}
    e^\perp_{n+1-k}\I_{\mu^{(1)},\dots,\mu^{(k)}}[X;1,1] = \sum_{\substack{\nu\vdash n\\ \ell(\nu)\le n-k+m}} c_\nu h_{\nu-(1^{n+1-k})}[X] = D_{k-1}[X;1,1].
\end{align*}
Then by Lemma~\ref{lemma: h-expansion for Nabla e_n}, we have 
\[
    c_\nu = (-1)^{k-1-\ell(\lambda)}\Krew(\lambda+(1^{n+1-k})),
\]
for the partition $\lambda$ such that $\nu=\lambda+(1^{n+1-k})$.
\end{proof}

In \cite{BG99}, Bergeron and Garsia demonstrated that the `dimension' of the Macdonald intersection polynomial is $n!/k$. More precisely, they proved Corollary~\ref{cor: n!/k}. In Appendix~\ref{sec: n!/k appendix}, we present another proof of Corollary~\ref{cor: n!/k} using Theorem~\ref{thm: h-expansion, Kreweras}.

\begin{corollary}\label{cor: n!/k} With the assumptions in Theorem~\ref{thm: h-expansion, Kreweras}, we have
\begin{equation*}
    \langle\I_{\mu^{(1)},\dots,\mu^{(k)}}[X;1,1],e_{(1^n)}\rangle=n!/k.
\end{equation*}    
\end{corollary}

\section{Future directions}\label{Sec: Future questions}
\subsection{Geometric proof of Theorem~\ref{thm: main theorem}}
The Macdonald polynomials and the diagonal coinvariant algebra have significant connections to the geometry of Hilbert schemes \cite{Hai01, Hai02} and affine Springer fibers \cite{GKM04, Hik14, CO18, Mel20, CM21}. While the relationship between the combinatorial and geometric aspects of Macdonald polynomials has been extensively studied, understanding our main results in this paper from a geometric perspective remains elusive. We hope that our paper inspires new geometric findings that contribute to the theory of Macdonald polynomials.

\subsection{Combinatorial formula for Macdonald intersection polynomials}
Theorem~\ref{thm: main theorem} (c) shows that the shuffle formula lives inside the Macdonald intersection polyomial. This implies that finding the $F$-expansion of Macdonald intersection polynomials is challenging. In \cite{KLO22}, the authors provide a (positive) combinatorial formula for the monomial and the fundamental quasisymmetric expansion of $\I_{\mu^{(1)},\mu^{(2)}}[X;q,t]$ introducing Butler permutations.

The science fiction conjecture predicts that the Macdonald intersection polynomials $\I_{\mu^{(1)},\dots,\mu^{(k)}}[X;q,t]$ is Schur positive. Nevertheless, we do not know such a Schur positivity result even for $k=2$, which was conjectured by Butler~\cite{But94}. Partial progress in this direction for $k=2$ was made in \cite{KLO22}.

 While we provide an explicit $h$-expansion for $\I_{\mu^{(1)},\dots,\mu^{(k)}}[X;1,1]$ (Theorem~\ref{thm: h-expansion, Kreweras}), we lack a (positive) combinatorial formula for the Schur expansion, even for a monomial expansion. The discovery of such formulas would be intriguing, potentially leading to combinatorial formulas for other specializations like $\I_{\mu^{(1)},\dots,\mu^{(k)}}[X;q,1]$ or even for $\I_{\mu^{(1)},\dots,\mu^{(k)}}[X;q,t]$ in general.

\subsection{Macdonald intersection polynomials for generalized Macdonald polynomials} 
Motivated by geometry and representation theory associated with Macdonald polynomials, several variants of Macdonald polynomials have been introduced. One notable example is the wreath Macdonald polynomial, initially introduced by Haiman in \cite{Hai02} and more recently investigated by Orr, Shimozono, and Wen \cite{Wen19, OSW22}. Additionally, Blasiak, Haiman, Morse, Pun, and Seelinger have defined a generalization of Macdonald polynomials \cite{BHMPS23} based on the theory of Catalanimals \cite{BHMPS21LLT}.

By examining these diverse variants of Macdonald polynomials within the framework of Macdonald intersection polynomials, one might try to expand the realm of the combinatorial theory of Macdonald polynomials, particularly related to the shuffle theorem. Moreover, this investigation has the potential to offer fresh insights into the science fiction conjecture itself.

\appendix
\section{Technical lemmas and the proof of Corollary~\ref{cor: n!/k}}\label{Sec: Appendix}
In this appendix, we present a collection of technical lemmas. Section~\ref{subsec: Inclusion-Exclusion principle for lightningbolts} provides Lemma~\ref{Lem: recurrence tilde LV and LV}, which explores the relationship between lightning bolt formulas $\LB(M)$ and $\widetilde{\LB}(M)$, which is analogous to Lemma~\ref{lem: bar Rk to Rk reduction}.  Section~\ref{subsec: Folklore lemmas} includes several technical lemmas. Notably, Lemma~\ref{Lemma: degree less than k-1 vanishes} is utilized in the proof of Lemma~\ref{Lemma: sum prod 1/(x_i-x_j)} and Lemma~\ref{lem: sf degree less k-1 vanishes}, and Corollary~\ref{Cor: symmetric polynomial version of Lemma A3} is employed to derive \eqref{eq: before involution}. In Section~\ref{sec: n!/k appendix}, we provide a detailed proof of Corollary~\ref{cor: n!/k}.

\subsection{Inclusion-Exclusion principle for lightning bolt formulas}\label{subsec: Inclusion-Exclusion principle for lightningbolts}

The following lemma connects the quantities $\LB(M)$ and $\widetilde{\LB}(M-E)$ for Fibonacci matrices $E$.

\begin{lem}\label{Lem: recurrence tilde LV and LV} Let $M\in\Mat^{\ge0}_{r\times\ell}$. Then we have
\begin{equation*}
    \sum_{E\in \mathrm{Fib}_{r,\ell}}
    (-1)^{|E|}q^{\LB(M;E)-\sum_{(i,j):E_{i,j}=1}M_{i,j}} 
    \widetilde{\LB}(M-E)= \LB(M).
\end{equation*}
\end{lem}
\begin{proof}
For $j>1$, let $X_{i,j}=\binom{\LB(M;(i,j))-1}{M_{i,j}}_q$ and $Y_{i,j}=q^{\LB(M;(i,j))-M_{i,j}}\binom{\LB(M;(i,j))-1}{M_{i,j}-1}_q$.
Given $E\in \mathrm{Fib}_{r,\ell}$ and $j>1$, we define a weight as follows:
\begin{equation*}
\wt(E;(i,j)) =
\begin{cases}
X_{i,j} & \text{if there exists } (a,b) \in \LB(i,j)\setminus\{(i,j)\} \text{ such that } E_{a,b}=1,\\
Y_{i,j} & \text{if } E_{i,j}=1,\\
X_{i,j}+Y_{i,j} & \text{otherwise}.
\end{cases}
\end{equation*}
Note that $\wt(E;(i,j))$ is well-defined since the first and second conditions cannot occur simultaneously, given that $E$ is a Fibonacci matrix. By applying the $q$-Pascal's identity, we have
\begin{equation*}
\binom{\LB(M;(i,j))}{M_{i,j}}_q=X_{i,j}+Y_{i,j}.
\end{equation*}
Therefore, by letting $M'=M-E$, we obtain
\begin{equation*}
\wt(E;(i,j)) =
\begin{cases}
\binom{\LB(M';(i,j))}{M'{i,j}}_q & \text{if } E_{i,j}=0,\\
q^{\LB(M;(i,j))-M{i,j}}\binom{\LB(M';(i,j))}{M'_{i,j}}_q & \text{if } E_{i,j}=1.
\end{cases}
\end{equation*}
Then we conclude that
\begin{equation*}q^{\LB(M;E)-\sum_{(i,j):E_{i,j}=1}M_{i,j}}
\widetilde{\LB}(M-E)=\binom{M_{1,1}+\dots+M_{r,1}}{M_{1,1},\dots,M_{r,1}}_q \prod_{\substack{1\leq i\leq r \\ 1<j\leq \ell}}\wt(E;(i,j))
\end{equation*}

Furthermore, from the definition of $\LB(M)$, we have
\begin{equation*}\LB(M)=\binom{M_{1,1}+\dots+M_{r,1}}{M_{1,1},\dots,M_{r,1}}_q \prod_{\substack{1\leq i\leq r \\ 1<j\leq \ell}}X_{i,j}.
\end{equation*}

We claim that the equation
\begin{equation}\label{eq: appendix eq1}
\sum_{E\in \mathrm{Fib}{r,\ell}}(-1)^{E}\prod_{\substack{1\leq i\leq r \\ 1<j\leq \ell}}\wt(E;(i,j))=\prod_{\substack{1\leq i\leq r \\ 1<j\leq \ell}}X_{i,j}
\end{equation}
holds true as polynomials in the variables $X_{i,j}$ and $Y_{i,j}$. From now on, we will consider $X_{i,j}$ and $Y_{i,j}$ as indeterminates.

For a $0$-$1$ matrix $A$ of size $r\times\ell$ such that the first column is a zero column, we define a monomial $Z_A=\prod_{\substack{1\leq i\leq r \\ 1<j\leq \ell}}Z_{i,j}$, where $Z_{i,j}=X_{i,j}$ if $A_{i,j}=0$ and $Z_{i,j}=Y_{i,j}$ if $A_{i,j}=1$. Then, for $E\in \mathrm{Fib}_{r,\ell}$, the term $\prod_{\substack{1\leq i\leq r \\ 1<j\leq \ell}}\wt(E;(i,j))$ is the sum of monomials $Z_A$. Moreover, $Z_A$ appears in $\prod_{\substack{1\leq i\leq r \\ 1<j\leq \ell}}\wt(E;(i,j))$ if and only if the following condition is satisfied: For each $(i,j)$ such that $E_{i,j}=1$, we have $A_{i,j}=1$ and $A_{a,b}=0$ for $(a,b)\in \LLB(i,j)\setminus\{(i,j+1)\}$.

Now, we define a Fibonacci matrix $E(A)$ associated with $A$ by letting its entries be defined as follows:
\[
    E(A)_{i,j} =
    \begin{cases}
        1 &\text{ if } A_{i,j}=1 \text{ and } A_{a,b}=0 \text{ for } (a,b) \in \LLB(i,j+1)\setminus\{(i,j)\}\\
        0 &\text{ otherwise}.
    \end{cases}
\]

The monomial $Z_A$ is present in the term $\prod_{\substack{1\leq i\leq r \\ 1<j\leq \ell}}\wt(E;(i,j))$ if and only if $E\leq E(A)$. Here, the $\leq$ denotes the entry-wise comparison. The coefficient of $Z_A$ on the left-hand side of \eqref{eq: appendix eq1} is given by $\sum_{E\leq E(A)}(-1)^{|E|}$, which is zero unless $E(A)$ is a zero matrix. On the other hand, if $A$ has a nonzero entry, the matrix $E(A)$ is not a zero matrix since the right-uppermost position $(i,j)$ where $A_{i,j}=1$ implies that $E(A)_{i,j}=1$. Therefore, only a single monomial $Z_O$ survives on the left-hand side of \eqref{eq: appendix eq1}, where $O$ is the zero matrix,  confirming the validity of the claim.
\end{proof}

\subsection{Folklore lemmas}\label{subsec: Folklore lemmas}

\begin{lem}\label{Lemma: degree less than k-1 vanishes}
    For $1\le i \le k$, let $g_i(z_1,\dots,z_k)$ be the polynomials in $z_1,\dots,z_k$. Suppose that the polynomials $g_i$ are of degree $d$ less than $k-1$ and satisfy the following condition: $g_i=g_j$ if $z_i=z_j$. Then we have
    \[
        \sum_{i=1}^k \dfrac{g_i}{\prod_{j\neq i}(z_j-z_i)} = 0.
    \]
\end{lem}
\begin{proof}
    Consider the polynomial
    \[
        P(z_1,\dots,z_k)=\sum_{i=1}^k (-1)^{i-1} {g_i}{\prod_{\substack{1\le a < b \le k \\ a,b\neq i}}(z_b-z_a)}.
    \]
    Choose two integers $\ell<m$. Then by substituting $z_\ell=z_m$, the product $\prod_{\substack{1\le a < b \le k \\ a,b\neq i}}(z_b-z_a)$ vanishes, unless $i=\ell$ or $i=m$. Moreover, at $z_\ell=z_m$, we have
    \[
        \prod_{\substack{1\le a < b \le k \\ a,b\neq \ell}}(z_b-z_a) = (-1)^{m-\ell-1}\prod_{\substack{1\le a < b \le k \\ a,b\neq m}}(z_b-z_a)
    \]
    Therefore, the value of the $P(z_1,\dots,z_k)$ at $z_\ell=z_m$ is
    \begin{align*}
        P(z_1,\dots,z_k)|_{z_\ell=z_m} &= (-1)^{\ell-1}(g_\ell|_{z_\ell=z_m}){\prod_{\substack{1\le a < b \le k \\ a,b\neq \ell}}(z_b-z_a)} + (-1)^{m-1}(g_m|_{z_\ell=z_m}){\prod_{\substack{1\le a < b \le k \\ a,b\neq m}}(z_b-z_a)}\\
        &= (g_m|_{z_\ell=z_m}) {\prod_{\substack{1\le a < b \le k \\ a,b\neq \ell}}(z_b-z_a)} ( (-1)^{\ell-1} + (-1)^{m-1 + m -\ell - 1})= 0.
    \end{align*}
    This implies that the polynomial $P(z_1,\dots,z_k)$ is divided by $(z_m-z_\ell)$ for all choices of $1\le m < \ell \le k$. Therefore, we have
    \[
        P(z_1,\dots,z_k) = Q(z_1,\dots,z_k)\prod_{1\le i < j \le k} (z_j-z_i) 
    \]
    for some polynomial $Q$. Meanwhile, the degree of $P$ is $d + \dfrac{k(k-1)}{2} - k+1< \dfrac{k(k-1)}{2}$. This implies that the polynomial $Q$ is identically zero. By dividing the both sides of the equation $P(z_1,\dots,z_k)=0$ by $\prod_{1\le i < j \le k} (z_j-z_i)$, we prove the claim.
\end{proof}

\begin{lem}\label{Lemma: monomial symmetric sum} For a partition $\lambda$ of size less or equal to $k-1$, let $m_\lambda[z_1,z_2,\dots,z_{k-1}]$ be the corresponding monomial symmetric polynomial. Then we have
\[
    \sum_{i=1}^k \dfrac{z_i^{k-1-|\lambda|}\eta_{k,i}(m_\lambda[z_1,z_2,\dots,z_{k-1}]) }{\prod_{j \neq i} (z_j-z_i)}=(-1)^{k-1-\ell(\lambda)}\binom{\ell(\lambda)}{m_1(\lambda),m_2(\lambda),\dots},
\]
where $m_i(\lambda)$ is the multiplicity of $i$ in $\lambda$ and $\eta_{k,i}$ is the operator defined in Definition \ref{def: operator eta}.
\end{lem}
\begin{proof}
    We proceed by induction on a size of $\lambda$. For the base case $|\lambda|=0$, we need to show the following:
    \[
        \sum_{i=1}^k \dfrac{z_i^{k-1}}{\prod_{j \neq i} (z_j-z_i)} =(-1)^{k-1}.
    \]
    Recall that the determinant of the Vandermonde matrix 
    \[
        V_n := \begin{vmatrix}
    1 & z_1 & z_1^2 & \cdots & z_1^{k-1} \\
    1 & z_2 & z_2^2 & \cdots & z_2^{k-1} \\
    \vdots & \vdots & \vdots & \ddots & \vdots \\
    1 & z_k & z_k^2 & \cdots & z_k^{k-1}
    \end{vmatrix}
    \]   
    is well-known to be
    \begin{equation*}
    \det(V_n) = \prod_{1\leq i < j \leq n} (z_j-z_i).
    \end{equation*}
    On the other hand, the cofactor expansion of $ \det(V_n)$ along the last column gives 
    \[
        \det(V_n)=\sum_{i=1}^k (-1)^{k-i} {z_i^{k-1}}{\prod_{\substack{1\le a < b \le k \\ a,b\neq i}}(z_b-z_a)},
    \] 
    thus we have
    \[ \prod_{1\leq i < j \leq n} (z_j-z_i) = \sum_{i=1}^k (-1)^{k-i} {z_i^{k-1}}{\prod_{\substack{1\le a < b \le k \\ a,b\neq i}}(z_b-z_a)}\]
    Dividing both sides by $\prod_{1\leq i < j \leq n} (z_j-z_i)$ proves the claim. Now consider a partition $\lambda$ such that $0<|\lambda|\leq k-1$ and assume we proved the claim for partitions of size less than $|\lambda|$.  Suppose there are $r$ distinct parts $a_1,\dots,a_r$ of multiplicity $b_1,\dots,b_r$ in $\lambda$. We have
    \[
        \eta_{k,i}(m_\lambda[z_1,\dots,z_{k-1}]) = m_{\lambda}[z_1,\dots,z_k] - \sum_{j=1}^{r}  z_i^{a_j } \eta_{k,i}(m_{\lambda^{(j)}}[z_1,\dots,z_{k-1}]),
    \]
    where $\lambda^{(j)}$ is the partition obtained by removing a part equals to $a_j$ in the partition $\lambda$. Therefore we obtain 
      \[
        \sum_{i=1}^k \dfrac{z_i^{k-1-|\lambda|}\eta_{k,i}(m_\lambda[z_1,\dots,z_{k-1}]) }{\prod_{j \neq i} (z_j-z_i)}=m_{\lambda}[z_1,\dots,z_k] \left(\sum_{i=1}^k \dfrac{z_i^{k-1-|\lambda|} }{\prod_{j \neq i} (z_j-z_i)}\right) - \sum_{j=1}^{r}(-1)^{k-1-|\lambda^{(j)}|}\binom{\ell(\lambda^{(j)})}{b_1,\dots,b_{j-1},b_{j}-1,b_{j+1}\dots,b_{r}}
    \]
    where we used the induction hypothesis for partitions $\lambda^{(j)}$'s. The first term on the right-hand side vanishes by Lemma~\ref{Lemma: degree less than k-1 vanishes} and since we have $\ell(\lambda^{(j)})=\ell(\lambda)-1$, the right-hand side becomes
    \begin{equation*}
        (-1)^{k-1-|\lambda|}\sum_{i=1}^{r}\binom{\ell(\lambda)-1}{b_1,\dots,b_{j-1},b_{j}-1,b_{j+1}\dots,b_{r}}
    \end{equation*}
    which equals $(-1)^{k-1-|\lambda|}\binom{\ell(\lambda)}{b_1,\dots,b_{r}}$ by Pascal's identity for multinomial coefficients. 
\end{proof}

\begin{corollary}\label{Cor: symmetric polynomial version of Lemma A3}Let $f[z_1,z_2,\dots,z_{k-1}]$ be a symmetric polynomial in $k-1$ variables of homogeneous degree $n\leq k-1$. Then we have
\[
    \sum_{i=1}^k \dfrac{z_i^{k-1-n}\eta_{k,i}(f[z_1,z_2,\dots,z_{k-1}])}{\prod_{j \neq i} (z_j-z_i)}=(-1)^{k-1} \sum_{\alpha\models n} (-1)^{\ell(\alpha)}[z^\alpha](f)
\]
where $[z^\alpha](f)$ represents the coefficient of the monomial $z^\alpha:=\prod_{i=1}^{\ell(\alpha)}z_i^{\alpha_i}$ in the polynomial $f$. 
\end{corollary}

\subsection{Proof of Corollary~\ref{cor: n!/k}}\label{sec: n!/k appendix}  For a vector $\alpha=(\alpha_1,\dots,\alpha_{\ell})$ consisting of nonnegative integers that sum to $n$, we abuse our notation and refer to the multinomial $\binom{n}{\alpha_1,\dots,\alpha_{\ell}}$ as $\binom{n}{\alpha}$. It is well-known that $\langle h_{\lambda},e_{(1^n)}\rangle=\binom{n}{\lambda}$ for a partition $\lambda\vdash n$. Therefore by Theorem \ref{thm: h-expansion, Kreweras}, Corollary~\ref{cor: n!/k} is equivalent to
\begin{equation}\label{eq: nonamee}
    \sum_{\lambda \vdash k-1}(-1)^{k-1-\ell(\lambda)}\Krew((\lambda+(1^{\ell(\lambda)}))\binom{n}{\lambda+(1^{n+1-k})}=\frac{n!}{k}.
\end{equation}
For each $\lambda\vdash k-1$, the number of compositions $\alpha \sim \lambda$ which are rearrangements of $\lambda$ equals $\binom{\ell(\lambda)}{m_1(\lambda),m_2(\lambda),\dots}$ where $m_{j}(\lambda)$ denotes the multiplicity of $j$ among the parts of $\lambda$. By equally distributing the value, we may write 
\begin{equation*}
   \Krew((\lambda+(1^{\ell(\lambda)}))\binom{n}{\lambda+(1^{n+1-k})}=\frac{n!}{k!}\sum_{\alpha\sim\lambda}\frac{(k-1+\ell(\alpha))!}{(\ell(\alpha))!\prod_{i\geq 1}(\alpha_i+1)!}
\end{equation*}
and \eqref{eq: nonamee} becomes
\begin{equation*}
    \sum_{m=0}^{k-2}(-1)^m\sum_{\substack{\alpha\models k-1\\\ell(\alpha)=k-1-m}}\frac{(2k-2-m)!}{(k-1-m)!\prod_{i\geq 1}(\alpha_i+1)!}=(k-1)!.
\end{equation*}
Now let 
\begin{equation*}
    \tilde{T}(n,k):=\frac{(2n-k)!}{(n-k)!}\sum_{\substack{\alpha\models n\\\ell(\alpha)=n-k}}\frac{1}{\prod_{i\geq1}(\alpha_i+1)!}.
\end{equation*}
Then it is enough to show the following:
\begin{equation}\label{eq: Ward alternating}
    \sum_{k=0}^{n-1}(-1)^{k}\tilde{T}(n,k)=n!.
\end{equation}

\begin{lem}
    We have
\begin{equation*}
    \tilde{T}(n,k)=T(n,k).
\end{equation*}
where $T(n,k)$ is the Ward number given in \cite[\href{https://oeis.org/A181996}{A181996}]{Sloane}.
\end{lem}
\begin{proof}
We first state the formula for $T(n,k)$ given in \cite[\href{https://oeis.org/A181996}{A181996}]{Sloane}. Let $S(n,k)$ be the Stirling number of the second kind given by:
\begin{equation*}
    S(n,k)=\frac{1}{k!}\sum_{j=0}^{k}(-1)^{k-j}\binom{k}{j}j^n.
\end{equation*}
Then the Ward number $T(n,k)$ is given by:
\begin{equation*}
    T(n,k)=\sum_{m=0}^{n-k}(-1)^{n-k-m}\binom{2n-k}{n+m}S(n+m,m).
\end{equation*}

 We let $A(m,\ell)^{\geq r}$ to be the set of $\alpha=(\alpha_1,\dots,\alpha_{\ell})$ consisting of nonnegative integers that sum to $m$ and each entry $\alpha_i\geq r$. 
We also let $A(m,\ell)^{\geq r}_B$ to be the set of $\alpha\in A(m,\ell)^{\geq r}$ such that $\alpha_i=1$ for $i\in B$. Then from the definition of $\tilde{T}(n,k)$ we have 
\begin{equation*}
    \tilde{T}(n,k)=\frac{1}{(n-k)!}\sum_{\alpha\in A(2n-k,n-k)^{\geq2}}\binom{2n-k}{\alpha}
\end{equation*}
and applying the inclusion-exclusion principle gives 
\begin{equation*}
    \tilde{T}(n,k)=\frac{1}{(n-k)!}\sum_{B\subseteq[n-k]}(-1)^{|B|}\sum_{\alpha\in A(2n-k,n-k)_{B}^{\geq1}}\binom{2n-k}{\alpha}.
\end{equation*}
Now we let  $A(m,\ell)^{\geq r}_{B,C}$ to be the set of $\alpha\in A(m,\ell)^{\geq r}_B$ such that $\alpha_i=0$ for $i\in C$. Again by the inclusion-exclusion principle we have
\begin{equation*}
    \sum_{\alpha\in A(2n-k,n-k)_{B}^{\geq1}}\binom{2n-k}{\alpha}=\sum_{C\subseteq [n-k]\setminus B}(-1)^{|C|}\sum_{\alpha\in A(2n-k,n-k)_{B,C}^{\geq0}}\binom{2n-k}{\alpha}.
\end{equation*}
For each $\alpha\in A(2n-k,n-k)^{\geq 0}_{B,C}$, we let $\beta$ to be the integer vector obtained from $\alpha$ by erasing entries $\alpha_i$ such that $i \in B\cup C$. Then $\beta\in A(2n-k-|B|,n-k-|B|-|C|)^{\geq0}$ and we have 
\begin{equation*}
    \binom{2n-k}{\alpha}=\frac{(2n-k)!}{(2n-k-|B|)!}\binom{2n-k-|B|}{\beta}.
\end{equation*}
Therefore we have 
\begin{align*}
    \sum_{\alpha\in A(2n-k,n-k)_{B,C}^{\geq0}}\binom{2n-k}{\alpha}&=\frac{(2n-k)!}{(2n-k-|B|)!}\sum_{\beta\in A(2n-k-|B|,n-k-|B|-|C|)^{\geq0}}\binom{2n-k-|B|}{\beta}\\&=\frac{(2n-k)!}{(2n-k-|B|)!}(n-k-|B|-|C|)^{2n-k-|B|}.
\end{align*}
Now we get 
\begin{align*}
    \sum_{\alpha\in A(2n-k,n-k)_{B}^{\geq1}}\binom{2n-k}{\alpha}&=\frac{(2n-k)!}{(2n-k-|B|)!}\sum_{c=0}^{n-k-|B|}(-1)^c\binom{n-k-|B|}{c}(n-k-|B|-c)^{2n-k-|B|}\\
    &=\frac{(2n-k)!(n-k-|B|)!S(2n-k-|B|,n-k-|B|)}{(2n-k-|B|)!}.
\end{align*}

We conclude
\begin{align*}
    \tilde{T}(n,k)&=\dfrac{\sum_{b=0}^{n-k}(-1)^{b}\binom{n-k}{b}(2n-k)!(n-k-b)!S(2n-k-b,n-k-b)}{(2n-k-b)!(n-k)!}\\&=\sum_{b=0}^{n-k}(-1)^{b}\binom{2n-k}{b}S(2n-k-b,n-k-b)\\&=\sum_{m=0}^{n-k}(-1)^{n-k-m}\binom{2n-k}{n-k-m}S(n+m,m)=\sum_{m=0}^{n-k}(-1)^{n-k-m}\binom{2n-k}{n+m}S(n+m,m)
\end{align*}
which coincides with $T(n,k)$.

\end{proof}
Now the following lemma finishes the proof of \eqref{eq: Ward alternating}.
\begin{lem}
    We have
\begin{equation}\label{eq: main}
    \sum_{k=0}^{n-1}(-1)^{k}T(n,k)=n!.
\end{equation}
\end{lem}

\begin{proof}
The Ward number $T(n,k)$ satisfies the recurrence given in \cite[\href{https://oeis.org/A181996}{A181996}]{Sloane}:
\begin{equation*}
    T(n,k)=(2n-1-k)T(n-1,k)+(n-k)T(n-1,k-1).
\end{equation*}
We proceed by induction on $n$ to establish \eqref{eq: main}. The base case can be checked easily. Assume we proved \eqref{eq: main} for $(n-1)$. We have 
\begin{align*}
    \sum_{k=0}^{n-1}(-1)^k T(n,k)&=\sum_{k=0}^{n-1}(-1)^k (2n-1-k)T(n-1,k)+\sum_{k=0}^{n-1}(-1)^k(n-k)T(n-1,k-1)\\
    &=\sum_{k=0}^{n-2}(-1)^k (2n-1-k)T(n-1,k)+\sum_{k=0}^{n-2}(-1)^{k+1}(n-k-1)T(n-1,k)\\&=\sum_{k=0}^{n-2}(-1)^k \left((2n-1-k)-(n-k-1)\right)T(n-1,k)\\&=n\sum_{k=0}^{n-2}(-1)^k T(n-1,k)=n\times (n-1)!=n!.
\end{align*}

\end{proof}
\section*{acknowledgement}
The authors are grateful to François Bergeron, James Haglund, Byung-Hak Hwang, Woo-Seok Jung, and Brendon Rhoades for helpful conversations. They also thank the anonymous referee for his or her insightful comments, which greatly enhanced the readability of the paper. D. Kim was supported by the National Research Foundation of Korea (NRF) grant funded by the Korean government (MEST) (No. 2019R1A6A1A10073437) and the individual NRF grant (No. 2022R1I1A1A01070620). S. J. Lee was supported by the National Research Foundation of Korea (NRF) grant funded by the Korean government (MSIT) (No.0450-20240021). J. Oh was supported by KIAS Individual Grant (CG083401, HP083401) at Korea Institute for Advanced Study.
\printbibliography

\end{document}